\pdfoutput=1
\RequirePackage{fix-cm}
\documentclass[smallextended]{svjour3}
\smartqed  % flush right qed marks, e.g. at end of proof

\usepackage{graphicx,float,subfig,epstopdf}
\usepackage{setspace}
\usepackage{booktabs}
\usepackage{units}
\usepackage{mathtools}
\usepackage{amsmath}
\usepackage{amsthm}
\usepackage{amssymb}
\usepackage{multirow}
\usepackage[numbers]{natbib}
\usepackage[unicode=true,pdfusetitle,
bookmarks=true,bookmarksnumbered=true,bookmarksopen=true,bookmarksopenlevel=2,
breaklinks=true,pdfborder={0 0 0},pdfborderstyle={},backref=page,colorlinks=true]{hyperref}
\usepackage{lineno}
\newtheorem{thm}{Theorem}
\newtheorem{prop}[thm]{Proposition}
\newtheorem{rem}[thm]{Remark}
\newtheorem{lem}[thm]{Lemma}

\numberwithin{equation}{section}

\allowdisplaybreaks

\begin{document}
\title{A full divergence-free of high order virtual finite element method to approximation of stationary inductionless magnetohydrodynamic equations on polygonal meshes}
\author{Xianghai Zhou  \and Haiyan Su }

\titlerunning{virtual element approximate for stationary IMHD}
\authorrunning{X. Zhou, H. Su }

\institute{Xianghai Zhou\at
	 College of Mathematics and System Sciences, Xinjiang University, Urumqi 830046, China. \\
	\email{zxhmath166@163.com}
	\and
	 Haiyan Su \at
	 Corresponding author. College of Mathematics and System Sciences, Xinjiang University, Urumqi 830046, China. \\
	\email{shymath@126.com}
}


\maketitle
\begin{abstract}
In this present paper we consider a full divergence-free of high order virtual finite element algorithm to approximate the stationary inductionless magnetohydrodynamic model on polygonal meshes. More precisely, we choice  appropriate virtual spaces and necessary degrees of freedom for velocity and current density to guarantee that their final discrete formats are both pointwise divergence-free. Moreover, we hope to achieve higher approximation accuracy at higher ``polynomial" orders $k_1\geq2,k_2\geq1$, while the full divergence-free property has always been satisfied. And then we processed rigorous error analysis to show that the proposed method is stable and convergent. Several numerical tests are presented, confirming the theoretical predictions.

\end{abstract}

%\subclass{65M60; 65M15; 76W05}

\section{Introduction\label{sec:Intro}}
The inductionless magnetohydrodynamic (MHD) model is used to describe a type of problem in which the magnetic field generated by the current in liquid metal can be ignored compared to the external magnetic field $\boldsymbol{B}$. We often use this set of equations to simulate several industrial processes, such as MHD generators, test blanket modules (TBMs) in nuclear fusion reactors, etc, refer to \cite{planas2011approximation,ni2007currentI,ni2007currentII}.
\\
Let $\mathrm{\Omega}\subset\mathbb{R}^2$ is bounded domain with Lipschitz-continuous boundary $\mathrm{\Gamma_{\Omega}}$, we consider the stationary inductionless MHD model with homogeneous boundary conditions as follows: %\cite{badia2014block,zhang2021coupled}:
\begin{subequations}
	\begin{align}
-\nu\Delta\boldsymbol{u}+\boldsymbol{u}\cdot\nabla\boldsymbol{u}+\nabla
p-S_{c}\boldsymbol{J}\times\boldsymbol{B}=\boldsymbol{f}&\quad\text{in}\; \mathrm{\Omega},\label{model:IMHD-u}\\
\nabla\cdot\boldsymbol{u}=0&\quad\text{in}\;\mathrm{\Omega},\label{model:IMHD-p}\\
\boldsymbol{J}+\nabla\phi-\boldsymbol{u}\times\boldsymbol{B}=\boldsymbol{g}&\quad\text{in} \;\mathrm{\Omega},\label{model:IMHD-J}\\
\nabla\cdot\boldsymbol{J}=0&\quad\text{in}\;\mathrm{\Omega},\label{model:IMHD-phi}\\
\boldsymbol{u}=\boldsymbol{0},\quad
\boldsymbol{J}\cdot \boldsymbol{n}=0&\quad\text{on}\;\mathrm{\Gamma_\Omega}.\label{model:IMHD-bound}
	\end{align}
\label{model:IMHD}
\end{subequations}
The unknowns of the model (\ref{model:IMHD}) are the velocity $\boldsymbol{u}$, the pressure $p$, the current density $\boldsymbol{J}$ and the electric potential $\phi$. Functions $\boldsymbol{f}$ and $\boldsymbol{g}$ denote external force terms, the $\boldsymbol{B}$ is assumed to be a given applied magnetic field. The parameter $\nu$ denotes viscosity coefficient, and the $S_{c}$ denotes interaction parameter. For a domain $\mathrm{\Omega}$, $\boldsymbol{n}_{\mathrm{\Omega}}$ or simply $\boldsymbol{n}$ will be the outward normal unit vector.

In the past few decades, numerical methods for this model has been extensively studied and analyzed. Especially the research on satisfying the divergence-free condition, namely the law of charge-conservation or mass-conservation, has attracted the attention of researchers. As mentioned in \cite{greif2010mixed,ni2007currentI,ni2007currentII}, it plays a crucial role in MHD systems and their numerical simulations.
Therefore, we have only collected some recent research results on satisfying conservation laws here.  Ni et al. proposed a class of consistent and charge-conservative finite volume schemes for the inductionless MHD equations on both structured and unstructured meshes in \cite{ni2007currentI,ni2007currentII,ni2012consistent}.
Li et al. proposed a charge-conservative finite element scheme, namely the $(\boldsymbol{J},\phi)$ via a mixed formulation in $\boldsymbol{H}_{0}(\textup{div},\mathrm{\Omega})\times L_{0}^2(\mathrm{\Omega})$ to obtain an exactly divergence-free current density , to approximate inductionless MHD problems and developed a robust solver for discrete problems in \cite{2019AI,li2019charge}. Subsequently, Zhang et al. successively considered three types of iteration methods based on charge-conservative finite element method and a full divergence-free scheme for inductionless MHD model in \cite{zhang2021coupled,zhang2022fully}.
Under the influence of this charge conservation scheme, many efforts have been made \cite{long2019analysis,long2023error,zhang2022decoupled,zhou2022two} and this method has been extended to many different models \cite{wang2023decoupled,dong2023electric,long2022convergence}.

In terms of mesh partition alone, traditional finite element methods are limited to standard triangular and quadrilateral (tetrahedral/hexahedral) meshes. Therefore, in the past decade, utilizing numerical methods on general polygonal (polyhedral) meshes to solve complex problems has received increasing attention. Brief representative of techniques using polygonal/polyhedral meshes is listed in  \cite{da2017divergence}. In particular, the virtual finite element method (VEM) has attracted our attention as a generalization of the finite element method in arbitrary element-geometry. The main difference between VEM and FEM is that in the construction and calculation of VEM schemes, we do not need to know the explicit basis of the discrete space, which makes all linear functionals of discrete variational formulas constructed using appropriate polynomial projections. Of course, these polynomial projections can always be calculated by carefully selecting degrees of freedom. At present, the virtual finite element method has developed relatively mature and has a wide range of applications. Here, we will only list some representative works \cite{da2018family,manzini2022conforming,da2017divergence,vacca2018h,dassi2020bricks,caceres2017mixed,beirao2014h,ahmad2013equivalent,
brezzi2014basic,2017Virtual,da2018virtual,beirao2023virtual}.  Reference \cite{beirao2014hitchhiker,herrera2023numerical} helps with specific implementation.

The purpose of this article is to construct a virtual finite element format that is full divergence-free and not limited to low order. Therefore, in the application of numerous VEM methods, the virtual element format of two types of problems has attracted our attention. Firstly, in \cite{da2017divergence,da2018virtual}, a type of VEM format designed using the Stokes and Navier-Stokes problems as examples. This virtual finite element space, carefully designed for velocity, can lead to precise and pointwise divergence-free discrete velocity under appropriate pressure space. Afterwards, some samples from other fluid flow models were listed as \cite{beirao2019stokes,beirao2020stokes,dassi2020bricks,
vacca2018h,da2018family,beirao2023virtual,chen2019divergence,frerichs2022divergence}.
Secondly, in \cite{brezzi2014basic,beirao2014h,beirao2016mixed}, De-Rahm complexes in the virtual element setting have been introduced. Later, taking the Brikman problem in \cite{caceres2017mixed} as an example, a VEM format was designed, which can obtain accurate and pointwise divergence-free velocity in spaces $\boldsymbol{H}_{0}(\textup{div},\mathrm{\Omega})$ and $L^2_{0}(\mathrm{\Omega})$. There are some similar applications, such as \cite{caceres2017mixed,gatica2018mixed,da2022virtualMax} . Overall, we note that the characteristics of the model considered in this article, the first  format is very suitable for VEM approximation of velocity and pressure, while the second format is also very suitable for VEM approximation of current density and potential. In this way, we will have the potential to design a full divergence-free high-order virtual element format that approximates the  inductionless MHD equations.

The rest of this paper: Some notations and weak formulation of the model are given, in section \ref{sec:Pre}. A virtual finite element scheme for solving the problem numerically is discussed, in section \ref{sec:virtual}. The well-posed and convergence of discrete problem are investigated. In particular, the optimal convergence rates of energy norm for velocity approximation and  $L^2$-norm of pressure,current density and electric potential approximations are derived, in section \ref{sec:Analysis}. Some numerical tests are designed to verify the validity of the virtual element scheme and the correctness of the theory, in section \ref{sec:Num}. Finally, some conclusion are obtained, in section \ref{sec:conclu}.
\section{Perliminary\label{sec:Pre}}

\subsection{Some notations and spaces\label{subsec:notes}}
Some necessary definitions and notations of Sobolev spaces with norm or semi-norm are briefly introduced, see \cite{adams2003sobolev} for details. Let $\mathrm{G}\subset\mathbb{R}^{2}$ is an open, bounded and connected subset, the $\mathrm{\Gamma}_{\mathrm{G}}$ denotes boundary of any domain $\mathrm{G}$. We realize the Sobolev space $W^{k,p}(\mathrm{G})$ with norm $\|\cdot\|_{k,p,\mathrm{G}}$ or semi-norm $|\cdot|_{k,p,\mathrm{G}}$, where $k$ is a nonnegative integer and $p\geq1$. For $p=2$, the $W^{k,2}(\mathrm{G})$ and the corresponding norm $\|\cdot\|_{k,2,\mathrm{G}}$ are customarily written as $H^{k}(\mathrm{G})$ and $\|\cdot\|_{k,\mathrm{G}}$, respectively. Moreover, if $k=0$, then we denote as
\begin{equation}
\nonumber
L^p(\mathrm{G}):=W^{0,p}(\mathrm{G}):=\{v:\|v\|_{0,p,\mathrm{G}}<\infty\},\quad L^{p}_{0}(\mathrm{G})=\Big\{v\in L^p(\mathrm{G}):\int_{\mathrm{G}} vd\boldsymbol{x}=0\Big\}.
\end{equation}
Particularly, we equip with inner product and norm
\begin{equation}
\nonumber
(u,v)_\mathrm{G}=\int_{G}uvd\boldsymbol{x},\quad\|v\|_{0,\mathrm{G}}:=\|v\|_{0,2,\mathrm{G}}=(v,v)^{\frac{1}{2}}_{\mathrm{G}}\quad \forall u,v\in L^{2}(\mathrm{G}).
\end{equation}

In this paper, bold letters are generally used to denote vector-valued fields such as $\boldsymbol{H}^{1}(\mathrm{G}):=[H^{1}(\mathrm{G})]^{2}$ and some function spaces are defined as
\begin{align*}
&\boldsymbol{V}(\mathrm{G}):=\boldsymbol{H}^{1}_{0}(\mathrm{G})=\Big\{\boldsymbol{v}\in\boldsymbol{H}^{1}(\mathrm{G}):\boldsymbol{v}=\boldsymbol{0} \;on\;\mathrm{\Gamma}_{\mathrm{G}}\Big\},\quad Q(\mathrm{G}):=L^{2}_{0}(\mathrm{G}),\\
&\boldsymbol{S}(\mathrm{G}):=\boldsymbol{H}_{0}(\textup{div},\mathrm{G})=\Big\{\boldsymbol{v}\in\boldsymbol{H}(\textup{div},\mathrm{G}):\boldsymbol{v}\cdot\boldsymbol{n}=\boldsymbol{0} \;on\;\mathrm{\Gamma}_{\mathrm{G}}\Big\},\quad \Psi(\mathrm{G}):=L^{2}_{0}(\mathrm{G}).
\end{align*}

Using standard VEM notations, for a geometric object $\mathrm{G}$ and $k\in\mathbb{N}$, we will denote by $\boldsymbol{x}_{\mathrm{G}}$ and $h_{\mathrm{G}}$ it barycenter and diameter, respectively. Let $\mathcal{P}_{k}(\mathrm{G})$, $[\mathcal{P}_{k}(\mathrm{G})]^{2}$ and $[\mathcal{P}_{k}(\mathrm{G})]^{2\times2}$ be the spaces of scalar, vectorial and matrix polynomials defined on $\mathrm{G}$ of degree less than or equal on $k$. We denote by $\mathcal{M}_{k}(\mathrm{G})$ the set of polynomials
\begin{equation}
\nonumber
\mathcal{M}_{k}(\mathrm{G}):=\bigcup_{i=0}^{k}\mathcal{M}_{i}^{*}(\mathrm{G})\subseteq \mathcal{P}_{k}(\mathrm{G}),\quad\mathcal{M}_{i}^{*}(\mathrm{G}):=\Big\{\Big(\frac{\boldsymbol{x}-\boldsymbol{x}_{\mathrm{G}}}{h_{\mathrm{G}}}\Big)^{\boldsymbol{s}},|\boldsymbol{s}| = i\Big\},
\end{equation}
where $\boldsymbol{s}\in \mathbb{N}^2$ is a multi-index with $|\boldsymbol{s}| = s_{1}+s_{2}$ and $\boldsymbol{x}^{\boldsymbol{s}}:=x_{1}^{s_{1}}x_{2}^{s_2}$.
It is clear that the set $\mathcal{M}_{k}(\mathrm{G})$ is a basis for $\mathcal{P}_{k}(\mathrm{G})$. Furthermore, we define the spaces
\begin{itemize}
\setstretch{1.25}
\item[$\bullet$]$\mathcal{B}_{k}(\mathrm{\Gamma}_\mathrm{G}):=\big\{\omega\in C^{0}(\mathrm{\Gamma}_\mathrm{G}):\;\omega|_{e}\in\mathcal{P}_{k}(e)\quad\forall e\subset\mathrm{\Gamma}_\mathrm{G}\big\}$;
\item[$\bullet$]$\boldsymbol{\mathcal{G}}_{k}(\mathrm{G}):=\nabla\big(\mathcal{P}_{k+1}(\mathrm{G})\big)\subseteq[\mathcal{P}_{k}(\mathrm{G})]^{2}$;
\item[$\bullet$]$\boldsymbol{\mathcal{G}}_{k}(\mathrm{G})^{\perp}\subseteq[\mathcal{P}_{k}(\mathrm{G})]^{2}$ the $L^2$-orthogonal complement to $\boldsymbol{\mathcal{G}}_{k}(\mathrm{G})$;
\item[$\bullet$]  $\boldsymbol{\mathcal{G}}_{k_2}(\mathrm{G})^{\perp}/\boldsymbol{\mathcal{G}}_{k_1}(\mathrm{G})^{\perp}$  denotes the polynomials in $\boldsymbol{\mathcal{G}}_{k_2}(\mathrm{G})^{\perp}$ that are $L^2$-orthogonal to all polynomials of $\boldsymbol{\mathcal{G}}_{k_1}(\mathrm{G})^{\perp}$
     \quad for any $k_{1}\leq k_{2}$ $(k_{1},k_{2}\in\mathbb{N})$;
\item[$\bullet$]
     $\hat{\mathcal{P}}_{n\backslash m}(\mathrm{G}):=\text{span}(\mathcal{M}_{i}^{*}(\mathrm{G}):m+1\leq i\leq n)$.
\end{itemize}

For the sake of convenience, we define some norm schemes and introduce the Poincar$\acute{e}$ type inequalities  as
\begin{align*}
&\|\boldsymbol{K}\|_{\textup{div},\mathrm{G}}=(\|\boldsymbol{K}\|_{0,\mathrm{G}}^2+\|\nabla\cdot\boldsymbol{K}\|_{0,\mathrm{G}}^2)^{\frac{1}{2}}\quad \forall\boldsymbol{K}\in\boldsymbol{S}(\mathrm{G}),\\
&\|(\boldsymbol{v},\boldsymbol{K})\|_{1,\mathrm{G}}=(|\boldsymbol{u}|_{1,\mathrm{G}}^2+\|\boldsymbol{K}\|_{\textup{div},\mathrm{G}}^2)^{\frac{1}{2}}\quad \forall (\boldsymbol{v},\boldsymbol{K})\in \boldsymbol{V}(\mathrm{G})\times\boldsymbol{S}(\mathrm{G}),\\
%&\|(q,\psi)\|_{1,\mathrm{G}}=(\|q\|_{0,\mathrm{G}}^2+\|\psi\|_{0,\mathrm{G}}^2)^{\frac{1}{2}}\quad\forall (q,\psi)\in Q(\mathrm{G})\times\Psi(\mathrm{G}),\\
&\|\boldsymbol{f}\|_{-1,\mathrm{G}}=\sup_{\boldsymbol{v}\in\boldsymbol{V}}\frac{\left\langle\boldsymbol{f},\boldsymbol{v}\right\rangle}{|\boldsymbol{v}|_{1,\mathrm{G}}},\quad\|\mathrm{F}\|_{*}=(\|\boldsymbol{f}\|_{-1,\mathrm{G}}+S_c\|\boldsymbol{g}\|_{0,\mathrm{G}})^{\frac{1}{2}},\\
&\|\boldsymbol{v}\|_{0,6,\mathrm{G}}\leq\lambda_{1}|\boldsymbol{v}|_{1,\mathrm{G}},\quad\|\boldsymbol{v}\|_{0,4,\mathrm{G}}^{2}\leq\lambda_{2}|\boldsymbol{v}|_{1,\mathrm{G}}^{2}\quad\forall \boldsymbol{v}\in \boldsymbol{V}(\mathrm{G}),
\end{align*}
where the positive constants $\lambda_1$ and $\lambda_2$ only depend on $\mathrm{\Omega}$. Throughout this paper, the $C$ represents positive constant independent of the discretization parameters, which takes different value in different situations.

\subsection{The variational formulation of the model\label{subsec:variational}}

To begin with, we define some bilinear and trilinear schemes
\[
\begin{array}{r}
a_{v}(\boldsymbol{u},\boldsymbol{v})=\nu\left(\nabla\boldsymbol{u},\nabla\boldsymbol{v}\right)_\mathrm{\Omega},\quad
a_{K}(\boldsymbol{J},\boldsymbol{K})=S_c\left(\boldsymbol{J},\boldsymbol{K}\right)_\mathrm{\Omega},\\
b_{q}(q,\boldsymbol{v})=\left(q,\nabla\cdot\boldsymbol{v}\right)_\mathrm{\Omega},\quad b_{\psi}(\psi,\boldsymbol{J})=S_c\left(\psi,\nabla\cdot\boldsymbol{J}\right)_\mathrm{\Omega},\\ d_{K}(\boldsymbol{K},\boldsymbol{v})=S_c\left(\boldsymbol{K}\times\boldsymbol{B},\boldsymbol{v}\right)_\mathrm{\Omega},\quad
c_{v}(\boldsymbol{w},\boldsymbol{u},\boldsymbol{v})=\frac{1}{2}\left(\boldsymbol{w}\cdot\nabla\boldsymbol{u},\boldsymbol{v}\right)_{\mathrm{\Omega}}-\frac{1}{2}\left(\boldsymbol{w}\cdot\nabla\boldsymbol{v},\boldsymbol{u}\right)_{\mathrm{\Omega}},
\end{array}
\]
for any $\boldsymbol{u},\;\boldsymbol{v},\;\boldsymbol{w}\in\boldsymbol{V}(\mathrm{\Omega})$, $\boldsymbol{J},\;\boldsymbol{K}\in\boldsymbol{S}(\mathrm{\Omega})$, $q\in Q(\mathrm{\Omega})$ and $\phi\in \Psi(\mathrm{\Omega})$. The notation $\boldsymbol{V}(\mathrm{\Omega})$ is simplified to $\boldsymbol{V}$, and so on.
In addition, let us define two kernel spaces
\begin{align*}
\boldsymbol{Z}=\Big\{\boldsymbol{v}\in \boldsymbol{V}: b_{q}(q,\boldsymbol{v})=0\quad \forall q\in Q\Big\},
\quad\boldsymbol{Y}=\Big\{\boldsymbol{K}\in \boldsymbol{S}: b_{\psi}(\psi,\boldsymbol{K})=0\quad \forall \psi\in \Psi\Big\}.
\end{align*}

Next, we use the above notations to obtain a weak formulation of (\ref{model:IMHD}) as: Find $(\boldsymbol{u},p,\boldsymbol{J},\phi)\in\boldsymbol{V}\times Q\times\boldsymbol{S}\times\Psi$ such that
\begin{subequations}
	\begin{align}
a_{v}(\boldsymbol{u},\boldsymbol{v})+c_{v}(\boldsymbol{u},\boldsymbol{u},\boldsymbol{v})
-b_{q}(p,\boldsymbol{v})-d_{K}(\boldsymbol{J},\boldsymbol{v})&=(\boldsymbol{f},\boldsymbol{v})_{\mathrm{\Omega}},\label{variational:continous-u}\\
b_{q}(q,\boldsymbol{u})&=0,\label{variational:continous-p}\\
a_{K}(\boldsymbol{J},\boldsymbol{K})-b_{\psi}(\phi,\boldsymbol{K})+d_{K}(\boldsymbol{K},\boldsymbol{u})&=S_c(\boldsymbol{g},\boldsymbol{K})_{\mathrm{\Omega}},\label{variational:continous-J}\\
b_{\psi}(\psi,\boldsymbol{J})&=0,\label{variational:continous-phi}	
\end{align}
\label{variational:continous}
\end{subequations}
for all $(\boldsymbol{v},q,\boldsymbol{K},\psi)\in\boldsymbol{V}\times Q\times\boldsymbol{S}\times\Psi$. Obviously, we observe that $(\boldsymbol{u},\boldsymbol{J})\in\boldsymbol{Z}\times\boldsymbol{Y}$.

At the end of the subsection, we give some important properties and estimates.
The bilinear form $a_{v}(\cdot,\cdot)$ satisfies continuous and coercive, i.e.,
\begin{align*}
a_{v}(\boldsymbol{u},\boldsymbol{v})\leq|\boldsymbol{u}|_{1,\mathrm{\Omega}}|\boldsymbol{v}|_{1,\mathrm{\Omega}}\quad\forall \boldsymbol{u},\boldsymbol{v}\in\boldsymbol{V},\quad
a_{v}(\boldsymbol{u},\boldsymbol{u})\geq|\boldsymbol{u}|_{1,\mathrm{\Omega}}^2\quad\forall \boldsymbol{u}\in\boldsymbol{V}.
\end{align*}
Besides, the bilinear forms $b_{q}(\cdot,\cdot)$ and $b_{\psi}(\cdot,\cdot)$ satisfy inf-sup conditions:
\begin{equation}
\beta_{1}\|q\|_{0,\mathrm{\Omega}}\leq\sup_{\boldsymbol{0}\neq\boldsymbol{v}\in\boldsymbol{V}}\frac{b_{q}(q,\boldsymbol{v})}{|\boldsymbol{v}|_{1,\mathrm{\Omega}}}\quad\forall q\in Q,\quad
\beta_{2}\|\psi\|_{0,\mathrm{\Omega}}\leq\sup_{\boldsymbol{0}\neq\boldsymbol{K}\in\boldsymbol{S}}\frac{b_{\psi}(\psi,\boldsymbol{K})}{\|\boldsymbol{K}\|_{\textup{div},\mathrm{\Omega}}}\quad\forall \psi\in \Psi,
\end{equation}
where $\beta_{1},\beta_{2}$ are two strictly positive constants.

The trilinear form $c_{v}$ satisfies continuous and skew-symmetric \cite{da2018virtual}:
\begin{align}
c_{v}(\boldsymbol{u},\boldsymbol{v},\boldsymbol{v})=0\quad&\forall\boldsymbol{u},\boldsymbol{v}\in\boldsymbol{V},\label{trilinear:skew-symmetric}\\
c_{v}(\boldsymbol{u},\boldsymbol{v},\boldsymbol{w})\leq \lambda_3|\boldsymbol{u}|_{1,\mathrm{\Omega}}|\boldsymbol{v}|_{1,\mathrm{\Omega}}|\boldsymbol{w}|_{1,\mathrm{\Omega}}
\quad&\forall\boldsymbol{u},\boldsymbol{v},\boldsymbol{w}\in\boldsymbol{V},\label{trilinear:continuous}
\end{align}
where $\lambda_3$ denotes a positive constant only dependent on $\mathrm{\Omega}$.

And then, these properties imply the well-posed and stability of variational formulation \ref{variational:continous} are shown in the lemma below \cite{zhang2021coupled}.
\begin{thm}
\label{lem:continue problem well-posed}
Suppose $\sigma:= \lambda_{3}\min\{\nu,S_c\}^{-2}\|\mathrm{F}\|_{*}<1$, then problem \ref{variational:continous} has a unique solution $(\boldsymbol{u},p,\boldsymbol{J},\phi)\in \boldsymbol{V}\times Q\times\boldsymbol{S}\times\Psi$, and satisfy the stability inequality
\begin{equation*}
\|(\boldsymbol{u},\boldsymbol{J})\|_{1,\mathrm{\Omega}}\leq \min\{\nu, S_c\}^{-1}\|\mathrm{F}\|_{*}.
\end{equation*}
\end{thm}
Particularly, we can formulate the equivalent kernel form of problem (\ref{variational:continous}) as:
Find $(\boldsymbol{u}, \boldsymbol{J})\in\boldsymbol{Z}\times \boldsymbol{Y}$ such that
\begin{subequations}
	\begin{align}
a_{v}(\boldsymbol{u},\boldsymbol{v})+c_{v}(\boldsymbol{u},\boldsymbol{u},\boldsymbol{v})
-d_{K}(\boldsymbol{J},\boldsymbol{v})&=(\boldsymbol{f},\boldsymbol{v})_{\mathrm{\Omega}},\label{variational:kernel-continous-u}\\
a_{K}(\boldsymbol{J},\boldsymbol{K})+d_{K}(\boldsymbol{K},\boldsymbol{u})&=S_c(\boldsymbol{g},\boldsymbol{K})_{\mathrm{\Omega}},\label{variational:kernel-continous-J}
\end{align}
\label{variational:kernel-continous}
\end{subequations}
for all $(\boldsymbol{v},\boldsymbol{K})\in\boldsymbol{Z}\times\boldsymbol{Y}$. Now, we notice that if $(\boldsymbol{u},\boldsymbol{J})\in\boldsymbol{V}\times\boldsymbol{S}$ is the solution to problem $(\ref{variational:continous})$, then it is also the solution to problem $(\ref{variational:kernel-continous})$ (see the definitions of two kernel spaces $\boldsymbol{Z}$ and $\boldsymbol{Y}$).

\section{The VEM approximation of the model\label{sec:virtual}}

\subsection{Mesh and VEM spaces\label{subsec:space}}
At the beginning of this subsection, we briefly describe the grid tools required for the virtual element technique, discrete Spaces, and some basic notations, refer to \cite{manzini2022conforming,beirao2013basic} for more details. Let's partition domain $\mathrm{\Omega}$ into a sequence $\mathcal{T}_{h}$ consisting of general non-overlapping polygonal element $\mathrm{E}$ with
\begin{equation*}
 h:=\sup_{\mathrm{E}\in\mathcal{T}_{h}}h_\mathrm{E}.
\end{equation*}
Let's further suppose that for any element $\mathrm{E}\in\mathcal{T}_{h}$ carry out mesh regularity conditions:
\begin{itemize}
\setstretch{1.25}
\item[$\bullet$] the shortest edge of polygonal element $\mathrm{E}$ is $h_e\geq\alpha_{e} h_\mathrm{E}$;
\item[$\bullet$]$\mathrm{E}$ is star-shaped with respect to a ball of radius $\geq\alpha_{s} h_\mathrm{E}$.
\end{itemize}
where $\alpha_{s}$, $\alpha_{e}$ represent two positive constants. Despite the above conditions not too restrictive in many practical cases, can be further relaxed, as studied in \cite{2017Virtual}.

Next, we define some finite dimensional spaces that approximate the inductionless MHD unknowns.

Firstly, for $k_{1}\geq2$, approximate space of the velocity field $\boldsymbol{u}$ \cite{vacca2018h,dassi2020bricks}:

\begin{align}
\label{Finite space:V_h^k(E)}
\boldsymbol{V}_{h}^{k_1}(\mathrm{E}):=
\Big\{\boldsymbol{v}\in\boldsymbol{U}_{h}^{k_1}(\mathrm{E}):\mathrm{s.t.}\;
(\boldsymbol{v}-\mathrm{\Pi}_{k_1}^{\nabla,\mathrm{E}}\boldsymbol{v},\boldsymbol{g}_{k_1}^{\perp})_{\mathrm{E}}=0\quad \textup{for}\;\textup{all}\;\boldsymbol{g}_{k_1}^{\perp}\in
\boldsymbol{\mathcal{G}}_{k_1}^{\perp}(\mathrm{E})/\boldsymbol{\mathcal{G}}_{k_1-2}^{\perp}(\mathrm{E})\Big\},
\end{align}

\begin{align}
\label{Finite space:V_h^k(Omega)}
\boldsymbol{V}_{h}:=
\Big\{\boldsymbol{v}\in \boldsymbol{H}^{1}_{0}:\mathrm{s.t.}\;
\boldsymbol{v}|_{E}\in \boldsymbol{V}_{h}^{k_1}(\mathrm{E})\quad \textup{for}\;\textup{all}\;\mathrm{E}\in\mathcal{T}_{h}\Big\},
\end{align}
where the $\mathrm{\Pi}_{k_1}^{\nabla,\mathrm{E}}$ denotes elliptic projection which will be defined later and the space $\boldsymbol{U}_{h}^{k_1}(\mathrm{E})$ defined as
\begin{align}
\nonumber
\boldsymbol{U}_{h}^{k_1}(\mathrm{E}):=&\Bigg\{ \boldsymbol{v}\in\boldsymbol{H}^{1}(\mathrm{E}):\mathrm{s.t.}\;\boldsymbol{v}|_{\partial \mathrm{E}}\in[\mathcal{B}_{k_1}(\mathrm{\Gamma}_\mathrm{E})]^2,\\\nonumber
&\begin{cases}
-\Delta\boldsymbol{v}-\nabla s\in\boldsymbol{\mathcal{G}}_{k_1}(\mathrm{E})^{\perp}\\
\nabla\cdot\boldsymbol{v}\in\mathcal{P}_{k_1-1}(\mathrm{E})
\end{cases}
\mathrm{for}\;\mathrm{some}\;s\in L^2(\mathrm{E})
\Bigg\}.
\end{align}

Secondly, approximate space of the pressure field $p$:
\begin{equation}
\label{Finite space:Q_h}
Q_{h}:=\Big\{q_{h}\in Q:q_{h}|_{\mathrm{E}}\in\mathcal{P}_{k_1-1}(\mathrm{E})\quad\textup{for all} \;\mathrm{E}\in\mathcal{T}_{h}\Big\}.
\end{equation}
As observed in \cite{da2017divergence}, we remark a key property
\begin{equation}
\label{Vh-Qh}
\nabla\cdot\boldsymbol{V}_{h}\subseteq Q_{h},
\end{equation}
which determines that we can ultimately obtain a divergence-free discrete solution $\boldsymbol{u}_h$.

Thirdly, for $k_2\geq1$, approximate space of the current density $\boldsymbol{J}$ \cite{brezzi2014basic,beirao2014h,caceres2017mixed}:
\begin{align}
\nonumber\label{Finite space:S_h^k(E)}
\boldsymbol{S}_{h}^{k_2}(\mathrm{E}):=\Big\{\boldsymbol{K}_{h}\in\boldsymbol{H}(\textup{div};\mathrm{E})\cap\boldsymbol{H}(\textup{rot};\mathrm{E})
\;\text{such that} \;\boldsymbol{K}_{h}\cdot\boldsymbol{n}|_{e}\in\mathcal{P}_{k_2}(\mathrm{E})\quad \text{for all } e\in\mathrm{\Gamma}_{\mathrm{E}},\\\nabla\cdot\boldsymbol{K}_{h}|_{\mathrm{E}}\in\mathcal{P}_{k_2}(\mathrm{E}) \;\text{and}\;\nabla\times\boldsymbol{K}_{h}|_{\mathrm{E}}\in\mathcal{P}_{k_2-1}(\mathrm{E})\Big\},
\end{align}
\begin{align}
\label{Finite space:S_h^k(Omega)}
\boldsymbol{S}_{h}:=\Big\{\boldsymbol{K}_{h}\in\boldsymbol{S}: \text{s.t.}\;\boldsymbol{J}\cdot\boldsymbol{n}=0 \;\text{and}\;\boldsymbol{K}_{h}|_{\mathrm{E}}\in\boldsymbol{S}_{h}^{k_2}(\mathrm{E}) \quad \textup{for all} \;\mathrm{E}\in\mathcal{T}_{h}\Big\},
\end{align}
where $\nabla\times\boldsymbol{K}_{h}:=\frac{\partial K_{h,2}}{\partial y}-\frac{\partial K_{h,1}}{\partial x}$ $\big(\boldsymbol{K}_{h} = (K_{h,1},K_{h,2})\big)$.

Finally, approximate space of the electric potential $\phi$:
\begin{equation}
\label{Finite space: Psi}
\Psi_{h}:=\Big\{\psi_{h}\in \Psi:\psi_{h}|_{\mathrm{E}}\in\mathcal{P}_{k_2}(\mathrm{E})\quad\textup{for all} \; \mathrm{E}\in\mathcal{T}_{h}\Big\}.
\end{equation}
By observing (\ref{Finite space:S_h^k(Omega)}) and (\ref{Finite space: Psi}), we remark a key property
\begin{equation}
\label{Sh-Pish}
\nabla\cdot\boldsymbol{S}_{h}\subseteq \Psi_{h},
\end{equation}
which determines that we can ultimately obtain a divergence-free discrete solution $\boldsymbol{J}_h$.

\begin{rem}
Particularly, taking $k_2=0$, then we will mimic the lowest order RT element. We will not consider this case in this article and will consider it in another article.
\end{rem}

At the end of the subsection, we give the local degrees of freedom for the spaces defined in (\ref{Finite space:V_h^k(E)}) and (\ref{Finite space:S_h^k(E)}), respectively.

 On the one hand, for any $\boldsymbol{v}_{h}\in \boldsymbol{V}_{h}^{k_1}(\mathrm{E})$ satisfies the following local degrees of freedom that recall from \cite{dassi2020bricks,vacca2018h}:
\begin{itemize}
\setstretch{1.25}
 \item[$\bullet$]$\mathrm{D}_{\textup{v}}^{\mathrm{E}}$: the values of $\boldsymbol{v}_{h}$ at the vertexes of the polygon $\mathrm{E}$;
 \item[$\bullet$] $\mathrm{D}_{\textup{e}}^{\mathrm{E}}$:  the value of $\boldsymbol{v}_{h}$ at the $k_1$ distinct internal points of the $(k_1+1)-$point Gauss-Lobatto rule on every edge $e\in\mathrm{\Gamma}_{\mathrm{E}}$;
\item[$\bullet$]$\mathrm{D}_{\textup{m}}^{\mathrm{E}}$: the moments of $\boldsymbol{v}_{h}$ in $\mathrm{E}$
$$(\boldsymbol{v}_h,\boldsymbol{g}_{k_1-2}^{\perp})_{\mathrm{E}}\quad\textup{for all}\; \boldsymbol{g}_{k_1-2}^{\perp}\in\boldsymbol{\mathcal{G}}_{k_1-2}(\mathrm{E})^{\perp};$$
\item[$\bullet$]$\mathrm{D}_{\textup{div}}^{\mathrm{E}}$: the moments of $\textup{div}\boldsymbol{v}$
$$(\nabla\cdot\boldsymbol{v},q_{k_1-1})_{\mathrm{E}}\quad \textup{for all}\;q_{k_1-1}\in\mathcal{P}_{k_1-1}(\mathrm{E})/\mathbb{R}.$$
 \end{itemize}

\begin{prop}
The dimension of $\boldsymbol{V}_{h}^{k_1}(\mathrm{E})$ is $N_{u,k_1}^{\mathrm{E}}=2l_\mathrm{E}k_1+\frac{(k_1-1)(k_1-2)}{2}+\frac{(k_1+1)k_1}{2}-1$, where $l_\mathrm{E}$ denotes total of edges $e\in\mathrm{\Gamma}_{\mathrm{E}}$, and that in $\boldsymbol{V}_{h}^{k_1}(\mathrm{E})$ the degrees of freedom $\mathrm{D}_{i}^{\mathrm{E}}(i:=\textup{v},\textup{e},\textup{m},\textup{div})$ are unisolvent.
\end{prop}
\begin{proof}
The dimension of $\boldsymbol{V}_{h}^{k_1}(\mathrm{E})$ is obvious. The idea of proving unisolvent is that if a function $\boldsymbol{v}\in\boldsymbol{V}_{h}^{k_1}(\mathrm{E})$ such that $\mathrm{D}_{i}^{\mathrm{E}}(i:=n,g,\bot)=0$ is identically zero, then the function is unisolvent. It has been carefully proved in \cite{vacca2018h} and will not be repeated here.
\end{proof}
On the other hand, the $\boldsymbol{K}_{h}\in\boldsymbol{S}_{h}^{k_2}(\mathrm{E})$ satisfies the following local degrees of freedom that recall from \cite{beirao2014h,caceres2017mixed}:
\begin{itemize}
\setstretch{1.25}
\item[$\bullet$]$\mathrm{D}_{\mathrm{n}}^{\mathrm{E}}$:
 $(\boldsymbol{K}_{h}\cdot\boldsymbol{n},q_{h})_{e}\quad$  for all $e\in\mathrm{\Gamma}_{\mathrm{E}}$, $q_{h}\in\mathcal{P}_{k_2}(e)$;
\item[$\bullet$]$\mathrm{D}_{\mathrm{g}}^{\mathrm{E}}$:
$(\boldsymbol{K}_{h},\nabla q_h)_{\mathrm{E}}\quad$  for all $ q_{h}\in\hat{\mathcal{P}}_{k_2\backslash 0}(\mathrm{E})$;
\item[$\bullet$]$\mathrm{D}_{\bot}^{\mathrm{E}}$: $(\boldsymbol{K}_{h},\boldsymbol{q}_{h})_{\mathrm{E}}\quad$  for all $\boldsymbol{q}_{h}\in\boldsymbol{\mathcal{G}}^{\bot}_{k}(\mathrm{E}).$
\end{itemize}
\begin{prop}
The dimension of $\boldsymbol{S}_{h}^{k_2}(\mathrm{E})$ is $N_{J,k_2}^{\mathrm{E}}=l_E(k_2+1)+\frac{(k_2+1)(k_2+2)}{2}-1+\frac{(k_2+1)k_2}{2}$, and that in $\boldsymbol{S}_{h}^{k_2}(\mathrm{E})$ the degrees of freedom $\mathrm{D}_{i}^{\mathrm{E}}(i:=n,g,\bot)$ are unisolvent.
\end{prop}

\begin{proof}
The dimension of $\boldsymbol{S}_{h}^{k_2}(\mathrm{E})$ is obvious. Here, we just briefly illustrate the proof idea of the unisolvent of the local VEM space $\boldsymbol{S}_{h}^{k_2}(\mathrm{E})$, that is if for a given $\boldsymbol{K}_{h}$ in $\boldsymbol{S}_{h}^{k_2}(\mathrm{E})$ all the degrees of freedom $\mathrm{D}_{i}^{\mathrm{E}}(i:=\mathrm{n},\mathrm{g},\bot)$ are zero, then we must have $\boldsymbol{K}_{h}=0$. For the specific argument process, we can refer to \cite{beirao2014h}.
\end{proof}

\begin{rem}
The dimensions of $Q_{h}$ and $\Psi_{h}$ are $\frac{k_1(k_1+1)}{2}N_t$ and $\frac{k_2(k_2+1)}{2}N_t$, respectively. The dimension of $\boldsymbol{V}_{h}$
is $N_t\Big(\frac{(k_1-1)(k_1-2)}{2}+\frac{(k_1+1)k_1}{2}-1\Big)+2(N_v+(k-1)N_e)$.
The dimension of $\boldsymbol{S}_{h}$ is $ N_e(k_2+1)+N_t\Big(\frac{(k_2+1)(k_2+2)}{2}-1+\frac{(k_2+1)k_2}{2}\Big)$.
 Here the $N_t$ (resp., $N_e$ and $N_v$) denotes total number of elements
 (resp., internal edges and vertexes) in $\mathcal{T}_h$.
\end{rem}

\subsection{Projections and computability \label{subsec:projections}}

The polynomial projections play an important role in VEM framework, which are the preconditions of constructing the approximated virtual element forms. Therefore, we give the definition of the required projections as follows:
\begin{itemize}
\setstretch{1.25}
\item[$\bullet$] the elliptic-projection $\mathrm{\Pi}_{k_{1}}^{\nabla,\mathrm{E}}:\boldsymbol{V}_{h}^{k_{1}}(\mathrm{E})\rightarrow[\mathcal{P}_{k_{1}}(\mathrm{E})]^2$,
    $\boldsymbol{v}\rightarrow\mathrm{\Pi}_{k_{1}}^{\nabla,\mathrm{E}}\boldsymbol{v}$ satisfying
\begin{align}
\nonumber
\begin{cases}
(\nabla\mathrm{\Pi}_{k_{1}}^{\nabla,\mathrm{E}}\boldsymbol{v},\nabla\boldsymbol{q}_{k_1})_{\mathrm{E}} =(\nabla\boldsymbol{v},\nabla\boldsymbol{q}_{k_1})_{\mathrm{E}}\\
\mathrm{P}_{0}^{\mathrm{E}}(\mathrm{\Pi}_{k_{1}}^{\nabla,\mathrm{E}}\boldsymbol{v}) = \mathrm{P}_{0}^{\mathrm{E}}(\boldsymbol{v})
\end{cases}
\quad \forall\boldsymbol{q}_{k_1}\in[\mathcal{P}_{k_{1}}(\mathrm{E})]^2,
\end{align}
where the $\mathrm{P}_{0}^{\mathrm{E}}$ denotes the constant projector on $\mathrm{E}$.

\item[$\bullet$] the $L^2$-projection $\mathrm{\Pi}_{k_{1}}^{0,\mathrm{E}}: \boldsymbol{V}_{h}^{k_{1}}(\mathrm{E})\rightarrow[\mathcal{P}_{k_{1}}(\mathrm{E})]^2$,
     $\boldsymbol{v}\rightarrow\mathrm{\Pi}_{k_{1}}^{0,\mathrm{E}}\boldsymbol{v}$ satisfying
\begin{align}
\nonumber
(\mathrm{\Pi}_{k_{1}}^{0,\mathrm{E}}\boldsymbol{v},\boldsymbol{q}_{k_1})_{\mathrm{E}} =(\boldsymbol{v},\boldsymbol{q}_{k_1})_{\mathrm{E}}\quad \forall\boldsymbol{q}_{k_1}\in[\mathcal{P}_{k_{1}}(\mathrm{E})]^2.
\end{align}

\item[$\bullet$] the $L^2$-projection $\mathrm{\Pi}_{k_{1}-1}^{0,\mathrm{E}}: \nabla\big(\boldsymbol{V}_{h}^{k_{1}}(\mathrm{E})\big)\rightarrow[\mathcal{P}_{k_{1}}(\mathrm{E})]^{2\times2}$, $\nabla\boldsymbol{v}\rightarrow\mathrm{\Pi}_{k_{1}-1}^{0,\mathrm{E}}\boldsymbol{v}$ satisfying
\begin{align}
\nonumber
(\mathrm{\Pi}_{k_{1}-1}^{0,\mathrm{E}}\nabla\boldsymbol{v},\boldsymbol{q}_{k_1})_{\mathrm{E}} =(\nabla\boldsymbol{v},\boldsymbol{q}_{k_1})_{\mathrm{E}}\quad \forall\boldsymbol{q}_{k_1}\in[\mathcal{P}_{k_{1}}(\mathrm{E})]^{2\times2}.
\end{align}

\item[$\bullet$] the $L^2$-projection $\mathrm{\Pi}_{k_{2}}^{0,\mathrm{E}}: \boldsymbol{S}_{h}^{k_{2}}(\mathrm{E})\rightarrow[\mathcal{P}_{k_{2}}(\mathrm{E})]^2$, $\boldsymbol{K}\rightarrow\mathrm{\Pi}_{k_{2}}^{0,\mathrm{E}}\boldsymbol{K}$ satisfying
\begin{align}
\nonumber
(\mathrm{\Pi}_{k_{2}}^{0,\mathrm{E}}\boldsymbol{K},\boldsymbol{q}_{k_2})_{\mathrm{E}} =(\boldsymbol{K},\boldsymbol{q}_{k_2})_{\mathrm{E}}\quad \forall\boldsymbol{q}_{k_2}\in[\mathcal{P}_{k_{2}}(\mathrm{E})]^2.
\end{align}
\end{itemize}

Now, we consider the computability of these projections, see \cite{dassi2020bricks,caceres2017mixed} for the more details.

\begin{prop}
\label{Projection:ellpro}
Applying the degrees of freedom $\mathrm{D}_{i}^{\mathrm{E}}(i:=\textup{v},\textup{e},\textup{m},\textup{div})$, we can compute exactly the projection $\mathrm{\Pi}_{k_{1}}^{\nabla,\mathrm{E}}$, which is equivalent to the moment
\begin{equation}
\label{moments:vq}
(\nabla\boldsymbol{v},\nabla\boldsymbol{q}_{k_{1}})_\mathrm{E},
\end{equation}
for all $\boldsymbol{v}\in\boldsymbol{V}_{h}^{k_1}(\mathrm{E})$, $\boldsymbol{q}_{k_{1}}\in[\mathcal{P}_{k_{1}}(\mathrm{E})]^2$ is computable.
\end{prop}

\begin{proof}
Integrating by parts yields
\begin{align}
\label{moments:nablav nablaq}
(\nabla\boldsymbol{v},\nabla\boldsymbol{q}_{k_{1}})_\mathrm{E} =
-(\boldsymbol{v},\Delta\boldsymbol{q}_{k_{1}})_\mathrm{E} +(\boldsymbol{v},\nabla\boldsymbol{q}_{k_{1}}\boldsymbol{n})_{\partial\mathrm{E}}.
\end{align}
Since $\Delta\boldsymbol{q}_{k_{1}}\in[\mathcal{P}_{k_1-2}(\mathrm{E})]^2$, there exists $q_{\alpha}\in\hat{\mathcal{P}}_{k_1-1\backslash 0}(\mathrm{E})$ and $\boldsymbol{g}_{\alpha}\in\boldsymbol{\mathcal{G}}_{k_1-2}(\mathrm{E})^{\perp}$ such that
\begin{align}
\label{decompose:Delta P_k^2}
\Delta\boldsymbol{q}_{k_{1}} = \nabla q_{\alpha} + \boldsymbol{g}_{\alpha},
\end{align}
applying (\ref{decompose:Delta P_k^2}) and Green formula to (\ref{moments:nablav nablaq}), we get
\begin{align}
\nonumber
(\nabla\boldsymbol{v},\nabla\boldsymbol{q}_{k_{1}})_\mathrm{E} &=
-(\boldsymbol{v},\nabla q_{\alpha})_\mathrm{E}
-(\boldsymbol{v},\boldsymbol{g}_{\alpha})_{\mathrm{E}}
+(\boldsymbol{v},\nabla\boldsymbol{q}_{k_{1}}\boldsymbol{n})_{\partial\mathrm{E}}\\\nonumber
&=(\nabla\cdot\boldsymbol{v},q_{\alpha})_{\mathrm{E}}-(\boldsymbol{v},\boldsymbol{g}_{\alpha})_{\mathrm{E}}
+(\boldsymbol{v},\nabla\boldsymbol{q}_{k_1}\boldsymbol{n}-q_{\alpha}\boldsymbol{n})_{\partial\mathrm{E}}.
\end{align}
Clearly, the first term can be computed by $\mathrm{D}_{\textup{div}}^{\mathrm{E}}$, the second term can be computed by $\mathrm{D}_{\textup{m}}^{\mathrm{E}}$. Moreover, we note that the virtual function $\boldsymbol{v}$ is a known vectorial polynomial on the boundary and
the vectorial polynomial $\nabla\boldsymbol{q}_{k_{1}}\boldsymbol{n}-q_{\alpha}\boldsymbol{n}$ of degree $k_{1}-1$, so we integrate a polynomial of degree $2k_{1}-1$ over edges by
$\mathrm{D}_{\textup{v}}^{\mathrm{E}}$ and $\mathrm{D}_{\textup{e}}^{\mathrm{E}}$.
It is noteworthy that $\boldsymbol{g}_{\alpha}=\boldsymbol{0}$ in the case of $k_{1} = 2$. The proof is over.
\end{proof}

\begin{prop}
\label{Projection:L^2k1-1}
Applying the degrees of freedom $\mathrm{D}_{i}^{\mathrm{E}}(i:=\textup{v},\textup{e},\textup{m},\textup{div})$, we can compute exactly the $L^2$-projection $\mathrm{\Pi}_{k_{1}-1}^{0,\mathrm{E}}$, which is equivalent to the moment
\begin{equation}
\label{moments:vq}
(\nabla\boldsymbol{v},\boldsymbol{q}_{k_{1}})_\mathrm{E},
\end{equation}
for all $\boldsymbol{v}\in\boldsymbol{V}_{h}^{k_1}(\mathrm{E})$, $\boldsymbol{q}_{k_{1}}\in[\mathcal{P}_{k_{1}}(\mathrm{E})]^{2\times2}$ is computable.
\end{prop}

\begin{proof}
Integrating by parts yields
\begin{align}
\nonumber
(\nabla\boldsymbol{v},\boldsymbol{q}_{k_1})_\mathrm{E}=
-(\boldsymbol{v},\nabla\cdot\boldsymbol{q}_{k_1})_{\mathrm{E}}+
(\boldsymbol{v},\boldsymbol{q}_{k_1}\boldsymbol{n})_{\partial\mathrm{E}}.
\end{align}
%The integral over the boundary is computable since the virtual function $\boldsymbol{v}$ is a vectorial polynomial on the boundary.
 Obviously, the computability of $\mathrm{\Pi}_{k_{1}-1}^{0,\mathrm{E}}$ follows from same arguments of $\mathrm{\Pi}_{k_{1}}^{\nabla,\mathrm{E}}$.
\end{proof}

\begin{prop}
\label{Projection:L^2k1}
Applying the degrees of freedom $\mathrm{D}_{i}^{\mathrm{E}}(i:=\textup{v},\textup{e},\textup{m},\textup{div})$, we can compute exactly the $L^2$-projection $\mathrm{\Pi}_{k_{1}}^{0,\mathrm{E}}$, which is equivalent to the moment
\begin{equation}
\label{moments:vq}
(\boldsymbol{v},\boldsymbol{q}_{k_{1}})_\mathrm{E},
\end{equation}
for all $\boldsymbol{v}\in\boldsymbol{V}_{h}^{k_1}(\mathrm{E})$, $\boldsymbol{q}_{k_{1}}\in[\mathcal{P}_{k_{1}}(\mathrm{E})]^2$ is computable.
\end{prop}

\begin{proof}
From subsection \ref{subsec:notes}, we have
\begin{equation}
\label{decompose:P_k1^2}
\boldsymbol{q}_{k_1} = \nabla q_{\alpha}+\boldsymbol{g}_{\alpha},
\end{equation}
where $\boldsymbol{q}_{k_1}\in[\mathcal{P}_{k_{1}}(\mathrm{E})]^2$, $q_{\alpha}\in\hat{\mathcal{P}}_{k_1+1\backslash 0}(\mathrm{E})$ and $\boldsymbol{g}_{\alpha}\in\boldsymbol{\mathcal{E}}_{k_1}(\mathrm{E})^{\perp}$. Next, applying the (\ref{decompose:P_k1^2}) and Green formula to (\ref{moments:vq}), we obtain
\begin{align}
\nonumber
(\boldsymbol{v},\boldsymbol{q}_{k_{1}})_\mathrm{E} &= (\boldsymbol{v},\nabla q_{\alpha}+\boldsymbol{g}_{\alpha})_\mathrm{E}\\\nonumber
&=-(\nabla\cdot\boldsymbol{v},q_{\alpha})_\mathrm{E} +
(\boldsymbol{v}\cdot\boldsymbol{n},q_{\alpha})_{\partial\mathrm{E}}+
(\boldsymbol{v},\boldsymbol{g}_{\alpha})_\mathrm{E}.
\end{align}
Evidently, the first term can be computed by $\mathrm{D}_{\textup{div}}^{\mathrm{E}}$ if the degrees of $q_{\alpha}$ satisfy $1\leq\alpha\leq k_{1}-1$, while since $\nabla\cdot\boldsymbol{v}\in[\mathcal{P}_{k_{1}}(\mathrm{E})]^2$ can be computed by $\mathrm{D}_{\textup{i}}^{\mathrm{E}}(i:=\textup{v},\textup{e},\textup{div})$
such that we compute such integral exactly in domain $\mathrm{E}$ when the degree of $q_{\alpha}$ satisfies $k_{1}\leq\alpha\leq k_{1}+1$. Regarding the second term, we note that the virtual function $\boldsymbol{v}$ is a known vectorial polynomial of degree $k_1$ on the boundary and the $q_{\alpha}\in\hat{\mathcal{P}}_{k_1+1\backslash 0}(\mathrm{E})$, so we are integrating a polynomial of degree $2k_1+1$ that not sufficient to compute exactly by $\mathrm{D}_{\textup{v}}^{\mathrm{E}}$ and $\mathrm{D}_{\textup{e}}^{\mathrm{E}}$. In such case we can reconstruct the vectorial polynomial $\boldsymbol{v}$ on the boundary and then employ a quadrature rule of degree $2k_1+1$, namely using Gauss-Lobatto rule with $k_1+2$ quadrature nodes. For the final term, we have
\begin{align}
\nonumber
\begin{cases}
\text{using} \;\mathrm{D}_{\textup{m}}^{\mathrm{E}}\quad &\text{if}\;\boldsymbol{g}_{\alpha}\in\boldsymbol{\mathcal{G}}_{k_1-2}(\mathrm{E})^{\perp},\\
(\boldsymbol{v},\boldsymbol{g}_{\alpha})_\mathrm{E} =
(\mathrm{\Pi}_{k_1}^{\nabla,E}\boldsymbol{v},\boldsymbol{g}_{\alpha})_\mathrm{E}
\quad&\text{if}\;\boldsymbol{g}_{\alpha}
\in\boldsymbol{\mathcal{G}}_{k_1}^{\perp}(\mathrm{E})/\boldsymbol{\mathcal{G}}_{k_1-2}^{\perp}(\mathrm{E}).
\end{cases}
\end{align}
In summary, the proof is over.
\end{proof}

\begin{prop}
\label{Projection:L^2k2}
Applying the degrees of freedom $\mathrm{D}_{i}^{\mathrm{E}}(i:=n,g,\bot)$, we can compute exactly the $L^2$-projection $\mathrm{\Pi}_{k_{2}}^{0,\mathrm{E}}$, which is equivalent to the moment
\begin{equation}
\label{moments:vq}
(\boldsymbol{K},\boldsymbol{q}_{k_{2}})_\mathrm{E},
\end{equation}
for all $\boldsymbol{K}\in\boldsymbol{S}_{h}^{k_2}(\mathrm{E})$, $\boldsymbol{q}_{k_{2}}\in[\mathcal{P}_{k_{2}}(\mathrm{E})]^2$ is computable.
\end{prop}
\begin{proof}
From subsection \ref{subsec:notes}, we have
\begin{equation}
\label{decompose:P_k2^2}
\boldsymbol{q}_{k_2} = \nabla q_{\alpha}+\boldsymbol{g}_{\alpha},
\end{equation}
where $\boldsymbol{q}_{k_2}\in[\mathcal{P}_{k_{2}}(\mathrm{E})]^2$, $q_{\alpha}\in\hat{\mathcal{P}}_{k_2+1\backslash 0}(\mathrm{E})$ and $\boldsymbol{g}_{\alpha}\in\boldsymbol{\mathcal{G}}_{k_2}(\mathrm{E})^{\perp}$. Next, applying the (\ref{decompose:P_k2^2}) and Green formula to (\ref{moments:vq}), we obtain
\begin{align}
\nonumber
(\boldsymbol{K},\boldsymbol{q}_{k_{1}})_\mathrm{E} &= (\boldsymbol{K},\nabla q_{\alpha}+\boldsymbol{g}_{\alpha})_\mathrm{E}\\\nonumber
&=-(\nabla\cdot\boldsymbol{K},q_{\alpha})_\mathrm{E} +
(\boldsymbol{K}\cdot\boldsymbol{n},q_{\alpha})_{\partial\mathrm{E}}+
(\boldsymbol{K},\boldsymbol{g}_{\alpha})_\mathrm{E}.
\end{align}
 Thanks to $\boldsymbol{K}\cdot\boldsymbol{n}$ can be computed by $\mathrm{D}_{\mathrm{n}}^{\mathrm{E}}$ and the definition of $\mathrm{D}_{\mathrm{g}}^{\mathrm{E}}$, then the $\nabla\cdot\boldsymbol{K}\in\mathcal{P}_{k_2}(E)$ is computed by using the identity
\begin{align}
\label{Div:J}
(\nabla\cdot\boldsymbol{K},q_{\alpha})_{\mathrm{E}} = -(\boldsymbol{K},\nabla q_{\alpha})_{\mathrm{E}}+(\boldsymbol{K}\cdot\boldsymbol{n},q_{\alpha})_{\partial\mathrm{E}},
\end{align}
so the first term and second term are computable. Obviously, the final term can be computed by $\mathrm{D}_{\bot}^{\mathrm{E}}$.
\end{proof}
\subsection{The construction of discrete bilinear and trilinear forms\label{subsec:bilinear and trilinear}}
In this subsection, we consider to construct discrete versions of bilinear and trilinear forms by using the projection operators and virtual spaces defined in subsections (\ref{subsec:space}) and (\ref{subsec:projections}), respectively. Before doing so, we decompose the bilinear and trilinear forms in the model into the following local contributions
\begin{align}
\nonumber
&a_{v}(\boldsymbol{u},\boldsymbol{v})=\sum_{\mathrm{E}\in\mathrm{\mathcal{T}}_{h}}a_{v}^{\mathrm{E}}(\boldsymbol{u},\boldsymbol{v})
\quad \forall\boldsymbol{u},\boldsymbol{v}\in \boldsymbol{V},\quad
a_{K}(\boldsymbol{J},\boldsymbol{K})=\sum_{\mathrm{E}\in\mathrm{\mathcal{T}}_{h}}a_{K}^{\mathrm{E}}(\boldsymbol{J},\boldsymbol{K})
\quad \forall\boldsymbol{J},\boldsymbol{K}\in \boldsymbol{S},\\\nonumber
&b_{q}(q,\boldsymbol{v})=\sum_{\mathrm{E}\in\mathrm{\mathcal{T}}_{h}}b_{q}^{\mathrm{E}}(q,\boldsymbol{v})
\quad \forall q\in Q\;\text{and}\;\boldsymbol{v}\in\boldsymbol{V},\quad
b_{\psi}(\psi,\boldsymbol{K})=\sum_{\mathrm{E}\in\mathrm{\mathcal{T}}_{h}}b_{\psi}^{\mathrm{E}}(\psi,\boldsymbol{K})
\quad \forall \psi\in \Psi\;\text{and}\;\boldsymbol{K}\in\boldsymbol{S},\\\nonumber
&d_{K}(\boldsymbol{K},\boldsymbol{v})=\sum_{\mathrm{E}\in\mathrm{\mathcal{T}}_{h}}d_{K}^{\mathrm{E}}(\boldsymbol{K},\boldsymbol{v})
\quad \forall\boldsymbol{K}\in \boldsymbol{S}\;\text{and}\;\boldsymbol{v}\in\boldsymbol{V},\quad
c_{v}(\boldsymbol{w},\boldsymbol{u},\boldsymbol{v})=\sum_{\mathrm{E}\in\mathrm{\mathcal{T}}_{h}}c_{v}^{\mathrm{E}}(\boldsymbol{w},\boldsymbol{u},\boldsymbol{v})
\quad \forall\boldsymbol{w}, \boldsymbol{u},\boldsymbol{v}\in \boldsymbol{V},
\end{align}
and the norms $|\cdot|_{k,p}$ and $\|\cdot\|_{k,p}$ are defined as
\begin{align}
\nonumber
&|\boldsymbol{v}|_{k,p} = \Big(\sum_{\mathrm{E}\in\mathrm{\mathcal{T}}_{h}}|\boldsymbol{v}|_{\boldsymbol{W}^{k,p}(\mathrm{E})}^2\Big)^{1/2},
\quad\|\boldsymbol{v}\|_{k,p} = \Big(\sum_{\mathrm{E}\in\mathrm{\mathcal{T}}_{h}}\|\boldsymbol{v}\|_{\boldsymbol{W}^{k,p}(\mathrm{E})}^2\Big)^{1/2}\quad\forall\boldsymbol{v}\in\boldsymbol{W}^{k,p}.
\end{align}

Now, we first consider computability of  the bilinear terms  $b_{q}(\cdot,\cdot)$ and $b_{\psi}(\cdot,\cdot)$, we set
\begin{align}
\label{bilinear:dis-term-v}
b_{q}(q,\boldsymbol{v})=\sum_{\mathrm{E}\in\mathrm{\mathcal{T}}_{h}}b_{q}^{\mathrm{E}}(q,\boldsymbol{v})
=\sum_{\mathrm{E}\in\mathrm{\mathcal{T}}_{h}}(\nabla\cdot\boldsymbol{v},q)_{\mathrm{E}} \quad \forall q\in Q_{h}\;\text{and}\;\boldsymbol{v}\in\boldsymbol{V}_{h}^{k_1},\\
b_{\psi}(\psi,\boldsymbol{K})=\sum_{\mathrm{E}\in\mathrm{\mathcal{T}}_{h}}b_{\psi}^{\mathrm{E}}(\psi,\boldsymbol{K})
=\sum_{\mathrm{E}\in\mathrm{\mathcal{T}}_{h}}S_c(\nabla\cdot\boldsymbol{K},\psi)_{\mathrm{E}} \quad \forall \psi\in \Psi_{h}\;\text{and}\;\boldsymbol{K}\in\boldsymbol{S}_{h}^{k_2}.\label{bilinear:dis-term-J}
\end{align}
It is not hard to notice that (\ref{bilinear:dis-term-v})-(\ref{bilinear:dis-term-J}) can be computed by using degrees of freedom $\mathrm{D}_{i}^{\mathrm{E}}(i:=\textup{v},\textup{e},\textup{div})$ and $\mathrm{D}_{i}^{\mathrm{E}}(i:=\mathrm{n},\mathrm{g})$, respectively. Clearly,
we do not introduce any approximation for the above two bilinear terms.

Next, we continue to deal with bilinear term $a_{v}^{\mathrm{E}}(\cdot,\cdot)$ by using a more careful way. Look out, the quantity $a_{v}(\boldsymbol{u},\boldsymbol{v})$ is not computable for an arbitrary pair $(\boldsymbol{u},\boldsymbol{v})
\in\boldsymbol{V}_{h}^{k_1}(\mathrm{E})\times\boldsymbol{V}_{h}^{k_1}(\mathrm{E})$. So we follow a standard procedure in the VEM framework to define a computable discrete local bilinear form
\begin{equation}
\label{discrete-local av}
a_{v_h}^{\mathrm{E}}(\cdot,\cdot):\;\boldsymbol{V}_{h}^{k_1}(\mathrm{E})
\times\boldsymbol{V}_{h}^{k_1}(\mathrm{E})\rightarrow\mathbb{R},
\end{equation}
and then we use the above formula to approximate the continuous form $a_{v}^{\mathrm{E}}(\cdot,\cdot)$. In addition, the $a_{v_h}^{\mathrm{E}}(\cdot,\cdot)$ satisfy the following properties:
\begin{itemize}
\setstretch{1.25}
 \item[$\bullet$] $k_1$-consistency:
 \begin{equation}
 \label{k_1-consistency}
 a_{v_h}^{\mathrm{E}}(\boldsymbol{q}_{k_1},\boldsymbol{v}) = a_{v}^{\mathrm{E}}(\boldsymbol{q}_{k_1},\boldsymbol{v}),
 \end{equation}
 for all $\boldsymbol{q}_{k_1}\in[\mathcal{P}_{k_1}(\mathrm{E})]^2$ and $\boldsymbol{v}\in\boldsymbol{V}_{h}^{k_1}(\mathrm{E})$.
 \item[$\bullet$] stability:
 \begin{equation}
 \label{stability-uh}
 \alpha_{*}a_{v}^{\mathrm{E}}(\boldsymbol{v},\boldsymbol{v})\leq
 a_{v_h}^{\mathrm{E}}(\boldsymbol{v},\boldsymbol{v})\leq
 \alpha^{*}a_{v}^{\mathrm{E}}(\boldsymbol{v},\boldsymbol{v}),
 \end{equation}
 for all $\boldsymbol{v}\in\boldsymbol{V}_{h}^{k_1}(\mathrm{E})$, these two positive constants $\alpha_{*}$ and $\alpha^{*}$ are independent of $h$ and $\mathrm{E}$.
\end{itemize}
 Meanwhile, we set $\mathcal{S}_{v}^{\mathrm{E}}:\boldsymbol{V}_{h}^{k_1}(\mathrm{E})\times\boldsymbol{V}_{h}^{k_1}(\mathrm{E})\rightarrow\mathbb{R}$ is a (symmetric) stabilizing bilinear form which satisfies
\begin{equation}
\label{define-S^E_v}
s_{*}a_{v}^{\mathrm{E}}(\boldsymbol{v},\boldsymbol{v})\leq
 \mathcal{S}_{v}^{\mathrm{E}}(\boldsymbol{v},\boldsymbol{v})\leq
 s^{*}a_{v}^{\mathrm{E}}(\boldsymbol{v},\boldsymbol{v})\quad\forall\;\boldsymbol{v}\in \boldsymbol{V}_{h}^{k_1}(\mathrm{E}),
\end{equation}
where the two positive constants $s_{*}$ and $s^{*}$ are independent of $h$ and $\mathrm{E}$.
Based on the above results, we define
\begin{equation}
 \label{stabilizing form-u}
a_{v_h}^{\mathrm{E}}(\boldsymbol{u},\boldsymbol{v}):=
a_{v}^{\mathrm{E}}(\mathrm{\Pi}_{k_1}^{\nabla,\mathrm{E}}\boldsymbol{u},\mathrm{\Pi}_{k_1}^{\nabla,\mathrm{E}}\boldsymbol{v})
+\mathcal{S}_{v}^{\mathrm{E}}\Big((\mathrm{I}-\mathrm{\Pi}_{k_1}^{\nabla,\mathrm{E}})\boldsymbol{u},(\mathrm{I}-\mathrm{\Pi}_{k_1}^{\nabla,\mathrm{E}})\boldsymbol{v}\Big),
\end{equation}
for all $\boldsymbol{u}$, $\boldsymbol{v}\in\boldsymbol{V}_{h}^{k_1}(\mathrm{E})$.
Clearly, we note that definition of the elliptic-projection $\mathrm{\Pi}_{k_1}^{\nabla,\mathrm{E}}$ and property (\ref{define-S^E_v}) imply the consistency and the stability of the bilinear form $a_{v_h}^{\mathrm{E}}(\cdot,\cdot)$.
Obviously, regarding to global approximated bilinear form $a_{v_h}(\cdot,\cdot):\boldsymbol{V}_{h}\times\boldsymbol{V}_{h}\rightarrow\mathbb{R}$, we can define it as follows:
\begin{equation}
\label{discrete-global-v}
a_{v_h}(\boldsymbol{u},\boldsymbol{v}):=\sum_{\mathrm{E}\in\mathcal{T}_{h}}a_{v_h}^{\mathrm{E}}(\boldsymbol{u},\boldsymbol{v})
\quad\forall \boldsymbol{u},\boldsymbol{v}\in\boldsymbol{V}_{h}.
\end{equation}
By noticing that the symmetry of $a_{v_h}^{\mathrm{E}}(\cdot,\cdot)$, property (\ref{stability-uh}) and the definiton of $a_{v}^{\mathrm{E}}(\cdot,\cdot)$ imply the continuity of $a_{v_h}^{\mathrm{E}}$:
\begin{align}
a_{v_h}^{\mathrm{E}}(\boldsymbol{u},\boldsymbol{v})\leq(a_{v_h}^{\mathrm{E}}(\boldsymbol{u},\boldsymbol{u}))^{\frac{1}{2}}
(a_{v_h}^{\mathrm{E}}(\boldsymbol{v},\boldsymbol{v}))^{\frac{1}{2}}
\leq(a_{v}^{\mathrm{E}}(\boldsymbol{u},\boldsymbol{u}))^{\frac{1}{2}}
(a_{v}^{\mathrm{E}}(\boldsymbol{v},\boldsymbol{v}))^{\frac{1}{2}}
\leq \alpha^{*}|\boldsymbol{u}|_{1,\mathrm{E}}|\boldsymbol{v}|_{1,\mathrm{E}},
\label{continue-E-a_v}
\end{align}
and using the H\'{o}lder inequality to get
\begin{equation}
a_{v_h}(\boldsymbol{u},\boldsymbol{v})\leq  \alpha^{*}|\boldsymbol{u}|_{1,\mathrm{\Omega}}|\boldsymbol{v}|_{1,\mathrm{\Omega}}\quad \text{for all } \boldsymbol{u},\boldsymbol{K}\in\boldsymbol{V}_{h}.
\label{continue-Omega-a_v}
\end{equation}

The above method for processing $a_{v}(\cdot,\cdot)$ is also applicable to bilinear term $a_{K}(\cdot,\cdot)$. The $a_{K_h}^{\mathrm{E}}(\cdot,\cdot)$ satisfy the following properties:
\begin{itemize}
\setstretch{1.25}
 \item[$\bullet$] $k_2$-consistency:
 \begin{equation}
 \label{k_2-consistency}
 a_{K_h}^{\mathrm{E}}(\boldsymbol{q}_{k_2},\boldsymbol{K}) = a_{K}^{\mathrm{E}}(\boldsymbol{q}_{k_2},\boldsymbol{K}),
 \end{equation}
 for all
 $\boldsymbol{q}_{k_2}\in[\mathcal{P}_{k_2}(\mathrm{E})]^2$ and $ \boldsymbol{K}\in\boldsymbol{S}_{h}^{k_2}(\mathrm{E})$.
 \item[$\bullet$] stability:
 \begin{equation}
 \label{stability-Jh}
 \eta_{*}a_{K}^{\mathrm{E}}(\boldsymbol{K},\boldsymbol{K})\leq
 a_{K_h}^{\mathrm{E}}(\boldsymbol{K},\boldsymbol{K})\leq
 \eta^{*}a_{K}^{\mathrm{E}}(\boldsymbol{K},\boldsymbol{K}),
 \end{equation}
 for all $\boldsymbol{K}\in\boldsymbol{S}_{h}^{k_2}(\mathrm{E})$, these two positive constants $\eta_{*}$ and $\eta^{*}$ are independent of $h$ and $\mathrm{E}$.
\end{itemize}
 Meanwhile, we set $\mathcal{S}_{K}^{\mathrm{E}}:\boldsymbol{S}_{h}^{k_2}(\mathrm{E})\times\boldsymbol{S}_{h}^{k_2}(\mathrm{E})\rightarrow\mathbb{R}$ is a (symmetric) stabilizing bilinear form which satisfies
\begin{equation}
\label{define-S^E_K}
s_{**}a_{K}^{\mathrm{E}}(\boldsymbol{K},\boldsymbol{K})\leq
 \mathcal{S}_{K}^{\mathrm{E}}(\boldsymbol{K},\boldsymbol{K})\leq
 s^{**}a_{K}^{\mathrm{E}}(\boldsymbol{K},\boldsymbol{K})\quad\text{for all}\;\boldsymbol{K}\in \boldsymbol{S}_{h}^{k_2}(\mathrm{E}),
\end{equation}
where the two positive constants $s_{**}$ and $s^{**}$ are independent of $h$ and $\mathrm{E}$.
Based on the above results, we define
\begin{equation}
\label{stability form -J}
a_{K_h}^{\mathrm{E}}(\boldsymbol{J},\boldsymbol{K}):=
a_{K}^{\mathrm{E}}(\mathrm{\Pi}_{k_2}^{0,\mathrm{E}}\boldsymbol{J},\mathrm{\Pi}_{k_2}^{0,\mathrm{E}}\boldsymbol{K})
+\mathcal{S}_{K}^{\mathrm{E}}\Big((\mathrm{I}-\mathrm{\Pi}_{k_2}^{0,\mathrm{E}})\boldsymbol{J},(\mathrm{I}-\mathrm{\Pi}_{k_2}^{0,\mathrm{E}})\boldsymbol{K}\Big),
\end{equation}
for all $\boldsymbol{J}$, $\boldsymbol{K}\in\boldsymbol{S}_{h}^{k_2}(\mathrm{E})$.
Evidently, we note that definition of the $L^2$-projection $\mathrm{\Pi}_{k_2}^{0,\mathrm{E}}$ and property (\ref{define-S^E_K}) imply the consistency and the stability of the bilinear form $a_{K_h}^{\mathrm{E}}(\cdot,\cdot)$. Regarding to the global approximated bilinear form $a_{K_h}(\cdot,\cdot):\boldsymbol{S}_{h}\times\boldsymbol{S}_{h}\rightarrow\mathbb{R}$, we can define it as follows:
\begin{equation}
\label{discrete-global-K}
a_{K_h}(\boldsymbol{J},\boldsymbol{K}):=\sum_{\mathrm{E}\in\mathcal{T}_{h}}a_{K_h}^{\mathrm{E}}(\boldsymbol{J},\boldsymbol{K})
\quad\forall \boldsymbol{J},\boldsymbol{K}\in\boldsymbol{S}_{h}.
\end{equation}
By noticing that the symmetry of $a_{K_h}^{\mathrm{E}}(\cdot,\cdot)$, property (\ref{stability-Jh}) and the definiton of $a_{K}^{\mathrm{E}}(\cdot,\cdot)$ imply the continuity of $a_{K_h}^{\mathrm{E}}$:
\begin{align}
a_{K_h}^{\mathrm{E}}(\boldsymbol{J},\boldsymbol{K})\leq\big(a_{K_h}^{\mathrm{E}}(\boldsymbol{J},\boldsymbol{J})\big)^{\frac{1}{2}}
\big(a_{K_h}^{\mathrm{E}}(\boldsymbol{J},\boldsymbol{K})\big)^{\frac{1}{2}}
\leq\big(a_{K}^{\mathrm{E}}(\boldsymbol{J},\boldsymbol{J})\big)^{\frac{1}{2}}
\big(a_{K}^{\mathrm{E}}(\boldsymbol{K},\boldsymbol{K})\big)^{\frac{1}{2}}
\leq\eta^{*}|\boldsymbol{J}|_{1,\mathrm{E}}|\boldsymbol{K}|_{1,\mathrm{E}},
\label{continue-E-a_K}
\end{align}
and using the H\'{o}lder inequality to get
\begin{equation}
a_{K_h}(\boldsymbol{J},\boldsymbol{K})\leq  \eta^{*}|\boldsymbol{J}|_{1,\mathrm{\Omega}}|\boldsymbol{K}|_{1,\mathrm{\Omega}},
\label{continue-Omega-a_K}
\end{equation}
for all $\boldsymbol{J},\boldsymbol{K}\in\boldsymbol{S}_{h}$.

Regarding to the approximation of the local trilinear form $c_{v}(\cdot,\cdot,\cdot)$, we set
\begin{equation}
\label{trilinear:local discrete}
c_{v_h}^{\mathrm{E}}(\boldsymbol{w},\boldsymbol{u},\boldsymbol{v}):=\frac{1}{2}\Big(\mathrm{\Pi}_{k_1}^{0,\mathrm{E}}\boldsymbol{w}
\cdot\mathrm{\Pi}_{k_1-1}^{0,\mathrm{E}}\nabla\boldsymbol{u},\mathrm{\Pi}_{k_1}^{0,\mathrm{E}}\boldsymbol{v}\Big)_{\mathrm{E}}
-\frac{1}{2}\Big(\mathrm{\Pi}_{k_1}^{0,\mathrm{E}}\boldsymbol{w}
\cdot\mathrm{\Pi}_{k_1-1}^{0,\mathrm{E}}\nabla\boldsymbol{v},\mathrm{\Pi}_{k_1}^{0,\mathrm{E}}\boldsymbol{u}\Big)_{\mathrm{E}}\quad\forall\boldsymbol{w},\boldsymbol{u},\boldsymbol{v}\in\boldsymbol{V}_{h}^{k_1}(\mathrm{E}),
\end{equation}
and then we readily check that the all quantities in (\ref{trilinear:local discrete}) are computable by observing Proposition \ref{Projection:L^2k1-1}, \ref{Projection:L^2k1}. Naturally, we define the global approximated trilinear form as follows:
\begin{equation}
\label{trilinear:discrete-global}
c_{v_h}(\boldsymbol{w},\boldsymbol{u},\boldsymbol{v}):=\sum_{\mathrm{E}\in\mathcal{T}_{h}}c_{v_h}^{\mathrm{E}}(\boldsymbol{w},\boldsymbol{u},\boldsymbol{v})\quad
\forall\boldsymbol{w},\boldsymbol{u},\boldsymbol{v}\in\boldsymbol{V}_{h}.
\end{equation}
\begin{prop}[\cite{da2018virtual}]
The trilinear form $c_{v_h}(\cdot,\cdot,\cdot)$ is uniformly continuous with respect to $h$, namely,
\begin{equation}
\delta:=\sup_{\boldsymbol{u},\boldsymbol{v},\boldsymbol{w},\in\boldsymbol{V}_{h}}
\frac{c_{v_h}(\boldsymbol{u},\boldsymbol{v},\boldsymbol{w})}{|\boldsymbol{u}|_{1,\mathrm{\Omega}}|\boldsymbol{v}|_{1,\mathrm{\Omega}}|\boldsymbol{w}|_{1,\mathrm{\Omega}}},
\label{prop:bilinear-continuous-uh}
\end{equation}
where the constant $\delta$ independent on $h$. The skew-symmetric property
 \begin{equation}
 c_{v_h}(\boldsymbol{u},\boldsymbol{v},\boldsymbol{v})=0\quad\forall\boldsymbol{u}, \boldsymbol{v}
 \in\boldsymbol{V}_{h}.
 \label{skew-symmetric-uh}
 \end{equation}
\end{prop}
Finally, regarding to the approximation of the local bilinear form $d_{K}(\cdot,\cdot)$, we set
\begin{equation}
\label{bilinear:local discrete-JXBv}
d_{K_h}^{\mathrm{E}}(\boldsymbol{K},\boldsymbol{v}):=S_c\Big(\mathrm{\Pi}_{k_2}^{0,\mathrm{E}}\boldsymbol{K}
\times\boldsymbol{B},\mathrm{\Pi}_{k_1}^{0,\mathrm{E}}\boldsymbol{v}\Big)_{\mathrm{E}}
\quad\forall\boldsymbol{K}\in\boldsymbol{S}_{h}^{k_2}(\mathrm{E}),\boldsymbol{v}\in\boldsymbol{V}_{h}^{k_1}(\mathrm{E}),
\end{equation}
and then we readily check that the all quantities in (\ref{bilinear:local discrete-JXBv}) are computable by observing Proposition \ref{Projection:L^2k1}, \ref{Projection:L^2k2}. Naturally, we define the global approximated bilinear form as follows
\begin{equation}
\label{trilinear:discrete-global}
d_{K_h}(\boldsymbol{K},\boldsymbol{v}):=\sum_{\mathrm{E}\in\mathcal{T}_{h}}d_{K_h}^{\mathrm{E}}(\boldsymbol{K},\boldsymbol{v})\quad
\forall\boldsymbol{K}\in\boldsymbol{S}_{h},\boldsymbol{v}\in\boldsymbol{V}_{h}.
\end{equation}
By using inverse estimate for polynomials, continuity of $\mathrm{\Pi}_{k_1}^{0,\mathrm{E}}$ with respect to $\|\cdot\|_{0,\mathrm{E}}$, the H\"{o}lder inequality and definition of $h_{\mathrm{E}}$, we briefly infer that the following a conclusion
\begin{equation}
\label{norm:L^6-continous}
\|\mathrm{\Pi}_{k_1}^{0,\mathrm{E}}\boldsymbol{v}\|_{0,6,\mathrm{E}}\leq Ch_{\mathrm{E}}^{-2/3}\|\mathrm{\Pi}_{k_1}^{0,\mathrm{E}}\boldsymbol{v}\|_{0,\mathrm{E}}
\leq Ch_{\mathrm{E}}^{-2/3}\|1\|_{0,3,\mathrm{E}}\|\boldsymbol{v}\|_{0,6,\mathrm{E}}
\leq  C\|\boldsymbol{v}\|_{0,6,\mathrm{E}},
\end{equation}
for all $\boldsymbol{v}\in\boldsymbol{V}_{h}^{k_1}(\mathrm{E})$.
Based on the above result, the symmetry of $\mathrm{\Pi}_{k_2}^{0,\mathrm{E}}$ with respect to $\|\cdot\|_{0,\mathrm{E}}$, the H\"{o}lder inequality  (for sequences), and the definiton of $d_{K_h}^{\mathrm{E}}(\cdot,\cdot)$, we notice that
\begin{equation}
d_{K_h}^{\mathrm{E}}(\boldsymbol{K},\boldsymbol{v})\leq Sc\|\mathrm{\Pi}_{k_2}^{0,\mathrm{E}}\boldsymbol{K}\|_{0,\mathrm{E}}\|\boldsymbol{B}\|_{0,3,\mathrm{E}}
\|\mathrm{\Pi}_{k_1}^{0,\mathrm{E}}\boldsymbol{v}\|_{0,6,\mathrm{E}}
\leq C\|\boldsymbol{K}\|_{0,\mathrm{E}}\|\boldsymbol{B}\|_{0,3,\mathrm{E}}
\|\boldsymbol{v}\|_{0,6,\mathrm{E}}
\label{continue-E-d_K}
\end{equation}
and
\begin{equation}
d_{K_h}(\boldsymbol{K},\boldsymbol{v})\leq C\|\boldsymbol{K}\|_{0,\mathrm{\Omega}}\|\boldsymbol{B}\|_{0,3,\mathrm{\Omega}}
\|\boldsymbol{v}\|_{0,6,\mathrm{\Omega}},
\label{continue-Omega-d_K}
\end{equation}
for all $\boldsymbol{K}\in\boldsymbol{S}_{h},\boldsymbol{v}\in\boldsymbol{V}_{h}$ .

\subsection{The construction of discrete load forms\label{subsec:load forms}}
In this subsection, we will construct computable approximation of the right-hand side $(\boldsymbol{f},\boldsymbol{v})$ and $(\boldsymbol{g},\boldsymbol{K})$ in (\ref{variational:continous}), respectively. Firstly, we define the approximated load term $\boldsymbol{f}_{h}$ as
\begin{equation}
\boldsymbol{f}_{h}:=\mathrm{\Pi}_{k_1}^{0,\mathrm{E}}\boldsymbol{f}\quad \forall\mathrm{E}\in\mathcal{T}_{h}
\end{equation}
and consider
\begin{equation}
\label{load term:f}
(\boldsymbol{f}_{h},\boldsymbol{v}_{h})=\sum_{\mathrm{E}\in\mathcal{T}_h}(\boldsymbol{f}_{h},\boldsymbol{v}_{h})_{\mathrm{E}}
=\sum_{\mathrm{E}\in\mathcal{T}_h}(\boldsymbol{f},\mathrm{\Pi}_{k_1}^{0,\mathrm{E}}\boldsymbol{v})_{\mathrm{E}}\quad \forall\boldsymbol{v}\in\boldsymbol{V}_{h},
\end{equation}
we note that the (\ref{load term:f}) is computable by observing Proposition \ref{Projection:L^2k1}.
Secondly, we define the approximated load term $\boldsymbol{g}_{h}$ as
\begin{equation}
\label{load term:Pig}
\boldsymbol{g}_{h}:=\mathrm{\Pi}_{k_2}^{0,\mathrm{E}}\boldsymbol{g}\quad \forall\mathrm{E}\in\mathcal{T}_{h}
\end{equation}
and consider
\begin{equation}
\label{load term:g}
(\boldsymbol{g}_{h},\boldsymbol{K})=\sum_{\mathrm{E}\in\mathcal{T}_h}(\boldsymbol{g}_{h},\boldsymbol{K})_{\mathrm{E}}
=\sum_{\mathrm{E}\in\mathcal{T}_h}(\boldsymbol{g},\mathrm{\Pi}_{k_2}^{0,\mathrm{E}}\boldsymbol{K})_{\mathrm{E}}\quad \forall\boldsymbol{K}\in\boldsymbol{S}_{h},
\end{equation}
we note that the (\ref{load term:g}) is computable by observing Proposition \ref{Projection:L^2k2}.

\section{Analysis of well-posed and convergence for discrete problem\label{sec:Analysis}}
\subsection{The discrete formulate of inductionless MHD}
Referring to (\ref{Finite space:V_h^k(Omega)}),
(\ref{Finite space:Q_h}),
(\ref{Finite space:S_h^k(Omega)}) and
(\ref{Finite space: Psi}), we consider the virtual element
 discretizations problem: Find $(\boldsymbol{u}_h,p_h,\boldsymbol{J}_h,\phi_h)\in\boldsymbol{V}_h\times Q_h\times\boldsymbol{S}_h\times\Psi_h$ such that
\begin{subequations}
	\begin{align}
a_{v_h}(\boldsymbol{u}_h,\boldsymbol{v}_h)+c_{v_h}(\boldsymbol{u}_h,\boldsymbol{u}_h,\boldsymbol{v}_h)
-b_{q}(p_h,\boldsymbol{v}_h)-d_{K_h}(\boldsymbol{J}_h,\boldsymbol{v}_h)&=(\boldsymbol{f}_h,\boldsymbol{v}_h)_{\mathrm{\Omega}},\label{variational:discretizations-u}\\
b_{q}(q_h,\boldsymbol{u}_h)&=0,\label{variational:discretizations-p}\\
a_{K_h}(\boldsymbol{J}_h,\boldsymbol{K}_h)-b_{\psi}(\phi_h,\boldsymbol{K}_h)+d_{K_h}(\boldsymbol{K}_h,\boldsymbol{u}_h)&=S_c(\boldsymbol{g}_h,\boldsymbol{K}_h)_{\mathrm{\Omega}},\label{variational:discretizations-J}\\
b_{\psi}(\psi_{h},\boldsymbol{J}_h)&=0,\label{variational:discretizations-phi}
\end{align}
\label{variational:discretizations}
\end{subequations}
for all $(\boldsymbol{v}_h,q_h,\boldsymbol{K}_h,\psi_h)\in\boldsymbol{V}_h\times Q_h\times\boldsymbol{S}_h\times\Psi_h$. Regarding the constructed the discrete bilinear forms $a_{v_h}(\cdot,\cdot)$ and $a_{K_h}(\cdot,\cdot)$ (see (\ref{stability-uh})-(\ref{stabilizing form-u}) and (\ref{stability form -J})) are (uniformly) stable with respect to the $|\cdot|_{1,\mathrm{\Omega}}$ and $\|\cdot\|_{\textup{div},\mathrm{\Omega}}$-norms, respectively.

It is remark that we can clearly recognize the discrete velocity $\boldsymbol{u}_h\in\boldsymbol{V}_h$ is exactly divergence-free
by observing the equation (\ref{variational:discretizations-p}) with property (\ref{Vh-Qh}). Similarly, the discrete current density $\boldsymbol{J}_h\in\boldsymbol{S}_h$ also exactly divergence-free
by observing the equation (\ref{variational:discretizations-phi}) with property (\ref{Sh-Pish}).

Let us define two discrete kernel spaces
\begin{align*}
\boldsymbol{Z}_{h}&=\Big\{\boldsymbol{v}_{h}\in \boldsymbol{V}_{h}: b_{q}(q_{h},\boldsymbol{v}_{h})=0\quad \forall q_{h}\in Q_{h}\Big\}
=\Big\{\boldsymbol{v}_{h}\in \boldsymbol{V}_{h}: \nabla\cdot\boldsymbol{v}_{h}=0\Big\},\\
 \boldsymbol{Y}_{h}&=\Big\{\boldsymbol{K}_{h}\in \boldsymbol{S}_{h}: b_{\psi}(\psi_{h},\boldsymbol{K}_{h})=0\quad \forall \psi_{h}\in \Psi_{h}\Big\}
 =\Big\{\boldsymbol{K}_{h}\in \boldsymbol{S}_{h}: \nabla\cdot\boldsymbol{K}_{h}=0\Big\},
\end{align*}
we can naturally check that $\boldsymbol{Z}_{h}\subset\boldsymbol{Z}$ and $\boldsymbol{Y}_{h}\subset\boldsymbol{Y}$.

 As a direct consequence of Proposition 4.3 in \cite{da2017divergence}, we have the following stability result.

\begin{prop}
\label{stability result1}
The discrete bilinear forms $b_{q}(\cdot,\cdot)$ satisfies inf-sup condition:
\begin{equation}
\label{inf-sup-uh}
\tilde{\beta_{1}}\|q_h\|_{0,\mathrm{\Omega}}\leq\sup_{\boldsymbol{0}\neq\boldsymbol{v}_h\in\boldsymbol{V}_h}\frac{b_{q}(q_h,\boldsymbol{v}_h)}{|\boldsymbol{v}_h|_{1,\mathrm{\Omega}}}\quad\forall q_h\in Q_h,
\end{equation}
where $\tilde{\beta_{1}}$ is a strictly positive constant independent of $h$.
\end{prop}
Meanwhile, for the given $\boldsymbol{v}\in \boldsymbol{Z}$, the discrete inf-sup condition (\ref{inf-sup-uh}) implies
\begin{equation}
\nonumber
\inf_{0\neq\boldsymbol{v}_{h}\in\boldsymbol{Z}_{h}}|\boldsymbol{v}-\boldsymbol{v}_{h}|_{1,\mathrm{\Omega}}
\leq C\inf_{0\neq\boldsymbol{w}_{h}\in\boldsymbol{V}_{h}}|\boldsymbol{v}-\boldsymbol{w}_{h}|_{1,\mathrm{\Omega}}.
\end{equation}
This means that the approximate accuracy of $\boldsymbol{Z}_{h}$ to $\boldsymbol{Z}$ is equal to the approximate accuracy of $\boldsymbol{V}_{h}$ to $\boldsymbol{Z}$. Notably, we suppose $\boldsymbol{v}\in\boldsymbol{H}^{r+1}\cap\boldsymbol{Z}$, $0<r\leq k_{1}$ and base on Lemma \ref{lem:interplate-u}, we can deduce that
\begin{equation}
\inf_{0\neq\boldsymbol{v}_{h}\in\boldsymbol{Z}_{h}}|\boldsymbol{v}-\boldsymbol{v}_{h}|_{1,\mathrm{\Omega}}\leq Ch^{r}|\boldsymbol{v}|_{r+1,\mathrm{\Omega}}.
\label{inf:u-Zh}
\end{equation}

Next, we can analogize the proof process of Corollary 5.5 in \cite{brezzi2014basic} to obtain following result.
\begin{prop}
\label{stability result2}
The discrete bilinear forms $b_{\psi}(\cdot,\cdot)$ satisfies inf-sup condition:
\begin{equation}
\label{inf-sup-Jh}
\tilde{\beta_{2}}\|\psi_h\|_{0,\mathrm{\Omega}}\leq\sup_{\boldsymbol{0}\neq\boldsymbol{K}_h\in\boldsymbol{S}_h}\frac{b_{\psi}(\psi_h,\boldsymbol{K}_h)}{\|\boldsymbol{K}_h\|_{\textup{div},\mathrm{\Omega}}}\quad\forall \psi_h\in \Psi_h,
\end{equation}
where $\tilde{\beta_{2}}$ is a strictly positive constant independent of $h$.
\end{prop}

In addition, for the given $\boldsymbol{K}\in \boldsymbol{Y}$, the discrete inf-sup condition (\ref{inf-sup-Jh}) implies
\begin{equation}
\nonumber
\inf_{0\neq\boldsymbol{K}_{h}\in\boldsymbol{Y}_{h}}\|\boldsymbol{K}-\boldsymbol{K}_{h}\|_{\tau,\mathrm{\Omega}}
\leq C\inf_{0\neq\boldsymbol{M}_{h}\in\boldsymbol{S}_{h}}\|\boldsymbol{K}-\boldsymbol{M}_{h}\|_{\tau,\mathrm{\Omega}}.
\end{equation}
This means that the approximate accuracy of $\boldsymbol{Y}_{h}$ to $\boldsymbol{Y}$ is equal to the approximate accuracy of $\boldsymbol{S}_{h}$ to $\boldsymbol{Y}$. Notably, we suppose $\boldsymbol{K}\in\boldsymbol{H}^{\tau}\cap\boldsymbol{Y}$, $\frac{1}{2}<\tau\leq k_{2}+1$ and base on Lemma \ref{lem:interplate-K}, we can deduce that
\begin{equation}
\inf_{0\neq\boldsymbol{K}_{h}\in\boldsymbol{Y}_{h}}\|\boldsymbol{K}-\boldsymbol{K}_{h}\|_{0,\mathrm{\Omega}}\leq Ch^{\tau}\|\boldsymbol{K}\|_{\tau,\mathrm{\Omega}}.
\label{inf:J-Yh}
\end{equation}

As we have the coercivity properties of $a_{v_h}(\cdot,\cdot)$ and $a_{K_h}(\cdot,\cdot)$, the skew-symmetry of $c_{v_h}(\cdot,\cdot,\cdot)$, and discrete inf-sup conditions (\ref{inf-sup-uh}), (\ref{inf-sup-Jh}), the well-posedness of virtual element problem (\ref{variational:discretizations}) can be readily obtained. Then we give the theorem as follows:

\begin{thm}
\label{lem:discrete problem well-posed}
Suppose $\mu:=\delta\check{C}_{\min}^{-2}\|\mathrm{F}_h\|_{*}<1$ with $\check{C}_{\min}=\min\{\alpha^{*}\nu,\eta^{*}S_c\}$, then problem (\ref{variational:discretizations}) has a unique solution $(\boldsymbol{u}_h,\boldsymbol{J}_h,p_h,\phi_h)\in \boldsymbol{V}_h\times\boldsymbol{S}_h\times Q_{h}\times \Psi_{h}$ and satisfy the stability inequality
\begin{equation*}
\|(\boldsymbol{u}_h,\boldsymbol{J}_h)\|_{1,\mathrm{\Omega}}\leq \frac{\|\mathrm{F}_{h}\|_{*}}{\check{C}_{\min}}.
\end{equation*}
\end{thm}
\begin{proof}
First, we prove the existence of the solution by applying the Banach fixed point theorem. Given $(\boldsymbol{u}_{h},\boldsymbol{J}_{h})\in\boldsymbol{Z}_{h}\times\boldsymbol{Y}_{h}$,
we consider the saddle problem that finding $(\boldsymbol{w}_{h},p_h,\boldsymbol{M}_{h},\phi_h)\in\boldsymbol{V}_{h}\times
Q_{h}\times\boldsymbol{S}_{h}\times\Psi_{h}$ such that
\begin{subequations}
	\begin{align}
a_{v_h}(\boldsymbol{w}_h,\boldsymbol{v}_h)+c_{v_h}(\boldsymbol{u}_h,\boldsymbol{w}_h,\boldsymbol{v}_h)
-b_{q}(p_h,\boldsymbol{v}_h)-d_{K_h}(\boldsymbol{M}_h,\boldsymbol{v}_h)&=(\boldsymbol{f}_{h},\boldsymbol{v}_h)_{\mathrm{\Omega}},\\
b_{q}(q_h,\boldsymbol{w}_h)&=0,\\
a_{K_h}(\boldsymbol{M}_h,\boldsymbol{K}_h)-b_{\psi}(\phi_h,\boldsymbol{K}_h)+d_{K_h}(\boldsymbol{K}_h,\boldsymbol{w}_h)&=S_c(\boldsymbol{g}_h,\boldsymbol{K}_h)_{\mathrm{\Omega}},\\
b_{\psi}(\psi_{h},\boldsymbol{M}_h)&=0,
\end{align}
\label{variational:discretizations-Banach}
\end{subequations}
for all $(\boldsymbol{v}_{h},q_h,\boldsymbol{K}_{h},\psi_h)\in\boldsymbol{V}_{h}\times
Q_{h}\times\boldsymbol{S}_{h}\times\Psi_{h}$.
From the saddle theory in \cite{glowinski1992finite}, the problem has a unique solution. Let $(\boldsymbol{v}_{h},q_h,\boldsymbol{K}_{h},\psi_h) = (\boldsymbol{w}_{h},p_h,\boldsymbol{M}_{h},\phi_h)$, applying (\ref{stability-uh}), (\ref{stability-Jh}) and (\ref{skew-symmetric-uh}), we obtain the estimate

\begin{equation}
\Check{C}_{\min}\|(\boldsymbol{w}_{h},\boldsymbol{M}_{h})\|_{1,\mathrm{\Omega}}\leq \|\mathrm{F}_h\|_{*},
\label{wM_leq_F}
\end{equation}
where $\Check{C}_{\min} = \min\{\alpha^{*}\nu,\eta^{*}S_c\}$ and $\|\mathrm{F}_h\|_{*}=(\|\boldsymbol{f}_h\|_{-1,\mathrm{\Omega}}+S_c\|\boldsymbol{g}_h\|_{0,\mathrm{\Omega}})^{\frac{1}{2}}$.
Inspired by this, we set $\mathcal{B}_{R} = \big\{(\boldsymbol{v}_h,\boldsymbol{K}_h)\in\boldsymbol{Z}_h\times\boldsymbol{Y}_h:\|(\boldsymbol{v}_h,\boldsymbol{K}_h)\|_{1,\mathrm{\Omega}}\leq R\big\}$ with $R = \|\mathrm{F}\|_{*}/\Check{C}_{\min}$ and consider a map $G:\mathcal{B}_{R}\rightarrow\mathcal{B}_{R}$, $(\boldsymbol{u}_h,\boldsymbol{J}_h)\mapsto(\boldsymbol{w}_h,\boldsymbol{M}_h)$.

Let $(\boldsymbol{u}_{h}^{i},\boldsymbol{J}_{h}^{i})\in \mathcal{B}_R$ and set $(\boldsymbol{w}_{h}^{i},\boldsymbol{M}_{h}^{i})=G(\boldsymbol{u}_{h}^{i},\boldsymbol{J}_{h}^{i})$,
$i = 1,\;2$. Due to definition of (\ref{variational:discretizations-Banach}), where exist $(p_{h}^{i},\phi_{h}^{i})\in Q_{h}\times S_{h}$ such that $(\boldsymbol{w}_{h}^{i},p_{h}^{i},\boldsymbol{M}_{h}^{i},\phi_{h}^{i})$ satisfy the equations
\begin{subequations}
	\begin{align}
a_{v_h}(\boldsymbol{w}_h^{i},\boldsymbol{v}_h)+c_{v_h}(\boldsymbol{u}_h^{i},\boldsymbol{w}_h^{i},\boldsymbol{v}_h)
-b_{q}(p_h^{i},\boldsymbol{v}_h)-d_{K_h}(\boldsymbol{M}_h^{i},\boldsymbol{v}_h)&=(\boldsymbol{f}_{h},\boldsymbol{v}_h)_{\mathrm{\Omega}},\\
b_{q}(q_h,\boldsymbol{w}_h^{i})&=0,\\
a_{K_h}(\boldsymbol{M}_h^{i},\boldsymbol{K}_h)-b_{\psi}(\phi_h^{i},\boldsymbol{K}_h)+d_{K_h}(\boldsymbol{K}_h,\boldsymbol{w}_h^{i})&=S_c(\boldsymbol{g}_h,\boldsymbol{K}_h)_{\mathrm{\Omega}},\\
b_{\psi}(\psi_{h},\boldsymbol{M}_h^{i})&=0,
\end{align}
\label{variational:discretizations-Banach_iwM}
\end{subequations}
for all $(\boldsymbol{v}_{h},q_h,\boldsymbol{K}_{h},\psi_h)\in\boldsymbol{V}_{h}\times
Q_{h}\times\boldsymbol{S}_{h}\times\Psi_{h}$.
Subtracting (\ref{variational:discretizations-Banach_iwM}) as $i = 2$ from (\ref{variational:discretizations-Banach_iwM}) as $i = 1$  and setting $(q_h,\psi_h) = (p_h^1-p_h^2,\phi_h^1-\phi_h^2)$, $(\boldsymbol{v}_h,\boldsymbol{K}_h) =
(\boldsymbol{w}_h^1-\boldsymbol{w}_h^2,\boldsymbol{M}_h^1-\boldsymbol{M}_h^2)$, we deduce
\begin{align}
a_{v_h}(\boldsymbol{w}_h^1-\boldsymbol{w}_h^2,\boldsymbol{w}_h^1-\boldsymbol{w}_h^2)
+a_{K_h}(\boldsymbol{M}_h^1-\boldsymbol{M}_h^2,\boldsymbol{M}_h^1-\boldsymbol{M}_h^2)
= -c_{v_h}(\boldsymbol{u}_h^1-\boldsymbol{u}_h^2,\boldsymbol{w}_h^{2},\boldsymbol{w}_h^1-\boldsymbol{w}_h^2).
\end{align}
Obviously, by using (\ref{stability-uh}), (\ref{stability-Jh}), (\ref{prop:bilinear-continuous-uh}) and (\ref{wM_leq_F}), we have
\begin{align}
\Check{C}_{\min}|(\boldsymbol{w}_h^1-\boldsymbol{w}_h^2,\boldsymbol{M}_h^1-\boldsymbol{M}_h^2)|_{1,\mathrm{\Omega}}\leq \delta|\boldsymbol{u}_h^1-\boldsymbol{u}_h^2|_{1,\mathrm{\Omega}}|\boldsymbol{w}_h^2|_{1,\mathrm{\Omega}}
\leq \delta|\boldsymbol{u}_h^1-\boldsymbol{u}_h^2|_{1,\mathrm{\Omega}}\frac{\|\mathrm{F}_h\|_{*}}{\Check{C}_{\min}},
\end{align}
which equivalent to
\begin{align}
\nonumber
|G(\boldsymbol{u}_h^1,\boldsymbol{J}_h^1)-G(\boldsymbol{u}_h^2,\boldsymbol{J}_h^2)|_{1,\mathrm{\Omega}}
=|(\boldsymbol{w}_h^1-\boldsymbol{w}_h^2,\boldsymbol{M}_h^1-\boldsymbol{M}_h^2)|_{1,\mathrm{\Omega}}
\leq \mu|(\boldsymbol{u}_h^1-\boldsymbol{u}_h^2,\boldsymbol{J}_h^1-\boldsymbol{J}_h^2)|_{1,\mathrm{\Omega}},
\end{align}
where $\mu = \delta\Check{C}_{\min}^{-2}\|\mathrm{F}_h\|_{*}$.
Supposing $\mu = \delta\Check{C}_{\min}^{-2}\|\mathrm{F}_h\|_{*}<1$, then $G$ is a contraction mapping on $\mathcal{B}_R\rightarrow\mathcal{B}_R$. As a consequence, an application of the Banach fixed point theorem shows that $G$ has a fixed point in $\mathcal{B}_{R}$, which is the solution of problem (\ref{variational:discretizations-Banach}).

Next, the uniqueness of the solutions will be proved. Suppose $(\boldsymbol{u}_h^i,p_{h}^i,\boldsymbol{J}_h^i,\phi_h^{i})\in\boldsymbol{V}_{h}\times
Q_{h}\times\boldsymbol{S}_{h}\times\Psi_{h}$ such that
\begin{subequations}
	\begin{align}
a_{v_h}(\boldsymbol{u}_h^{i},\boldsymbol{v}_h)+c_{v_h}(\boldsymbol{u}_h^{i},\boldsymbol{u}_h^{i},\boldsymbol{v}_h)
-b_{q}(p_h^{i},\boldsymbol{v}_h)-d_{K_h}(\boldsymbol{J}_h^{i},\boldsymbol{v}_h)&=(\boldsymbol{f}_{h},\boldsymbol{v}_h)_{\mathrm{\Omega}},\\
b_{q}(q_h,\boldsymbol{u}_h^{i})&=0,\\
a_{K_h}(\boldsymbol{J}_h^{i},\boldsymbol{K}_h)-b_{\psi}(\phi_h^{i},\boldsymbol{K}_h)+d_{K_h}(\boldsymbol{K}_h,\boldsymbol{u}_h^{i})&=S_c(\boldsymbol{g}_h,\boldsymbol{K}_h)_{\mathrm{\Omega}},\\
b_{\psi}(\psi_{h},\boldsymbol{J}_h^{i})&=0,
\end{align}
\label{variational:discretizations-Banach_ui}
\end{subequations}
for all $(\boldsymbol{v}_{h},q_h,\boldsymbol{K}_{h},\psi_h)\in\boldsymbol{V}_{h}\times
Q_{h}\times\boldsymbol{S}_{h}\times\Psi_{h}$. Subtracting (\ref{variational:discretizations-Banach_ui}) as $i=2$ from (\ref{variational:discretizations-Banach_ui}) as $i=1$ yields
\begin{subequations}
	\begin{align}
\nonumber
a_{v_h}(\boldsymbol{u}_h^{1}-\boldsymbol{u}_h^{2},\boldsymbol{v}_h)+c_{v_h}(\boldsymbol{u}_h^{1},\boldsymbol{u}_h^{1},\boldsymbol{v}_h)
-c_{v_h}(\boldsymbol{u}_h^{2},\boldsymbol{u}_h^{2},\boldsymbol{v}_h)
\\-b_{q}(p_h^{1}-p_h^{2},\boldsymbol{v}_h)-d_{K_h}(\boldsymbol{J}_h^{1}-\boldsymbol{J}_h^{2},\boldsymbol{v}_h)&=0,\\
b_{q}(q_h,\boldsymbol{u}_h^{1}-\boldsymbol{u}_h^{2})&=0,\\
a_{K_h}(\boldsymbol{J}_h^{1}-\boldsymbol{J}_h^{2},\boldsymbol{K}_h)-b_{\psi}(\phi_h^{1}-\phi_h^{2},\boldsymbol{K}_h)+d_{K_h}(\boldsymbol{K}_h,\boldsymbol{u}_h^{1}-\boldsymbol{u}_h^{2})&=0,\\
b_{\psi}(\psi_h,\boldsymbol{J}_h^{1}-\boldsymbol{J}_h^{2})&=0.
\end{align}
\label{variational:discretizations-Banach_ui12}
\end{subequations}
Then, we take $(\boldsymbol{v}_{h},\boldsymbol{K}_{h})=(\boldsymbol{u}_{h}^{1}-\boldsymbol{u}_{h}^{2},\boldsymbol{J}_{h}^{1}-\boldsymbol{J}_{h}^{2})$, $(q_{h},\psi_{h})=(p_h^{1}-p_h^{2},\phi_h^{1}-\phi_h^{2})$ in (\ref{variational:discretizations-Banach_ui12}), we can obtain
\begin{align}
\nonumber
a_{v_h}(\boldsymbol{u}_h^{1}-\boldsymbol{u}_h^{2},\boldsymbol{u}_h^{1}-\boldsymbol{u}_h^{2})+
a_{K_h}(\boldsymbol{J}_h^{1}-\boldsymbol{J}_h^{2},\boldsymbol{J}_h^{1}-\boldsymbol{J}_h^{2})=
-c_{v_h}(\boldsymbol{u}_h^{1}-\boldsymbol{u}_h^{2},\boldsymbol{u}_h^{2},\boldsymbol{u}_h^{1}-\boldsymbol{u}_h^{2}).
\end{align}
Using (\ref{stability-uh}), (\ref{stability-Jh}), (\ref{prop:bilinear-continuous-uh}) and (\ref{wM_leq_F}) again, we have
\begin{align}
\nonumber
|(\boldsymbol{u}_h^1-\boldsymbol{u}_h^2,\boldsymbol{J}_h^1-\boldsymbol{J}_h^2)|_{1,\mathrm{\Omega}}
\leq \mu|(\boldsymbol{u}_h^1-\boldsymbol{u}_h^2,\boldsymbol{J}_h^1-\boldsymbol{J}_h^2)|_{1,\mathrm{\Omega}},
\end{align}
Due to $\mu<1$, we can infer
\begin{align}
\boldsymbol{u}_{h}^1 = \boldsymbol{u}_{h}^2,\quad \boldsymbol{J}_{h}^1 = \boldsymbol{J}_{h}^2.
\label{u1_u2,J1_J2}
\end{align}
Taking (\ref{u1_u2,J1_J2}) into (\ref{variational:discretizations-Banach_ui12}), we clearly see
\begin{align*}
\nonumber
b_{q}(p_h^{1}-p_h^{2},\boldsymbol{v}_h)=0 \quad\forall\boldsymbol{v}\in\boldsymbol{V}_{h},\quad
b_{\psi}(\phi_h^{1}-\phi_h^{2},\boldsymbol{K}_h)=0\quad\forall\boldsymbol{K}\in\boldsymbol{S}_{h}.
\end{align*}
Finally, combining the discrete inf-sup conditions (\ref{inf-sup-uh}) and (\ref{inf-sup-Jh}) to get $p_{h}^{1}=p_{h}^{2}$, $\phi_h^{1} = \phi_h^{2}$. This completes the proof.
\end{proof}
At the end of this subsection, we formulate the equivalent kernel form of problem (\ref{variational:discretizations}) as: Find $(\boldsymbol{u}_h$, $p_h$, $\boldsymbol{J}_h$, $
\phi_h)\in\boldsymbol{Z}_h\times Q_h\times\boldsymbol{Y}_h\times\Psi_h$ such that
\begin{subequations}
	\begin{align}
a_{v_h}(\boldsymbol{u}_h,\boldsymbol{v}_h)+c_{v_h}(\boldsymbol{u}_h,\boldsymbol{u}_h,\boldsymbol{v}_h)
-d_{K_h}(\boldsymbol{J}_h,\boldsymbol{v}_h)&=(\boldsymbol{f}_h,\boldsymbol{v}_h)_{\mathrm{\Omega}},\label{variational:kernel-discretizations-u}\\
a_{K_h}(\boldsymbol{J}_h,\boldsymbol{K}_h)+d_{K_h}(\boldsymbol{K}_h,\boldsymbol{u}_h)&=S_c(\boldsymbol{g}_h,\boldsymbol{K}_h)_{\mathrm{\Omega}},\label{variational:kernel-discretizations-J}
\end{align}
\label{variational:kernel-discretizations}
\end{subequations}
for all $(\boldsymbol{v}_h,q_h,\boldsymbol{K}_h,\psi_h)\in\boldsymbol{Z}_h\times Q_h\times\boldsymbol{Y}_h\times\Psi_h$. Now, we notice that if $(\boldsymbol{u}_h,\boldsymbol{J}_h)\in\boldsymbol{V}_h\times\boldsymbol{S}_h$ is the solution to problem $(\ref{variational:discretizations})$, then it is also the solution to problem $(\ref{variational:kernel-discretizations})$ (see the definition of two kernel spaces $\boldsymbol{Z}_h$ and $\boldsymbol{Y}_h$).

\subsection{Preliminary estimates}
%In this subsection, we recall the following some interpolation estimates results for enhanced spaces $\boldsymbol{V}_{h}$ and $\boldsymbol{S}_{h}$ and
%some classical approximation results for $\mathcal{P}_{k}$ polynomials on star-shaped domains.

On the one hand, we provide the following some interpolation error estimation results.

\begin{lem}[\cite{vacca2018h,gatica2018mixed}]
Let $\boldsymbol{w}\in\boldsymbol{H}^{1+r}\cap\boldsymbol{V}$ for $0<r\leq k_{1}$, then there exists $\boldsymbol{w}_{I}\in\boldsymbol{V}_{h}$ such that
\begin{equation}
\nonumber
\|\boldsymbol{w}-\boldsymbol{w}_{I}\|_{0,\mathrm{\Omega}}+h|\boldsymbol{w}-\boldsymbol{w}_{I}|_{1,\mathrm{\Omega}}\leq Ch^{r+1}|\boldsymbol{w}|_{r+1,\mathrm{\Omega}}.
\end{equation}
\label{lem:interplate-u}
\end{lem}

\begin{lem}[\cite{beirao2016mixed,caceres2017mixed}]
Let $\boldsymbol{K}\in\boldsymbol{H}^{\tau}\cap\boldsymbol{S}$ for $\frac{1}{2}\leq\tau\leq k_{2}+1$, then there exists $\boldsymbol{K}_{I}\in\boldsymbol{S}_{h}$ such that
\begin{equation}
\nonumber
\|\boldsymbol{K}-\boldsymbol{K}_{I}\|_{0,\mathrm{\Omega}}\leq Ch^{\tau}|\boldsymbol{K}|_{\tau,\mathrm{\Omega}}.
\end{equation}
\label{lem:interplate-K}
\end{lem}
Besides that, we remark that using classical theory and setting $q\in H^{r}$ for $0<r\leq k_{1}$ and
$ \psi\in H^{\tau}$ for $0<r\leq k_{2}+1$, there exist $q_{h}\in Q_{h}$ and $\psi_{h}\in\Psi_{h}$ such that
\begin{equation}
\|q-q_{h}\|_{0,\mathrm{\Omega}}\leq Ch^{r}|q|_{r,\mathrm{\Omega}}
\label{interplate-qh}
\end{equation}
and
\begin{equation}
\|\psi-\psi_{h}\|_{0,\mathrm{\Omega}}\leq Ch^{\tau}|\psi|_{\tau,\mathrm{\Omega}}.
\label{interplate-psih}
\end{equation}

On the other hand, we provide some classical approximation results for $[\mathcal{P}_{k}(\mathrm{E})]^2$ polynomials on star-shaped domains $\mathrm{E}$.

\begin{lem}[\cite{brenner2008mathematical}]
Let $\mathrm{E}\in\mathcal{T}_{h}$, while we set two real numbers $r$, $m$ with $0\leq r\leq k_{1}$ and $1\leq m\leq\infty$. Then there exists a polynomial function $\boldsymbol{v}_{\mathrm{\pi}}\in[\mathcal{P}_{k_1}(\mathrm{E})]^2$ such that
\begin{equation}
\nonumber
\|\boldsymbol{v}-\boldsymbol{v}_{\mathrm{\pi}}\|_{L^{m}(\mathrm{E})}+h_{\mathrm{E}}|\boldsymbol{v}-\boldsymbol{v}_{\mathrm{\pi}}|_{\boldsymbol{W}^{1,m}(\mathrm{E})}\leq Ch_{\mathrm{E}}^{r+1}|\boldsymbol{v}|_{\boldsymbol{W}^{r+1,m}(\mathrm{E})}
\end{equation}
for all $\boldsymbol{v}\in\boldsymbol{H}^{r+1}(\mathrm{E})$.
\label{lem:polymials-uerr}
\end{lem}

\begin{lem}[Section 3.2 of \cite{beirao2016mixed}]
Let $\mathrm{E}\in\mathcal{T}_{h}$, while we set one real number $\tau$ with $0\leq \tau\leq k_{2}+1$. Then there exists a polynomial function $\boldsymbol{K}_{\mathrm{\pi}}\in[\mathcal{P}_{k_2}(\mathrm{E})]^2$ such that
\begin{equation}
\nonumber
\|\boldsymbol{K}-\boldsymbol{K}_{\mathrm{\pi}}\|_{0,\mathrm{E}}\leq Ch_{\mathrm{E}}^{\tau}|\boldsymbol{K}|_{\tau,\mathrm{E}}
\end{equation}
for all $\boldsymbol{K}\in\boldsymbol{H}^{\tau}(\mathrm{E})$.
\label{lem:polymials-Kerr}
\end{lem}

\begin{lem}[\cite{brenner2008mathematical}]
Let $\boldsymbol{v}\in\boldsymbol{H}^{r+1}\cap\boldsymbol{V}$ with $0\leq r\leq k_{1}$. Then it holds that
\begin{equation}
\nonumber
|c_{v}(\boldsymbol{v},\boldsymbol{v},\boldsymbol{w})
-c_{v_h}(\boldsymbol{v},\boldsymbol{v},\boldsymbol{w})|\leq Ch^{r}(\|\boldsymbol{v}\|_{r,\mathrm{\Omega}}+|\boldsymbol{v}|_{1,\mathrm{\Omega}}+\|\boldsymbol{v}\|_{r+1,\mathrm{\Omega}})\|\boldsymbol{v}\|_{r+1,\mathrm{\Omega}}|\boldsymbol{w}|_{1,\mathrm{\Omega}}
\end{equation}
for all $\boldsymbol{w}\in\boldsymbol{V}$.
\label{trilinear:c_v-c_vh}
\end{lem}

\begin{lem}[\cite{brenner2008mathematical}]
Let $\delta$ be the constant defined in (\ref{prop:bilinear-continuous-uh}). Then it holds that
\begin{equation}
\nonumber
|c_{v_h}(\boldsymbol{u},\boldsymbol{u},\boldsymbol{w})
-c_{v_h}(\boldsymbol{v},\boldsymbol{v},\boldsymbol{w})|\leq \delta\big(|\boldsymbol{v}|_{1,\mathrm{\Omega}}|\boldsymbol{w}|_{1,\mathrm{\Omega}}+|\boldsymbol{u}-\boldsymbol{v}+\boldsymbol{w}|_{1,\mathrm{\Omega}}(|\boldsymbol{u}|_{1,\mathrm{\Omega}}+|\boldsymbol{v}|_{1,\mathrm{\Omega}})\big)|\boldsymbol{w}|_{1,\mathrm{\Omega}}
\end{equation}
for all $\boldsymbol{u},\boldsymbol{v},\boldsymbol{w}\in\boldsymbol{V}$.
\label{trilinear:c_vh-c_vh}
\end{lem}

Finally, we note the results regarding the approximation of the load terms.

\begin{lem}[\cite{beirao2013basic,brenner2008mathematical}]
Assuming $\boldsymbol{f}\in\boldsymbol{H}^{r+1}$, $-1\leq r\leq k_{1}$ and setting $\boldsymbol{f}_{h}:=\mathrm{\Pi}_{k_1}^{0,\mathrm{E}}\boldsymbol{f}$. Then we have
\begin{equation}
|(\boldsymbol{f}_{h}-\boldsymbol{f},\boldsymbol{v})_{\mathrm{\Omega}}|\leq Ch^{r+2}|\boldsymbol{f}|_{r+1,\mathrm{\Omega}}|\boldsymbol{v}|_{1,\mathrm{\Omega}}\quad\forall\boldsymbol{v}\in\boldsymbol{V}_{h}.
\end{equation}
\label{load term:f-fh_error}
\end{lem}
\begin{lem}
Assuming $\boldsymbol{g}\in\boldsymbol{H}^{\tau}$, $0\leq \tau\leq k_{2}+1$ and setting $\boldsymbol{g}_{h}:=\mathrm{\Pi}_{k_2}^{0,\mathrm{E}}\boldsymbol{g}$. Then we have
\begin{equation}
|(\boldsymbol{g}_{h}-\boldsymbol{g},\boldsymbol{K})_{\mathrm{\Omega}}|\leq Ch^{\tau}|\boldsymbol{g}|_{\tau,\mathrm{\Omega}}\|\boldsymbol{K}\|_{0,\mathrm{\Omega}}\quad\forall\boldsymbol{K}\in\boldsymbol{S}_{h}.
\end{equation}
\label{load term:g-gh_error}
\end{lem}
\begin{proof}
By observing (\ref{load term:Pig})-(\ref{load term:g}), then the standard $L^2$-orthogonality and approximation estimates on star-shaped domains yield
\begin{equation}
\nonumber
(\boldsymbol{g}_{h},\boldsymbol{K})_{\mathrm{\Omega}}-(\boldsymbol{g},\boldsymbol{K})_{\mathrm{\Omega}} = \sum_{\mathrm{E}\in\mathcal{T}_{h}}(\mathrm{\Pi}_{k_2}^{0,\mathrm{E}}\boldsymbol{g}-\boldsymbol{g},\boldsymbol{K})_{\mathrm{E}}\quad\forall\boldsymbol{K}\in \boldsymbol{S}_{h},
\end{equation}
then applying Lemma \ref{lem:polymials-Kerr}, we infer that
\begin{equation}
\nonumber
(\boldsymbol{g}_{h},\boldsymbol{K})_{\mathrm{\Omega}}-(\boldsymbol{g},\boldsymbol{K})_{\mathrm{\Omega}} \leq C\sum_{\mathrm{E}\in\mathcal{T}_{h}}h_{\mathrm{E}}^{\tau}|\boldsymbol{g}|_{\tau,\mathrm{E}}\|\boldsymbol{K}\|_{0,\mathrm{E}}
\leq Ch^{\tau}\Big(\sum_{\mathrm{E}\in\mathcal{T}_{h}}|\boldsymbol{g}|_{\tau,\mathrm{E}}\Big)^{\frac{1}{2}}\|\boldsymbol{K}\|_{0,\mathrm{\Omega}}.
\end{equation}
\end{proof}

\subsection{Convergence analysis of the model}

\begin{thm}
Under the assumption conditions of Theorem \ref{lem:continue problem well-posed} and \ref{lem:discrete problem well-posed}, we set $(\boldsymbol{u},\boldsymbol{J},p,\phi)$ is the solution of problem \ref{variational:continous} and $(\boldsymbol{u}_h,\boldsymbol{J}_h,p_h,\phi_h)$ is the solution of virtual problem \ref{variational:discretizations}, while assuming $\boldsymbol{u}\in\boldsymbol{H}^{r+1}$, $p\in\boldsymbol{H}^{r}$, $\boldsymbol{J}\in\boldsymbol{H}^{\tau}$ and $\phi\in\boldsymbol{H}^{\tau}$ with $0<r\leq k_1$, $\frac{1}{2}\leq\tau\leq k_2+1$, then we have
\begin{align}
|\boldsymbol{u}-\boldsymbol{u}_h|_{1,\mathrm{\Omega}}+
\|\boldsymbol{J}-\boldsymbol{J}_h\|_{0,\mathrm{\Omega}}&\leq h^{\min\{r,\tau\}}\mathcal{V}_{2}(\|\boldsymbol{u}\|_{r+1,\mathrm{\Omega}},|\boldsymbol{f}|_{r+1,\mathrm{\Omega}},|\boldsymbol{g}|_{\tau,\mathrm{\Omega}},|\boldsymbol{J}|_{\tau,\mathrm{\Omega}}),\label{error:u-uh J-Jh}
\\\|p-p_{h}\|_{0,\mathrm{\Omega}}&\leq h^{\min\{\tau,r\}}\mathcal{V}_{3}(\|\boldsymbol{u}\|_{r+1,\mathrm{\Omega}},|\boldsymbol{f}|_{r+1,\mathrm{\Omega}},|\boldsymbol{g}|_{\tau,\mathrm{\Omega}},|\boldsymbol{J}|_{\tau,\mathrm{\Omega}},|p|_{r,\mathrm{\Omega}}), \label{error:p-ph}
\\\|\phi-\phi_{h}\|_{0,\mathrm{\Omega}}&\leq h^{\min\{\tau,r\}}\mathcal{V}_{4}(\|\boldsymbol{u}\|_{r+1,\mathrm{\Omega}},|\boldsymbol{f}|_{r+1,\mathrm{\Omega}},|\boldsymbol{g}|_{\tau,\mathrm{\Omega}},|\boldsymbol{J}|_{\tau,\mathrm{\Omega}},|\phi|_{\tau,\mathrm{\Omega}}), \label{error:phi-phih}
\end{align}
\label{thm:error}
\end{thm}
where $\mathcal{V}_{i}$ $(i = 2,3,4)$ are suitable functions independent of $h$.
\begin{proof}
Let $(\boldsymbol{u}_{I},\boldsymbol{J}_{I})\in\boldsymbol{Z}_{h}\times\boldsymbol{Y}_{h}$ be an interpolation approximate of  $(\boldsymbol{u},\boldsymbol{J})$ and satisfy (\ref{inf:u-Zh}) and (\ref{inf:J-Yh}), respectively. Two symbols $\mathcal{E}_h^u:=\boldsymbol{u}_h-\boldsymbol{u}_I$ and $\mathcal{E}_h^J:=\boldsymbol{J}_h-\boldsymbol{J}_I$ are introduced,
 and then we add and subtract $(\boldsymbol{u}_{I},\boldsymbol{J}_{I})$ in (\ref{variational:kernel-discretizations}) to find that
\begin{subequations}
	\begin{align}
a_{v_h}(\mathcal{E}_h^u,\boldsymbol{v}_h)-d_{K_h}(\mathcal{E}_h^J,\boldsymbol{v}_h)
=&(\boldsymbol{f}_h,\boldsymbol{v}_h)_{\mathrm{\Omega}} - a_{v_h}(\boldsymbol{u}_I,\boldsymbol{v}_h)
-c_{v_h}(\boldsymbol{u}_h,\boldsymbol{u}_h,\boldsymbol{v}_h)
+d_{K_h}(\boldsymbol{J}_I,\boldsymbol{v}_h),\label{variational:KD-uI}\\
a_{K_h}(\mathcal{E}_h^J,\boldsymbol{K}_h)+d_{K_h}(\boldsymbol{K}_h,\mathcal{E}_h^u)=
&S_c(\boldsymbol{g}_h,\boldsymbol{K}_h)_{\mathrm{\Omega}}-a_{K_h}(\boldsymbol{J}_I,\boldsymbol{K}_h)-
d_{K_h}(\boldsymbol{K}_h,\boldsymbol{u}_I).\label{variational:KD-JI}
\end{align}
\label{variational:KD-uI-JI}
\end{subequations}

Next, we set $\boldsymbol{v}_h=\mathcal{E}_h^u$ in (\ref{variational:KD-uI}) and $\boldsymbol{K}_h=\mathcal{E}_h^J$ in (\ref{variational:KD-JI}), then add the resulting equations together to obtain the following concise forms
\begin{align*}
LHS:= a_{v_h}(\mathcal{E}_h^u,\mathcal{E}_h^u)+a_{K_h}(\mathcal{E}_h^J,\mathcal{E}_h^J)
\end{align*}
and
\begin{align*}
RHS:=\;&(\boldsymbol{f}_h,\mathcal{E}_h^u)_{\mathrm{\Omega}} +S_c(\boldsymbol{g}_h,\mathcal{E}_h^J)_{\mathrm{\Omega}}- a_{v_h}(\boldsymbol{u}_I,\mathcal{E}_h^u)-a_{K_h}(\boldsymbol{J}_I,\mathcal{E}_h^J)
-\\&c_{v_h}(\boldsymbol{u}_h,\boldsymbol{u}_h,\mathcal{E}_h^u)
+d_{K_h}(\boldsymbol{J}_I,\mathcal{E}_h^u)- d_{K_h}(\mathcal{E}_h^J,\boldsymbol{u}_I).
\end{align*}
Subsequently, we observe that $(\mathcal{E}_h^u,\mathcal{E}_h^J)\in\boldsymbol{Z}\times\boldsymbol{Y}$ such that the equation (\ref{variational:kernel-continous}) with
$(\boldsymbol{v},\boldsymbol{J})=(\mathcal{E}_h^u,\mathcal{E}_h^J)$ can be rewritten as
\begin{subequations}
	\begin{align}
a_{v}(\boldsymbol{u},\mathcal{E}_h^u)+c_{v}(\boldsymbol{u},\boldsymbol{u},\mathcal{E}_h^u)
-d_{K}(\boldsymbol{J},\mathcal{E}_h^u)&=(\boldsymbol{f},\mathcal{E}_h^u)_{\mathrm{\Omega}},\label{variational:KD-rewritten with E_h^u}\\
a_{K}(\boldsymbol{J},\mathcal{E}_h^J)+d_{K}(\mathcal{E}_h^J,\boldsymbol{u})&=S_c(\boldsymbol{g},\mathcal{E}_h^J)_{\mathrm{\Omega}},\label{variational:KD-rewritten with E_h^J}
\end{align}
\label{variational:KD-rewritten with E_h^u-E_h^J}
\end{subequations}
we substitute (\ref{variational:KD-rewritten with E_h^u-E_h^J}) in $RHS$ to yield
\begin{align*}
FRHS:=\;&(\boldsymbol{f}_h,\mathcal{E}_h^u)_{\mathrm{\Omega}}-(\boldsymbol{f},\mathcal{E}_h^u)_{\mathrm{\Omega}}+
S_c(\boldsymbol{g}_h,\mathcal{E}_h^J)_{\mathrm{\Omega}}-S_c(\boldsymbol{g},\mathcal{E}_h^J)_{\mathrm{\Omega}}\\&-
a_{v_h}(\boldsymbol{u}_I,\mathcal{E}_h^u)+a_{v}(\boldsymbol{u},\mathcal{E}_h^u)-
a_{K_h}(\boldsymbol{J}_I,\mathcal{E}_h^J)+a_{K}(\boldsymbol{J},\mathcal{E}_h^J)\\&-
c_{v_h}(\boldsymbol{u}_h,\boldsymbol{u}_h,\mathcal{E}_h^u)+c_{v}(\boldsymbol{u},\boldsymbol{u},\mathcal{E}_h^u)+
d_{K_h}(\boldsymbol{J}_I,\mathcal{E}_h^u)-d_{K}(\boldsymbol{J},\mathcal{E}_h^u)-
d_{K_h}(\mathcal{E}_h^J,\boldsymbol{u}_I)+d_{K}(\mathcal{E}_h^J,\boldsymbol{u})
\\=\;&\Big\{(\boldsymbol{f}_h-\boldsymbol{f},\mathcal{E}_h^u)_{\mathrm{\Omega}}
+S_c(\boldsymbol{g}_h-\boldsymbol{g},\mathcal{E}_h^J)_{\mathrm{\Omega}}\Big\}
+\Big\{a_{v}(\boldsymbol{u},\mathcal{E}_h^u)-a_{v_h}(\boldsymbol{u}_I,\mathcal{E}_h^u)\Big\}
\\&+\Big\{a_{K}(\boldsymbol{J},\mathcal{E}_h^J)-a_{K_h}(\boldsymbol{J}_I,\mathcal{E}_h^J)\Big\}
+\Big\{c_{v}(\boldsymbol{u},\boldsymbol{u},\mathcal{E}_h^u)-c_{v_h}(\boldsymbol{u}_h,\boldsymbol{u}_h,\mathcal{E}_h^u)\Big\}
\\&+\Big\{d_{K_h}(\boldsymbol{J}_I,\mathcal{E}_h^u)-d_{K}(\boldsymbol{J},\mathcal{E}_h^u)\Big\}
+\Big\{d_{K}(\mathcal{E}_h^J,\boldsymbol{u})-d_{K_h}(\mathcal{E}_h^J,\boldsymbol{u}_I)\Big\}\\
:=\;&\mathcal{N}_1+\mathcal{N}_2+\mathcal{N}_3+\mathcal{N}_4+\mathcal{N}_5+\mathcal{N}_6.
\end{align*}

Now, we separately estimate these six items $\mathcal{N}_{i}$ $(i = 1,2,\cdots,6)$.
\begin{itemize}
\setstretch{1.25}
\item[$\blacktriangleright$] The item $\mathcal{N}_1$: From Lemma \ref{load term:f-fh_error}, \ref{load term:g-gh_error}, we easily obtain
\begin{align}
\mathcal{N}_1 \leq C\big(h^{r+2}|\boldsymbol{f}|_{r+1,\mathrm{\Omega}}|\mathcal{E}_h^u|_{1,\mathrm{\Omega}}+h^{\tau}|\boldsymbol{g}|_{\tau,\mathrm{\Omega}}|\mathcal{E}_h^J|_{0,\mathrm{\Omega}}\big).
\label{mathrcal_N1}
\end{align}

\item[$\blacktriangleright$] The item $\mathcal{N}_2$: By adding and subtracting $\boldsymbol{u}_{\pi}$ (the piecewise linear polynomial approximate of $\boldsymbol{u}$), while applying the $k_1$-consistency (\ref{k_1-consistency}), (\ref{continue-Omega-a_v}), Lemma \ref{lem:polymials-uerr} and (\ref{inf:u-Zh}), we find
\begin{align}
\nonumber
\mathcal{N}_2 &=
\sum_{\mathrm{E}\in\mathcal{T}_{h}}\big(a_{v}^{\mathrm{E}}(\boldsymbol{u},\mathcal{E}_h^u)
-a_{v_h}^{\mathrm{E}}(\boldsymbol{u}_I,\mathcal{E}_h^u)\big)\\\nonumber
&=\sum_{\mathrm{E}\in\mathcal{T}_{h}}\big(a_{v}^{\mathrm{E}}(\boldsymbol{u}-\boldsymbol{u}_{\pi},\mathcal{E}_h^u)
-a_{v_h}^{\mathrm{E}}(\boldsymbol{u}_I-\boldsymbol{u}_{\pi},\mathcal{E}_h^u)\big)
\\\nonumber&\leq C \sum_{\mathrm{E}\in\mathcal{T}_{h}}\big(|\boldsymbol{u}-\boldsymbol{u}_{\pi}|_{1,\mathrm{E}}+
|\boldsymbol{u}-\boldsymbol{u}_I|_{1,\mathrm{E}}\big)|\mathcal{E}_h^u|_{1,\mathrm{E}}\\
&\leq Ch^{r}|\boldsymbol{u}|_{r+1,\mathrm{\Omega}}|\mathcal{E}_h^u|_{1,\mathrm{\Omega}}.
\label{mathrcal_N2}
\end{align}

\item[$\blacktriangleright$] The item $\mathcal{N}_3$: By adding and subtracting $\boldsymbol{u}_{\pi}$ (the piecewise linear polynomial approximate of $\boldsymbol{u}$), while applying the $k_2$-consistency (\ref{k_2-consistency}), (\ref{continue-Omega-a_K}), Lemma \ref{lem:polymials-Kerr} and (\ref{inf:J-Yh}), we find

\begin{align}
\nonumber
\mathcal{N}_3 &=
\sum_{\mathrm{E}\in\mathcal{T}_{h}}\big(a_{K}^{\mathrm{E}}(\boldsymbol{J},\mathcal{E}_h^J)-a_{K_h}^{\mathrm{E}}(\boldsymbol{J}_I,\mathcal{E}_h^J)\big)\\\nonumber
&=\sum_{\mathrm{E}\in\mathcal{T}_{h}}\big(a_{K}^{\mathrm{E}}(\boldsymbol{J}-\boldsymbol{J}_{\pi},\mathcal{E}_h^J)
-a_{K_h}^{\mathrm{E}}(\boldsymbol{J}_I-\boldsymbol{J}_{\pi},\mathcal{E}_h^J)
\big)
\\\nonumber&\leq C \sum_{\mathrm{E}\in\mathcal{T}_{h}}\big(\|\boldsymbol{J}-\boldsymbol{J}_{\pi}\|_{0,\mathrm{E}}+
\|\boldsymbol{J}-\boldsymbol{J}_I\|_{0,\mathrm{E}}\big)\|\mathcal{E}_h^J\|_{0,\mathrm{E}}\\
&\leq Ch^{\tau}|\boldsymbol{J}|_{\tau,\mathrm{\Omega}}\|\mathcal{E}_h^J\|_{0,\mathrm{\Omega}}.
\label{mathrcal_N3}
\end{align}

\item[$\blacktriangleright$] The item $\mathcal{N}_4$: Obviously, we clearly observe that
\begin{align}
\mathcal{N}_4 &= \big(c_{v}(\boldsymbol{u},\boldsymbol{u},\mathcal{E}_h^u)-c_{v_h}(\boldsymbol{u},\boldsymbol{u},\mathcal{E}_h^u)\big)
+\big(c_{v_h}(\boldsymbol{u},\boldsymbol{u},\mathcal{E}_h^u)-c_{v_h}(\boldsymbol{u}_h,\boldsymbol{u}_h,\mathcal{E}_h^u)\big),
\label{trilinear term:u-uh-estimate results-equality1}
\end{align}
while taking the $\mathcal{E}_h^u = \boldsymbol{u}_{h}-\boldsymbol{u}_I$, Lemma \ref{trilinear:c_v-c_vh}, \ref{trilinear:c_vh-c_vh}, (\ref{inf:u-Zh}) and into (\ref{trilinear term:u-uh-estimate results-equality1}), then we infer that
\begin{align}
\nonumber
\mathcal{N}_{4} &\leq Ch^{r}\Big(\|\boldsymbol{u}\|_{r,\mathrm{\Omega}}+|\boldsymbol{u}|_{1,\mathrm{\Omega}}+\|\boldsymbol{u}\|_{r+1,\mathrm{\Omega}}\Big)\|\boldsymbol{u}\|_{r+1,\mathrm{\Omega}}|\mathcal{E}_h^u|_{1,\mathrm{\Omega}}
\\\nonumber&\quad+\delta\Big(|\boldsymbol{u}_h|_{1,\mathrm{\Omega}}|\mathcal{E}_h^u|_{1,\mathrm{\Omega}}
+|\boldsymbol{u}-\boldsymbol{u}_I|_{1,\mathrm{\Omega}}(|\boldsymbol{u}|_{1,\mathrm{\Omega}}+|\boldsymbol{u}_h|_{1,\mathrm{\Omega}})\Big)|\mathcal{E}_h^u|_{1,\mathrm{\Omega}}
\\\nonumber\quad&\leq
Ch^{r}(\|\boldsymbol{u}\|_{r,\mathrm{\Omega}}+|\boldsymbol{u}|_{1,\mathrm{\Omega}}+\|\boldsymbol{u}\|_{r+1,\mathrm{\Omega}})\|\boldsymbol{u}\|_{r+1,\mathrm{\Omega}}|\mathcal{E}_h^u|_{1,\mathrm{\Omega}}
\\\nonumber&\quad+\delta\big(|\boldsymbol{u}_h|_{1,\mathrm{\Omega}}|\mathcal{E}_h^u|_{1,\mathrm{\Omega}}
+Ch^{r}\|\boldsymbol{u}\|_{r+1,\mathrm{\Omega}}(|\boldsymbol{u}|_{1,\mathrm{\Omega}}+|\boldsymbol{u}_h|_{1,\mathrm{\Omega}})\big)|\mathcal{E}_h^u|_{1,\mathrm{\Omega}}
\\&\leq Ch^{r}\|\boldsymbol{u}\|_{r+1,\mathrm{\Omega}}\big(\|\boldsymbol{u}\|_{r,\mathrm{\Omega}}+2|\boldsymbol{u}|_{1,\mathrm{\Omega}}+
\|\boldsymbol{u}\|_{r+1,\mathrm{\Omega}}+|\boldsymbol{u}_h|_{1,\mathrm{\Omega}}\big)|\mathcal{E}_h^u|_{1,\mathrm{\Omega}}
+\delta|\boldsymbol{u}_h|_{1,\mathrm{\Omega}}|\mathcal{E}_h^u|_{1,\mathrm{\Omega}}^2.
\label{mathrcal_N4}
\end{align}

\item[$\blacktriangleright$]
The item $\mathcal{N}_5$:
By using Poincar\'{e} type inequality, continuity of $\mathrm{\Pi}_{k_2}^{0,\mathrm{E}}$ with respect to $\|\cdot\|_{0,\mathrm{E}}$, (\ref{continue-E-d_K}), (\ref{inf:J-Yh}) and Lemma  \ref{lem:polymials-Kerr}, we obtain
\begin{align}
\nonumber
\mathcal{N}_{5} &= \sum_{\mathrm{E}\in\mathcal{T}_{h}}\big(d_{K_h}^{\mathrm{E}}(\boldsymbol{J}_I-\boldsymbol{J}_{\pi},
\mathcal{E}_h^u)-d_{K}^{\mathrm{E}}(\boldsymbol{J}-\boldsymbol{J}_{\pi},\mathcal{E}_h^u)\big)
\\\nonumber&\leq C\sum_{\mathrm{E}\in\mathcal{T}_{h}}\big(\|\boldsymbol{J}-\boldsymbol{J}_{I}\|_{0,\mathrm{E}}\|\boldsymbol{B}\|_{0,3,\mathrm{E}}\|\mathcal{E}_h^u\|_{0,6,\mathrm{E}}
+\|\boldsymbol{J}-\boldsymbol{J}_{\pi}\|_{0,\mathrm{E}}\|\boldsymbol{B}\|_{0,3,\mathrm{E}}\|\mathcal{E}_h^u\|_{0,6,\mathrm{E}}\big)
\\&\leq Ch^{\tau}|\boldsymbol{J}|_{\tau,\mathrm{\Omega}}
|\mathcal{E}_h^u|_{1,\mathrm{\Omega}}.
\label{mathrcal_N5}
\end{align}
\item[$\blacktriangleright$] The item $\mathcal{N}_6$: By using Poincar\'{e} type inequality, continuity of $\mathrm{\Pi}_{k_2}^{0,\mathrm{E}}$ with respect to $|\cdot|_{0,\mathrm{E}}$, (\ref{continue-E-d_K}), (\ref{inf:u-Zh}) and Lemma \ref{lem:polymials-uerr}, we obtain
\begin{align}
\nonumber
\mathcal{N}_{6} &= \sum_{\mathrm{E}\in\mathcal{T}_{h}}\big(
d_{K_h}(\mathcal{E}_h^J,\boldsymbol{u}_{\pi}-\boldsymbol{u}_I)
+d_{K_h}(\mathcal{E}_h^J,\boldsymbol{u}-\boldsymbol{u}_{\pi})\big)
\\\nonumber&\leq C\sum_{\mathrm{E}\in\mathcal{T}_{h}}\big(\|\mathcal{E}_h^J\|_{0,\mathrm{E}}\|\boldsymbol{B}\|_{0,3,\mathrm{E}}\|\boldsymbol{u}-\boldsymbol{u}_{\pi}\|_{0,6,\mathrm{E}}
+\|\mathcal{E}_h^J\|_{0,\mathrm{E}}\|\boldsymbol{B}\|_{0,3,\mathrm{E}}\|\boldsymbol{u}-\boldsymbol{u}_{I}\|_{0,6,\mathrm{E}}\big)
\\&\leq Ch^{r}
|\boldsymbol{u}|_{r+1,\mathrm{\Omega}}\|\mathcal{E}_h^J\|_{0,\mathrm{\Omega}}.
\label{mathrcal_N6}
\end{align}
\end{itemize}

Based on the above estimated results, we have
\begin{align}
\nonumber
RHS&\leq Ch^{r}\Big(h^2|\boldsymbol{f}|_{r+1,\mathrm{\Omega}}+\|\boldsymbol{u}\|_{r+1,\mathrm{\Omega}}\big(\|\boldsymbol{u}\|_{r,\mathrm{\Omega}}+2\|\boldsymbol{u}\|_{r+1,\mathrm{\Omega}}+2|\boldsymbol{u}|_{1,\mathrm{\Omega}}+|\boldsymbol{u}_h|_{1,\mathrm{\Omega}}\big)\Big)|\mathcal{E}_h^u|_{1,\mathrm{\Omega}} \\\nonumber&\quad+ Ch^{\tau}\big(|\boldsymbol{g}|_{\tau,\mathrm{\Omega}}+|\boldsymbol{J}|_{\tau,\mathrm{\Omega}}\big)\|\mathcal{E}_h^J\|_{0,\mathrm{\Omega}}
+Ch^{\tau}|\boldsymbol{J}|_{\tau,\mathrm{\Omega}}
|\mathcal{E}_h^u|_{1,\mathrm{\Omega}}+Ch^{r}|\boldsymbol{u}|_{r+1,\mathrm{\Omega}}\|\mathcal{E}_h^J\|_{0,\mathrm{\Omega}}+\delta|\boldsymbol{u}_h|_{1,\mathrm{\Omega}}|\mathcal{E}_h^u|_{1,\mathrm{\Omega}}^2
\\\nonumber&\leq
Ch^{r}\Big(h^2|\boldsymbol{f}|_{r+1,\mathrm{\Omega}}+|\boldsymbol{J}|_{\tau,\mathrm{\Omega}}+\|\boldsymbol{u}\|_{r+1,\mathrm{\Omega}}\big(\|\boldsymbol{u}\|_{r,\mathrm{\Omega}}+2\|\boldsymbol{u}\|_{r+1,\mathrm{\Omega}}+2|\boldsymbol{u}|_{1,\mathrm{\Omega}}+|\boldsymbol{u}_h|_{1,\mathrm{\Omega}}\big)\Big)(|\mathcal{E}_h^u|_{1,\mathrm{\Omega}}^2\\\nonumber&\quad+\|\mathcal{E}_h^J\|_{0,\mathrm{\Omega}}^2)^{\frac{1}{2}} +  Ch^{\tau}\big(|\boldsymbol{g}|_{\tau,\mathrm{\Omega}}+|\boldsymbol{J}|_{\tau,\mathrm{\Omega}}+|\boldsymbol{u}|_{r+1,\mathrm{\Omega}}\big)(|\mathcal{E}_h^u|_{1,\mathrm{\Omega}}^2+\|\mathcal{E}_h^J\|_{0,\mathrm{\Omega}}^2)^{\frac{1}{2}} \\\nonumber&\quad+\delta|\boldsymbol{u}_h|_{1,\mathrm{\Omega}}(|\mathcal{E}_h^u|_{1,\mathrm{\Omega}}^2+\|\mathcal{E}_h^J\|_{0,\mathrm{\Omega}}^2)^{\frac{1}{2}}.
\end{align}

Now, we bring stabilities (\ref{stability-uh}) and (\ref{stability-Jh}) into the $LHS$,
\begin{align}
\nonumber
LHS\geq \check{C}_{\min}(|\mathcal{E}_h^u|_{1,\mathrm{\Omega}}^2+\|\mathcal{E}_h^J\|_{0,\mathrm{\Omega}}^2),
\end{align}
while combining with estimation of $RHS$ and Theorem \ref{lem:discrete problem well-posed}, \ref{lem:continue problem well-posed} to obtain that
\begin{align}
(1-\delta\check{C}_{\min}^{-2}\|\mathrm{F}\|_{*})(|\mathcal{E}_h^u|_{1,\mathrm{\Omega}}+\|\mathcal{E}_h^J\|_{0,\mathrm{\Omega}})\leq h^{\min\{r,\tau\}}\mathcal{V}_{1}(\|\boldsymbol{u}\|_{r+1,\mathrm{\Omega}};|\boldsymbol{f}|_{r+1,\mathrm{\Omega}};|\boldsymbol{g}|_{\tau,\mathrm{\Omega}};|\boldsymbol{J}|_{\tau,\mathrm{\Omega}}),
\end{align}
where $\mathcal{V}_{1}$ is suitable function independent of $h$.

Particularly, we obverse that $(1-\delta\check{C}_{\min}^{-2}\|\mathrm{F}\|_{*})=(1-\mu)>0$, and applying (\ref{inf:u-Zh}), (\ref{inf:J-Yh}) and triangle inequality, it implies that
\begin{align}
\nonumber
|\boldsymbol{u}-\boldsymbol{u}_h|_{1,\mathrm{\Omega}}+\|\boldsymbol{J}-\boldsymbol{J}_h\|_{0,\mathrm{\Omega}}\leq h^{\min\{r,\tau\}}\mathcal{V}_{2}(\|\boldsymbol{u}\|_{r+1,\mathrm{\Omega}};|\boldsymbol{f}|_{r+1,\mathrm{\Omega}};\|\boldsymbol{g}\|_{\tau,\mathrm{\Omega}};\|\boldsymbol{J}\|_{\tau,\mathrm{\Omega}}),
\end{align}
where $\mathcal{V}_{2}$ is suitable function independent of $h$.
This means that we have completed the proof of (\ref{error:u-uh J-Jh}).

Based on the above conclusion, we begin to prove (\ref{error:p-ph}). Let $(\boldsymbol{u},p,\boldsymbol{J},\phi)\in\boldsymbol{H}^{r+1}\cap\boldsymbol{V}\times H^{r}\cap Q\times\boldsymbol{H}^{\tau}\cap\boldsymbol{S}\times H^{\tau}\cap\Psi$ be the solution of problem (\ref{variational:continous}) and $(\boldsymbol{u}_{h},p_{h},\boldsymbol{J}_{h},\phi_{h})\in\boldsymbol{V}_{h}\times Q_{h}\times\boldsymbol{S}_{h}\times\Psi_{h}$ be the solution of problem (\ref{variational:discretizations}). Then there exists $q_h\in Q_{h}$ such that
\begin{align}
\nonumber
b_{q}(q_h-p_h,\boldsymbol{v}_h) =\;& b_{q}(q_h-p,\boldsymbol{v}_h)+\big\{(\boldsymbol{f}_h,\boldsymbol{v}_h)-(\boldsymbol{f},\boldsymbol{v}_h)\big\}+\big\{a_{v}(\boldsymbol{u},\boldsymbol{v}_h)
-a_{v_h}(\boldsymbol{u}_h,\boldsymbol{v}_h)\big\}+ \\\nonumber&
\big\{c_{v}(\boldsymbol{u},\boldsymbol{u},\boldsymbol{v}_h)-c_{v_h}(\boldsymbol{u}_h,\boldsymbol{u}_h,\boldsymbol{v}_h)\big\}
+\big\{d_{K_h}(\boldsymbol{J}_h,\boldsymbol{v}_h)-d_{K}(\boldsymbol{J},\boldsymbol{v}_h)\big\}
\\\nonumber=\;&\mathcal{M}_{1}+\mathcal{M}_{2}+\mathcal{M}_{3}+\mathcal{M}_{4}+\mathcal{M}_{5}
\end{align}
for all $(\boldsymbol{v}_{h},q_{h},\boldsymbol{K}_{h},\psi_{h})\in\boldsymbol{V}_{h}\times Q_{h}\times\boldsymbol{S}_{h}\times\boldsymbol{\Psi}_{h}$. Now, we separately estimate these six items $\mathcal{M}_{i}$ $(i = 1,2,3,4,5)$.
\begin{itemize}
\setstretch{1.25}
\item[$\blacktriangleright$] The item $\mathcal{M}_1$: From (\ref{interplate-qh}), we easily obtain
\begin{align}
\mathcal{M}_1 \leq Ch^{r}|p|_{r,\mathrm{\Omega}}|\boldsymbol{v}_h|_{1,\mathrm{\Omega}}.
\label{mathrcal_M1}
\end{align}
\item[$\blacktriangleright$] The item $\mathcal{M}_2$: From Lemma \ref{load term:f-fh_error}, we easily obtain
\begin{align}
\mathcal{M}_2 \leq Ch^{r+2}|\boldsymbol{f}|_{r+1,\mathrm{\Omega}}|\boldsymbol{v}_h|_{1,\mathrm{\Omega}}.
\label{mathrcal_M2}
\end{align}

\item[$\blacktriangleright$] The item $\mathcal{M}_3$: By adding and subtracting $\boldsymbol{u}_{\pi}$ (the piecewise linear polynomial approximate of $\boldsymbol{u}$), while applying the $k_1$-consistency (\ref{k_1-consistency}), (\ref{continue-E-a_v}), we find
\begin{align}
\nonumber
\mathcal{M}_3 &=
\sum_{\mathrm{E}\in\mathcal{T}_{h}}\big(a_{v}^{\mathrm{E}}(\boldsymbol{u},\boldsymbol{v}_h)
-a_{v_h}^{E}(\boldsymbol{u}_h,\boldsymbol{v}_h)\big)\\\nonumber
&=\sum_{\mathrm{E}\in\mathcal{T}_{h}}\big(a_{v}^{\mathrm{E}}(\boldsymbol{u}-\boldsymbol{u}_{\pi},\boldsymbol{v}_h)
-a_{v_h}^{E}(\boldsymbol{u}_h-\boldsymbol{u}_{\pi},\boldsymbol{v}_h)\big)
\\&\leq C \sum_{\mathrm{E}\in\mathcal{T}_{h}}\big(|\boldsymbol{u}-\boldsymbol{u}_{\pi}|_{1,\mathrm{E}}+
|\boldsymbol{u}-\boldsymbol{u}_h|_{1,\mathrm{E}}\big)|\boldsymbol{v}_h|_{1,\mathrm{E}}.
\label{mathrcal_M3}
\end{align}

\item[$\blacktriangleright$] The item $\mathcal{M}_3$: Obviously, we clearly observe that
\begin{align}
\mathcal{M}_4 &= \big(c_{v}(\boldsymbol{u},\boldsymbol{u},\boldsymbol{v}_h)-c_{v_h}(\boldsymbol{u},\boldsymbol{u},\boldsymbol{v}_h)\big)
+\big(c_{v_h}(\boldsymbol{u}-\boldsymbol{u}_h,\boldsymbol{u}_h,\boldsymbol{v}_h)+c_{v_h}(\boldsymbol{u},\boldsymbol{u}-\boldsymbol{u}_h,\boldsymbol{v}_h)\big),
\label{trilinear term:u-uh-estimate results-equality1}
\end{align}
while taking Lemma \ref{trilinear:c_v-c_vh} and (\ref{prop:bilinear-continuous-uh}) into (\ref{trilinear term:u-uh-estimate results-equality1}), then we infer that
\begin{align}
\nonumber
\mathcal{M}_{4} &\leq Ch^{r}\Big(\|\boldsymbol{u}\|_{r,\mathrm{\Omega}}+|\boldsymbol{u}|_{1,\mathrm{\Omega}}+\|\boldsymbol{u}\|_{r+1,\mathrm{\Omega}}\Big)\|\boldsymbol{u}\|_{r+1,\mathrm{\Omega}}|\boldsymbol{v}_{h}|_{1,\mathrm{\Omega}}
+\\&\quad\delta\Big(|\boldsymbol{u}_h|_{1,\mathrm{\Omega}}|\boldsymbol{u}-\boldsymbol{u}_h|_{1,\mathrm{\Omega}}
+|\boldsymbol{u}|_{1,\mathrm{\Omega}}|\boldsymbol{u}-\boldsymbol{u}_h|_{1,\mathrm{\Omega}}\Big)|\boldsymbol{v}_{h}|_{1,\mathrm{\Omega}}.
\label{mathrcal_M4}
\end{align}

\item[$\blacktriangleright$]
The item $\mathcal{M}_5$:
By using Poincar$\acute{e}$ type inequality, continuity of $\mathrm{\Pi}_{k_2}^{0,\mathrm{E}}$ with respect to $\|\cdot\|_{0,\mathrm{E}}$ and (\ref{norm:L^6-continous}), we obtain
\begin{align}
\nonumber
\mathcal{M}_{5} &= \sum_{\mathrm{E}\in\mathcal{T}_{h}}\big(d_{K_h}^{\mathrm{E}}(\boldsymbol{J}_h-\boldsymbol{J}_{\pi},
\boldsymbol{v}_{h})-d_{K}^{\mathrm{E}}(\boldsymbol{J}-\boldsymbol{J}_{\pi},\boldsymbol{v}_{h})\big)
\\\nonumber&\leq
C\sum_{\mathrm{E}\in\mathcal{T}_{h}}\big(\|\boldsymbol{J}-\boldsymbol{J}_{h}\|_{0,\mathrm{E}}\|\boldsymbol{B}\|_{0,3,\mathrm{E}}\|\boldsymbol{v}_{h}\|_{0,6,\mathrm{E}}
+\|\boldsymbol{J}-\boldsymbol{J}_{\pi}\|_{0,\mathrm{E}}\|\boldsymbol{B}\|_{0,3,\mathrm{E}}\|\boldsymbol{v}_{h}\|_{0,6,\mathrm{E}}\big)
\\&\leq C\sum_{\mathrm{E}\in\mathcal{T}_{h}}\big(\|\boldsymbol{J}-\boldsymbol{J}_{h}\|_{0,\mathrm{E}}|\boldsymbol{v}_{h}|_{1,\mathrm{E}}
+\|\boldsymbol{J}-\boldsymbol{J}_{\pi}\|_{0,\mathrm{E}}|\boldsymbol{v}_{h}|_{1,\mathrm{E}}\big).
\label{mathrcal_M5}
\end{align}
\end{itemize}
Based on the above results, Lemma \ref{lem:polymials-uerr}, \ref{lem:polymials-Kerr}, (\ref{error:u-uh J-Jh}), (\ref{interplate-qh}), inf-sup condition (\ref{inf-sup-uh}) and triangle inequality, we readily infer that
\begin{align}
\nonumber
\|p-p_h\|_{0,\mathrm{\Omega}}\leq h^{\min\{\tau,r\}}\mathcal{V}_{3}(\|\boldsymbol{u}\|_{r+1,\mathrm{\Omega}},|\boldsymbol{f}|_{r+1,\mathrm{\Omega}},\|\boldsymbol{g}\|_{\tau,\mathrm{\Omega}},\|\boldsymbol{J}\|_{\tau,\mathrm{\Omega}},\|p\|_{\tau,\mathrm{\Omega}}),
\end{align}
where $\mathcal{S}_{3}$ is suitable function independent of $h$.

Finally, we begin to prove (\ref{error:phi-phih}). Similar to the proof process of (\ref{error:p-ph}), there exists $\psi_h\in \Psi_{h}$ such that
\begin{align}
\nonumber
b_{\psi}(\psi_{h}-\phi_h,\boldsymbol{K}_h) =\;& b_{\psi}(\psi_{h}-\phi,\boldsymbol{K}_h)+\big\{(\boldsymbol{g}_h,\boldsymbol{K}_h)-(\boldsymbol{g},\boldsymbol{K}_h)\big\}+\big\{a_{K}(\boldsymbol{J},\boldsymbol{K}_h)
-a_{K_h}(\boldsymbol{J}_h,\boldsymbol{K}_h)\big\}+ \\\nonumber&
+\big\{d_{K}(\boldsymbol{K}_{h},\boldsymbol{u})-d_{K_h}(\boldsymbol{K}_{h},\boldsymbol{u}_h)\big\}
\\\nonumber=\;&\mathcal{Q}_{1}+\mathcal{Q}_{2}+\mathcal{Q}_{3}+\mathcal{Q}_{4},
\end{align}
Now, we separately estimate these six items $\mathcal{Q}_{i}$ $(i = 1,2,3,4)$.
\begin{itemize}
\setstretch{1.25}
\item[$\blacktriangleright$] The item $\mathcal{Q}_1$: From (\ref{interplate-psih}), we easily obtain
\begin{align}
\mathcal{Q}_1 \leq Ch^{\tau}|\phi|_{\tau}\|\boldsymbol{K}_h\|_{\textup{div}}.
\label{mathrcal_Q1}
\end{align}
\item[$\blacktriangleright$] The item $\mathcal{Q}_2$: From Lemma \ref{load term:g-gh_error}, we easily obtain
\begin{align}
\mathcal{Q}_2 \leq Ch^{\tau}|\boldsymbol{g}|_{\tau}\|\boldsymbol{K}_h\|_{0}.
\label{mathrcal_Q2}
\end{align}

\item[$\blacktriangleright$] The item $\mathcal{Q}_3$: By adding and subtracting $\boldsymbol{J}_{\pi}$ (the piecewise linear polynomial approximate of $\boldsymbol{J}$), while applying the $k_2$-consistency (\ref{k_2-consistency}), (\ref{continue-E-a_K}) and Lemma \ref{lem:polymials-Kerr}, we find
\begin{align}
\nonumber
\mathcal{Q}_3 &=
\sum_{\mathrm{E}\in\mathcal{T}_{h}}\big(a_{K}^{\mathrm{E}}(\boldsymbol{J},\boldsymbol{K}_h)
-a_{K_h}^{E}(\boldsymbol{J}_h,\boldsymbol{K}_h)\big)\\\nonumber
&=\sum_{\mathrm{E}\in\mathcal{T}_{h}}\big(a_{K}^{\mathrm{E}}(\boldsymbol{J}-\boldsymbol{J}_{\pi},\boldsymbol{K}_h)
-a_{K_h}^{E}(\boldsymbol{J}_h-\boldsymbol{J}_{\pi},\boldsymbol{K}_h)
\big)
\\&\leq C \sum_{\mathrm{E}\in\mathcal{T}_{h}}\big(\|\boldsymbol{J}-\boldsymbol{J}_{\pi}\|_{0,\mathrm{E}}+
\|\boldsymbol{J}-\boldsymbol{J}_h\|_{0,\mathrm{E}}\big)\|\boldsymbol{K}_h\|_{0,\mathrm{E}}.
\label{mathrcal_Q3}
\end{align}

\item[$\blacktriangleright$] The item $\mathcal{Q}_4$: By using Poincar\'{e} type inequality, continuity of $\mathrm{\Pi}_{k_2}^{0,\mathrm{E}}$ with respect to $\|\cdot\|_{0,\mathrm{E}}$, (\ref{norm:L^6-continous}), (\ref{inf:u-Zh}) and Lemma \ref{lem:polymials-uerr}, we obtain
\begin{align}
\nonumber
\mathcal{Q}_{4} &= \sum_{\mathrm{E}\in\mathcal{T}_{h}}\big(d_{K}(\boldsymbol{K}_{h},\boldsymbol{u}-\boldsymbol{u}_{\pi})
+d_{K_h}(\boldsymbol{K}_{h},\boldsymbol{u}_{\pi}-\boldsymbol{u}_h)
\big)
\\\nonumber&\leq C\sum_{\mathrm{E}\in\mathcal{T}_{h}}\big(\|\boldsymbol{K}_{h}\|_{0,\mathrm{E}}\|\boldsymbol{B}\|_{0,3,\mathrm{E}}\|\boldsymbol{u}-\boldsymbol{u}_{\pi}\|_{0,6,\mathrm{E}}
+\|\boldsymbol{K}_{h}\|_{0,\mathrm{E}}\|\boldsymbol{B}\|_{0,3,\mathrm{E}}\|\boldsymbol{u}_{\pi}-\boldsymbol{u}_{h}\|_{0,6,\mathrm{E}}\big)
\\&\leq C\sum_{\mathrm{E}\in\mathcal{T}_{h}}\big(\|\boldsymbol{K}_{h}\|_{0,\mathrm{E}}|\boldsymbol{u}-\boldsymbol{u}_{\pi}|_{1,\mathrm{E}}
+\|\boldsymbol{K}_{h}\|_{0,\mathrm{E}}|\boldsymbol{u}-\boldsymbol{u}_{h}|_{1,\mathrm{E}}\big).
\label{mathrcal_Q4}
\end{align}
\end{itemize}
Based on the above results (\ref{mathrcal_Q1})-(\ref{mathrcal_Q4}), Lemma \ref{lem:polymials-uerr}, \ref{lem:polymials-Kerr}, (\ref{error:u-uh J-Jh}), (\ref{interplate-psih}), inf-sup condition (\ref{inf-sup-Jh}) and triangle inequality, we readily infer that
\begin{align}
\nonumber
\|\phi-\phi_h\|_{0}\leq h^{\min\{\tau,r\}}\mathcal{V}_{4}(\|\boldsymbol{u}\|_{r+1},|\boldsymbol{f}|_{r+1},|\boldsymbol{g}|_{\tau},|\boldsymbol{J}|_{\tau},|\phi|_{\tau}),
\end{align}
where $\mathcal{V}_{4}$ is suitable function independent of $h$.
\end{proof}

\section{Numerical experiments\label{sec:Num}}

In this section, we present two numerical examples to verify the effectiveness of the proposed full divergence-free hight order virtual finite element method for solving inductionless MHD equations on polygonal meshes. In the first experiment, we investigate the convergence and conservations of the proposed method.
 %In the second experiment, we study the pressure and electric potential robust of the proposed.
Throughout all examples, Stokes iterative method is used to deal with nonlinear term and iterative tolerance $\|\boldsymbol{u}_{h}^{n+1}-\boldsymbol{u}_{h}^{n}\|_{0} = 10^{-6}$ (the $\boldsymbol{u}_{h}^{n}$ and $\boldsymbol{u}_{h}^{n+1}$ denote the numerical solution for velocity at iteration steps $n$ and $n+1$, respectively), the numerical experiment results are realized by MATLAB software.

\subsection{Inductionless MHD problem with a smooth solution\label{subsec:Smooth test}}
In this test, we consider the inductionless MHD equations on $\mathrm{\Omega}=[0,1]^2\in\mathbb{R}^{2}$ with homogeneous boundary conditions, the physics parameters $\nu=S_{c}=1$, the $\boldsymbol{B}=(0,0,1)$ , and we let the functions $\boldsymbol{f}$ and $\boldsymbol{g}$ such that exact solution is
\begin{align*}
&p = sin(2\pi x)sin(2\pi y),\quad \phi = sin(x)-sin(y),\\
&J_1 = sin(\pi x)cos(\pi y),\quad J_2 = -sin(\pi y)cos(\pi x),\\
&u_1 = -0.5cos(x)^2cos(y)sin(y),\quad u_2 = 0.5cos(y)^2cos(x)sin(x).
\end{align*}
Two sets of polynomial degree of accuracy are designed for this numerical experiment which are $k_1 = 2$, $k_2=1$ and $k_1 = 3$, $k_2 = 2$. In regard to mesh partition of the computational domain, we use the following three types, see Fig.\ref{fig.1}-\ref{fig.2}:

\begin{itemize}
\setstretch{1.25}
\item [$\bullet$] $\{\mathcal{T}_{h}^{a}\}_{N_t^a}$: Sets of Rondom meshes with $N_t^a,$
\item [$\bullet$] $\{\mathcal{T}_{h}^{b}\}_{N_t^b}$: Sets of Remapped square meshes with $N_t^b,$
\item [$\bullet$] $\{\mathcal{T}_{h}^{c}\}_{N_t^c}$: Sets of nonConvex meshes with
    $N_t^c,$
\end{itemize}
where $N_t^i = 25$, $100$, $225$, $400$, $625$ denote the total number of elements in $\mathcal{T}_h^{i}$ ($i = a$, $b$, $c$).

\begin{figure}
\centering%图片居中
\subfloat[$\{\mathcal{T}_{h}^{a}\}_{100}$]{
\label{fig.1a}
\includegraphics[width=5.8cm]{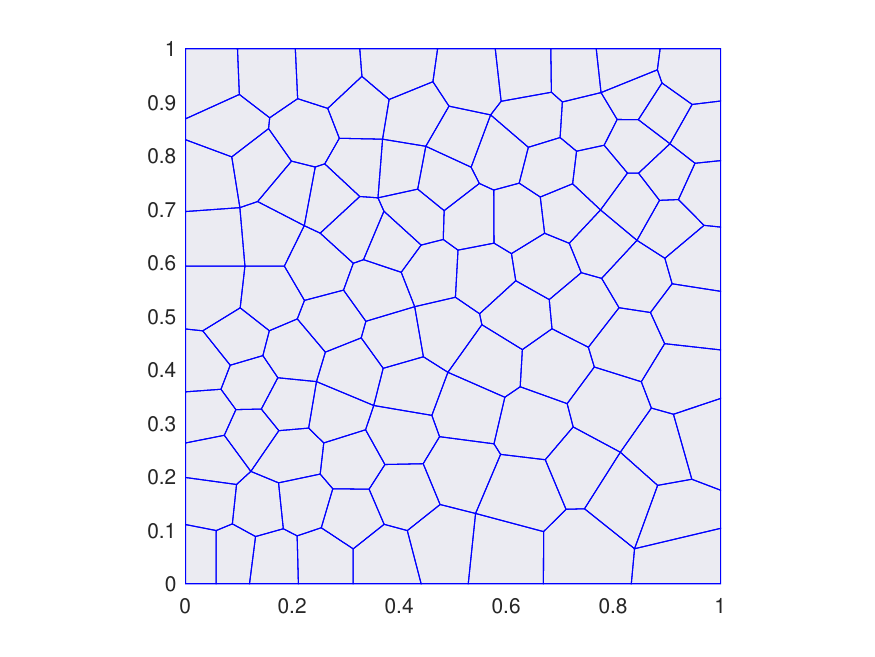}
}
\hspace{-1.20cm}
\subfloat[$\{\mathcal{T}_{h}^{b}\}_{100}$]{
\label{fig.1b}
\includegraphics[width=5.8cm]{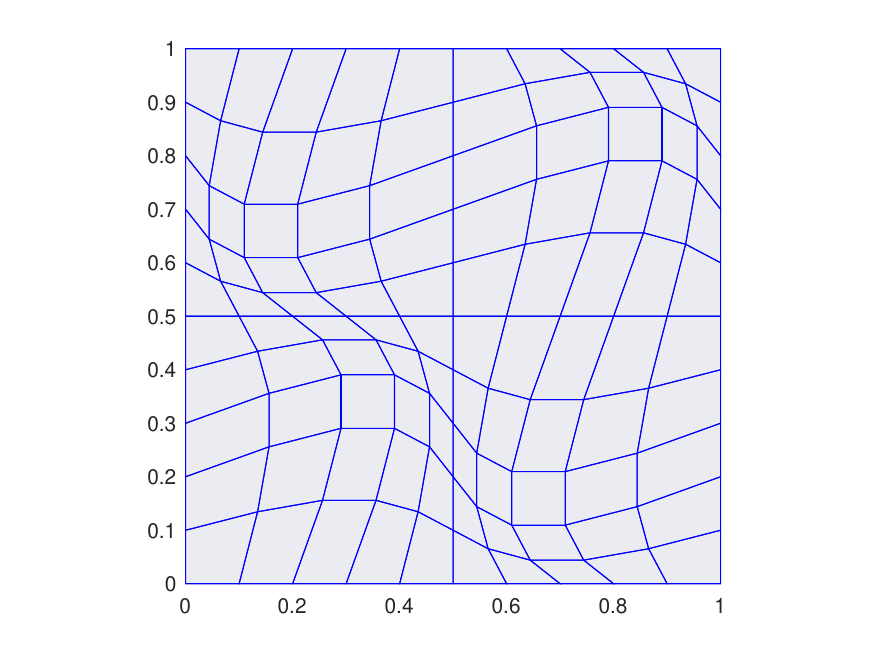}}
\hspace{-1.20cm}
\subfloat[$\{\mathcal{T}_{h}^{c}\}_{100}$]{
\label{fig.1c}
\includegraphics[width=5.8cm]{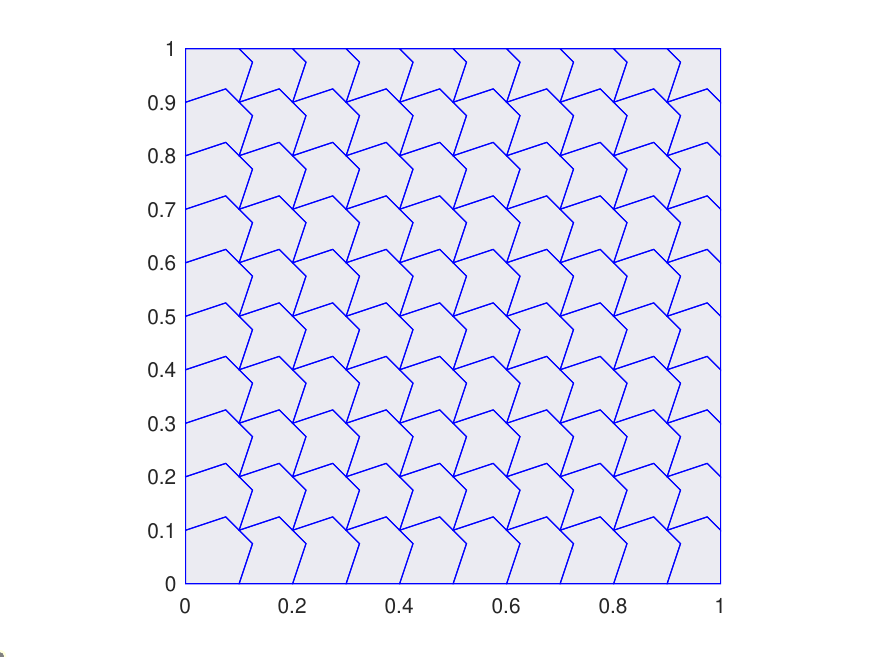}}
\caption{A set of coarse polygonal meshes ($N_{t}^{i} = 100$) used to the first numerical experiment.}
\label{fig.1}%与\ref连用，在文章与表格之间建立链接
\end{figure}

\begin{figure}
\centering%图片居中
\subfloat[$\{\mathcal{T}_{h}^{a}\}_{625}$]{
\label{fig.2a}
\includegraphics[width=5.8cm]{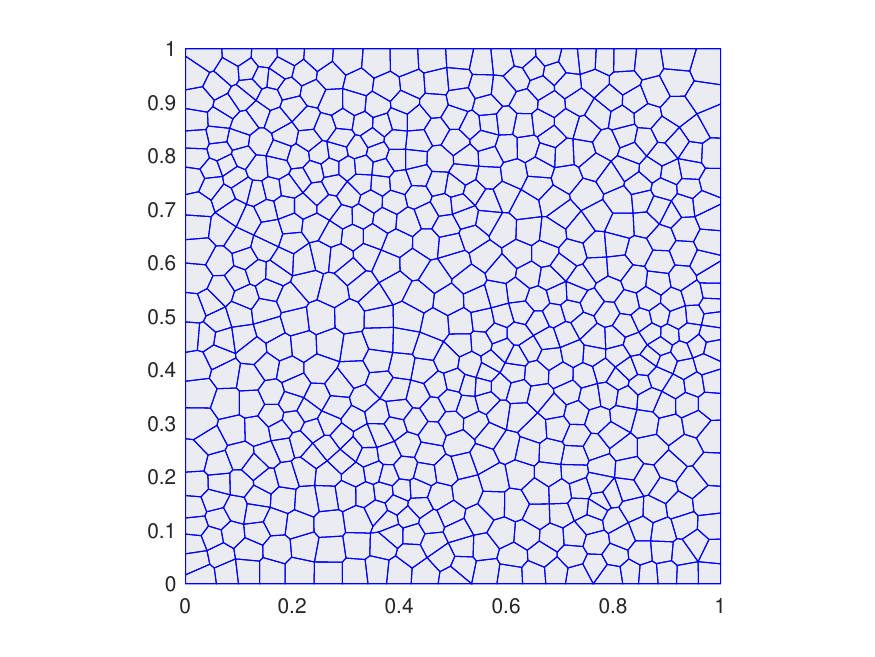}
}
\hspace{-1.20cm}
\subfloat[$\{\mathcal{T}_{h}^{b}\}_{625}$]{
\label{fig.2b}
\includegraphics[width=5.8cm]{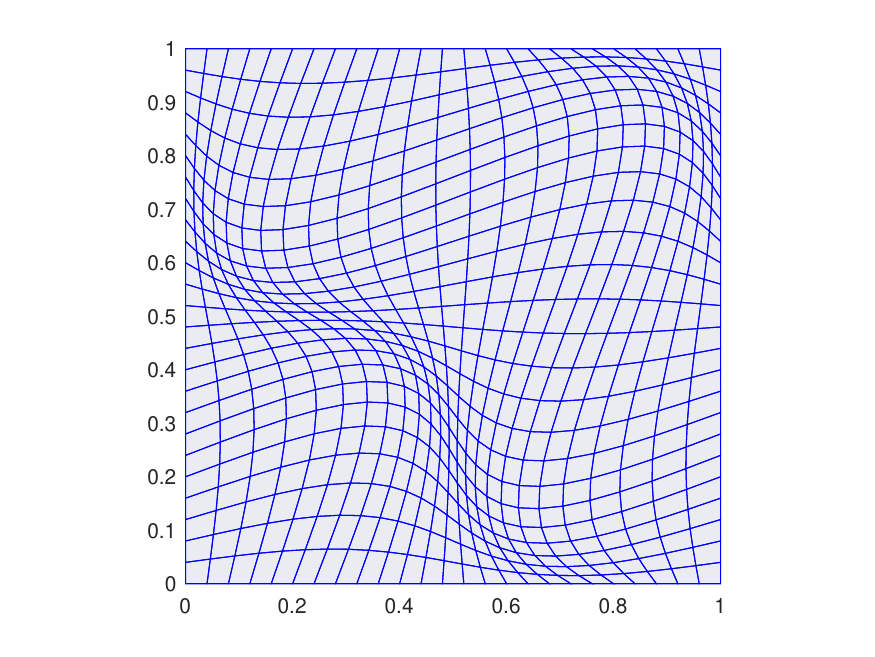}}
\hspace{-1.20cm}
\subfloat[$\{\mathcal{T}_{h}^{c}\}_{625}$]{
\label{fig.2c}
\includegraphics[width=5.8cm]{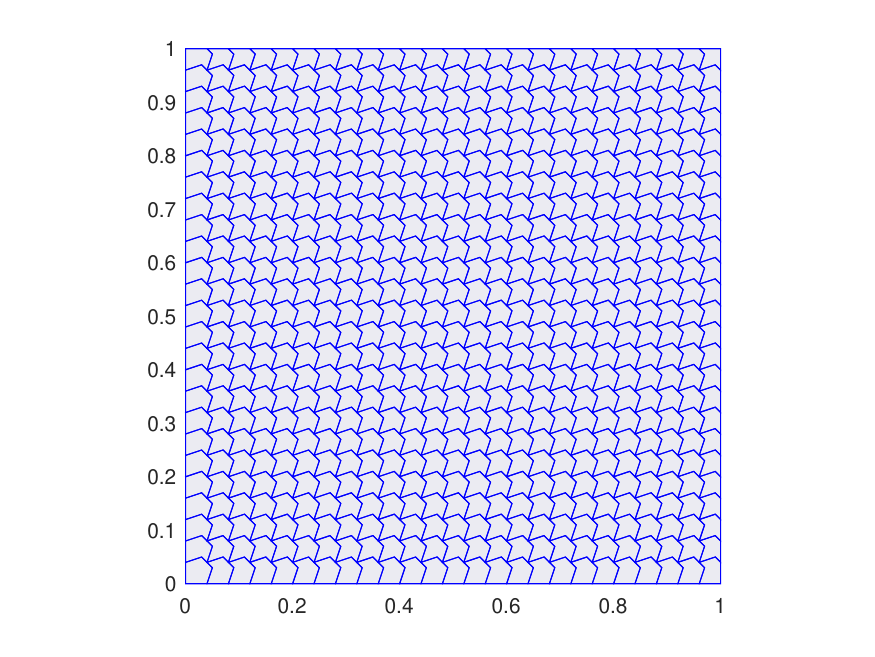}}
\caption{A set of fine polygonal meshes ($N_{t}^{i} = 625$) used to the first numerical experiment.}
\label{fig.2}%与\ref连用，在文章与表格之间建立链接
\end{figure}

In order to continue the convergence test, we have to note that the VEM solutions $\boldsymbol{u}_{h}$ and $\boldsymbol{J}_{h}$ are not explicitly known inside the elements, so we choose polynomial projections $\mathrm{\Pi}_{k_1}^{\nabla}\boldsymbol{u}$ and $\mathrm{\Pi}_{k_2}^{0}\boldsymbol{J}$ to compare with $\boldsymbol{u}$ and $\boldsymbol{J}$, respectively. Additionally, the variables $p_{h}$ and $\phi_{h}$ are piecewise polynomials such that we can directly to compute errors. In short, we apply the computable scheme
\begin{align*}
&\text{error}(\boldsymbol{u},H^{1}):=\sqrt{\sum_{E\in\mathcal{T}_{h}}\|\nabla\boldsymbol{u}-\nabla\mathrm{\Pi}_{k_1}^{\nabla}\boldsymbol{u}\|_{0,E}^2},\\
&\text{error}(\boldsymbol{J},L^{2}):=\sqrt{\sum_{E\in\mathcal{T}_{h}}\|\boldsymbol{J}-\mathrm{\Pi}_{k_2}^{0}\boldsymbol{J}\|_{0,E}^2},\\
&\text{error}(p,L^2):=\|p-p_h\|_{0,\mathrm{\Omega}},\quad\text{error}(\phi,L^2):=\|\phi-\phi_h\|_{0,\mathrm{\Omega}},
\end{align*}
and we associate with each discretization an average mehs-size
\begin{equation}
\nonumber
h:=\frac{1}{N_t}\sum_{j=1}^{N_t}h_{E_{j}}.
\end{equation}

Next, we give the convergent order test under the first set of polynomial degree of accuracy $k_1 = 2$, $k_2=1$, and we consider three different types of meshes, see Fig.\ref{fig.3}-\ref{fig.4}. It is not difficult see that the convergence rates in Fig.\ref{fig.3}-\ref{fig.4} achieve the expected results.
\begin{figure}
\centering%图片居中
\subfloat[$\{\mathcal{T}_{h}^{a}\}_{N_t^a}$]{
\label{fig.3a}
\includegraphics[width=5.5cm]{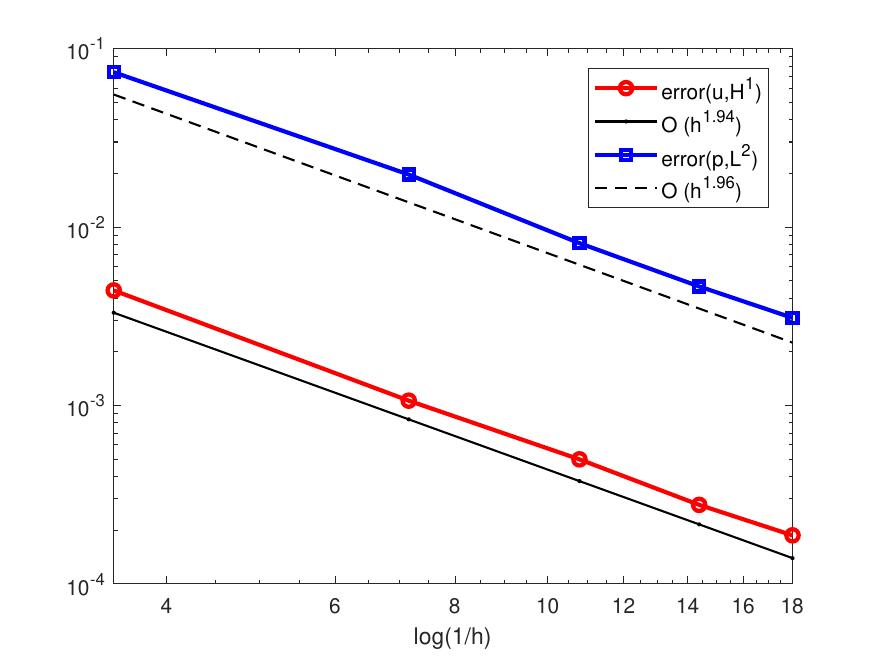}
}
\hspace{-0.7cm}
\subfloat[$\{\mathcal{T}_{h}^{b}\}_{N_t^b}$]{
\label{fig.3b}
\includegraphics[width=5.5cm]{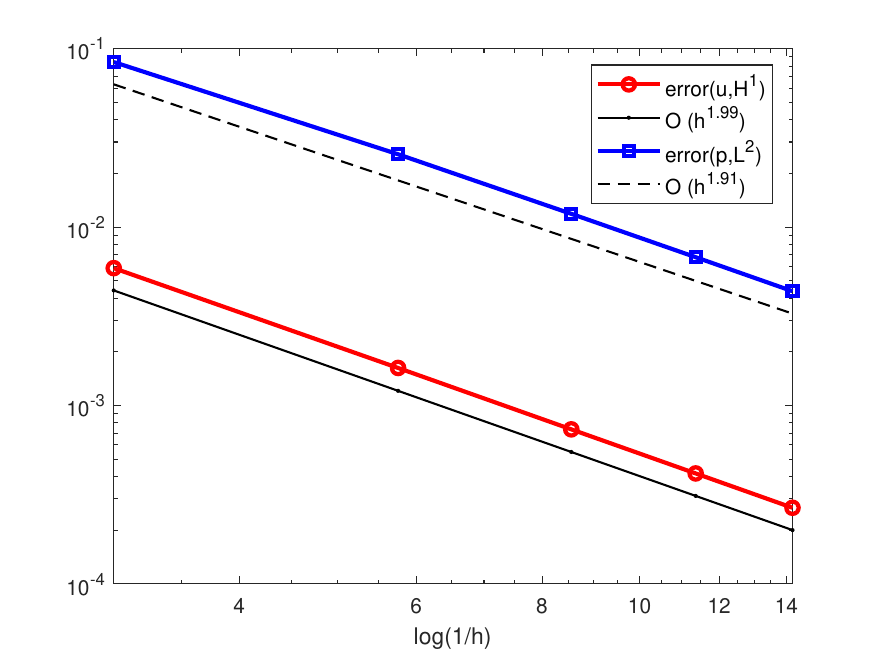}}
\hspace{-0.7cm}
\subfloat[$\{\mathcal{T}_{h}^{c}\}_{N_t^c}$]{
\label{fig.3b}
\includegraphics[width=5.5cm]{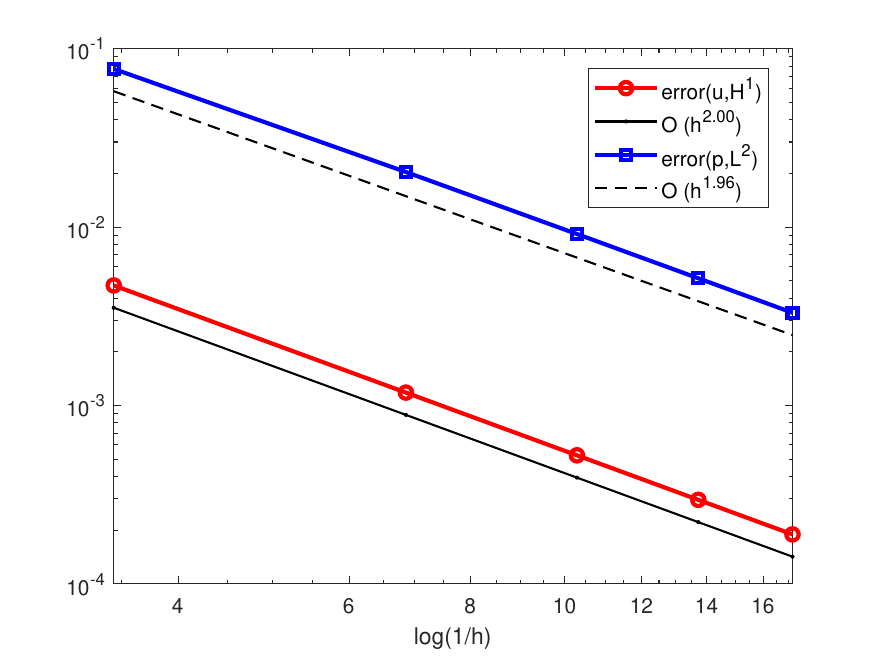}}
\caption{Convergence lines for the velocity and pressure with $k_1=2$, $k_2=1$.}
\label{fig.3}%与\ref连用，在文章与表格之间建立链接
\end{figure}
\begin{figure}
\centering%图片居中
\subfloat[$\{\mathcal{T}_{h}^{a}\}_{N_t^a}$]{
\label{fig.4a}
\includegraphics[width=5.5cm]{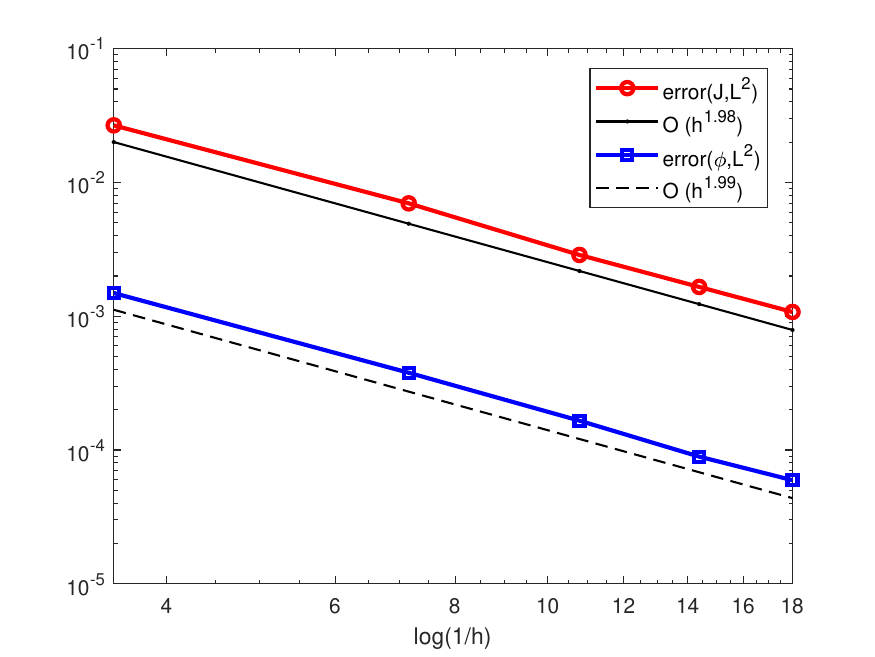}
}
\hspace{-0.7cm}
\subfloat[$\{\mathcal{T}_{h}^{b}\}_{N_t^b}$]{
\label{fig.4b}
\includegraphics[width=5.5cm]{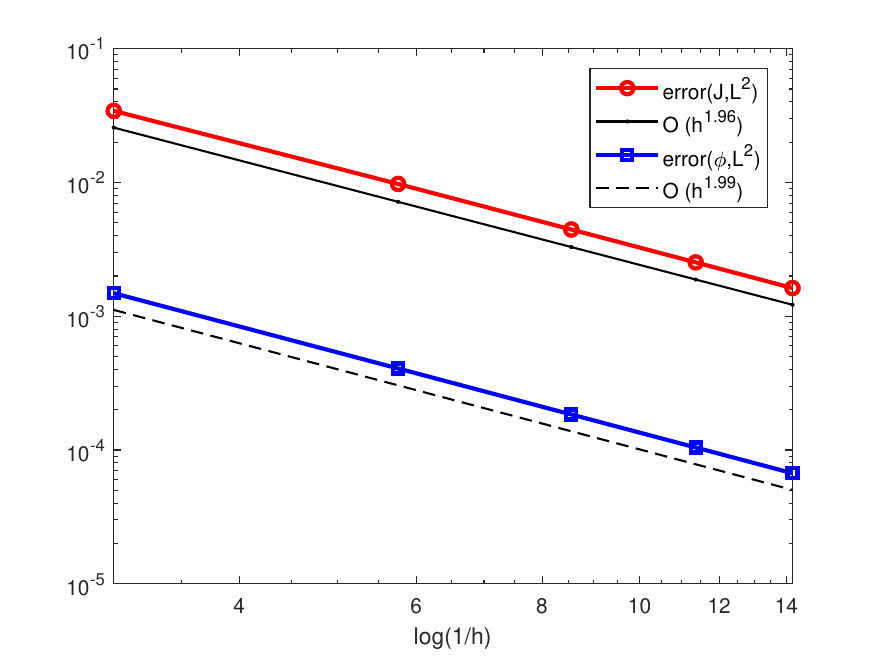}}
\hspace{-0.7cm}
\subfloat[$\{\mathcal{T}_{h}^{c}\}_{N_t^c}$]{
\label{fig.4b}
\includegraphics[width=5.5cm]{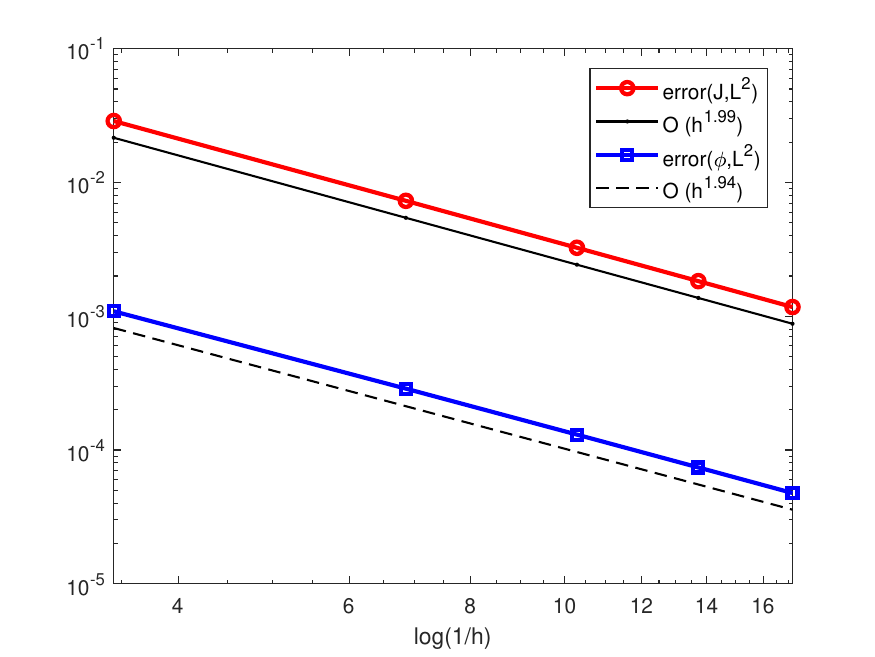}}
\caption{Convergence lines for the current density and electric potential with $k_1=2$, $k_2=1$.}
\label{fig.4}%与\ref连用，在文章与表格之间建立链接
\end{figure}
Moreover, we also give the convergent order test under the first set of polynomial degree of accuracy $k_1 = 3$, $k_2=2$, and we consider three different types of meshes, see Fig.\ref{fig.1}-\ref{fig.2}. It is not difficult see that the convergence rates in Fig.\ref{fig.5}-\ref{fig.6} achieve the expected results.
\begin{figure}
\centering%图片居中
\subfloat[$\{\mathcal{T}_{h}^{a}\}_{N_t^a}$]{
\label{fig.5a}
\includegraphics[width=5.5cm]{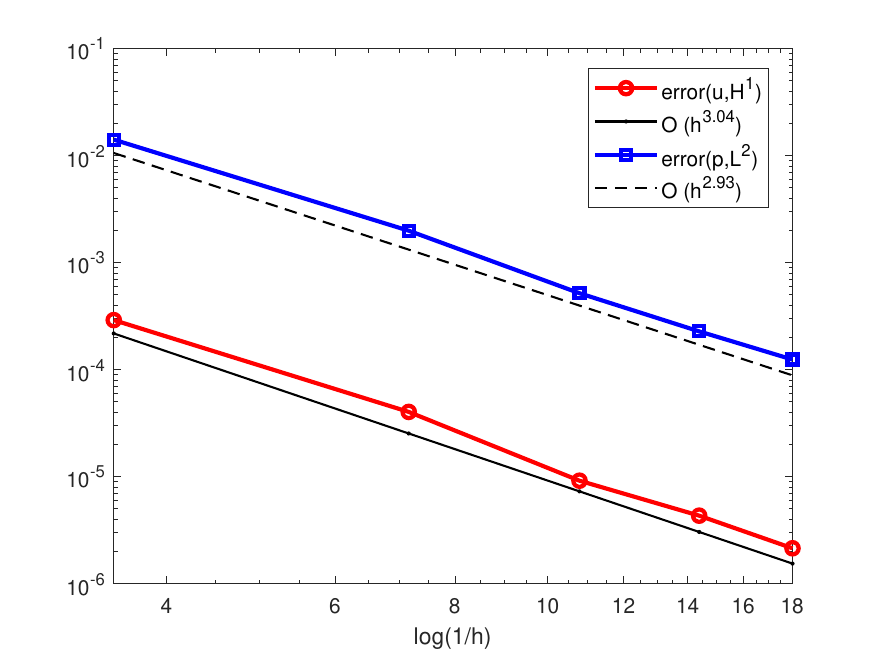}
}
\hspace{-0.7cm}
\subfloat[$\{\mathcal{T}_{h}^{b}\}_{N_t^b}$]{
\label{fig.5b}
\includegraphics[width=5.5cm]{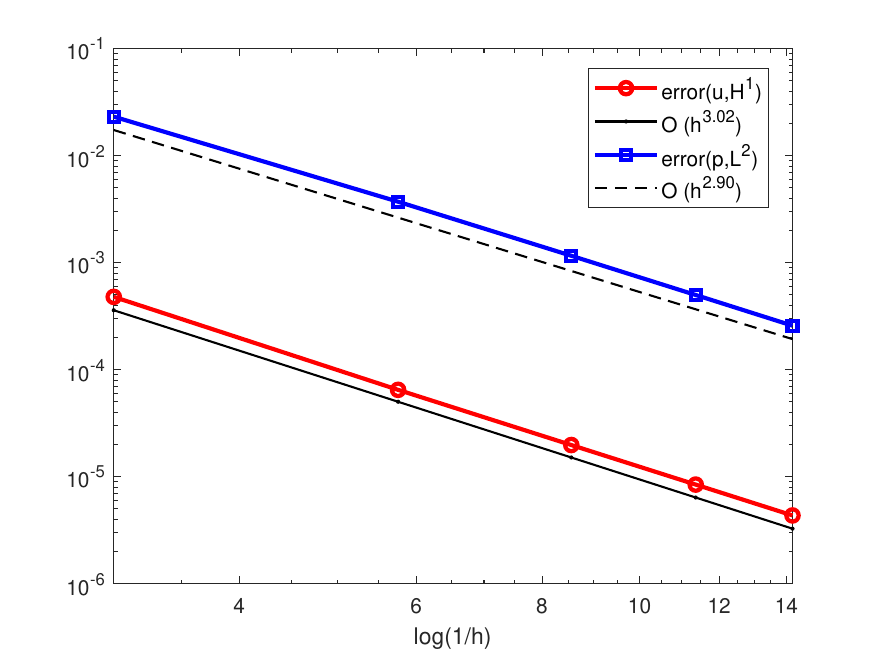}}
\hspace{-0.7cm}
\subfloat[$\{\mathcal{T}_{h}^{c}\}_{N_t^c}$]{
\label{fig.5b}
\includegraphics[width=5.5cm]{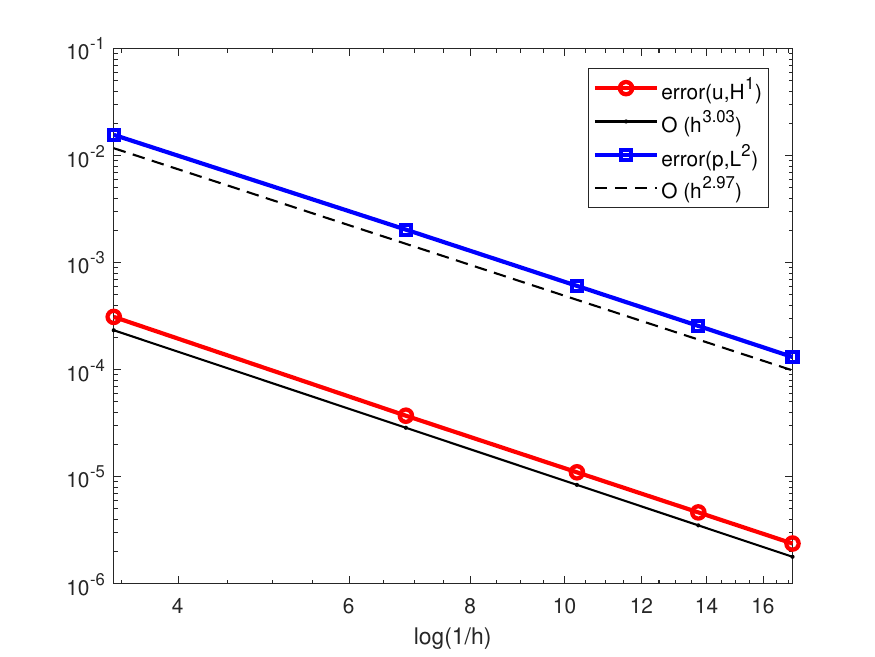}}
\caption{Convergence lines for the velocity and pressure with $k_1=3$, $k_2=2$.}
\label{fig.5}%与\ref连用，在文章与表格之间建立链接
\end{figure}
\begin{figure}
\centering%图片居中
\subfloat[$\{\mathcal{T}_{h}^{a}\}_{N_t^a}$]{
\label{fig.6a}
\includegraphics[width=5.5cm]{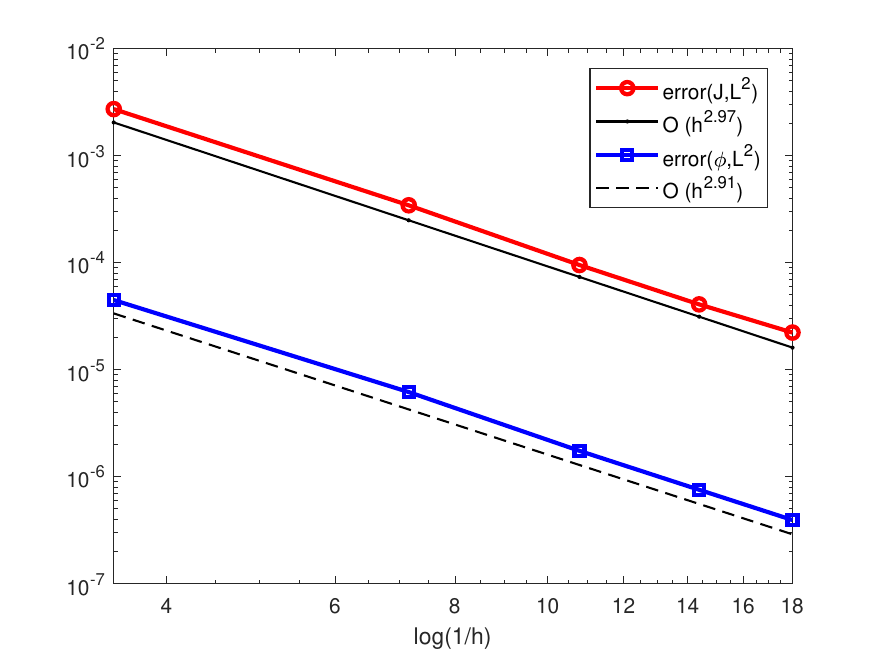}
}
\hspace{-0.7cm}
\subfloat[$\{\mathcal{T}_{h}^{b}\}_{N_t^b}$]{
\label{fig.6b}
\includegraphics[width=5.5cm]{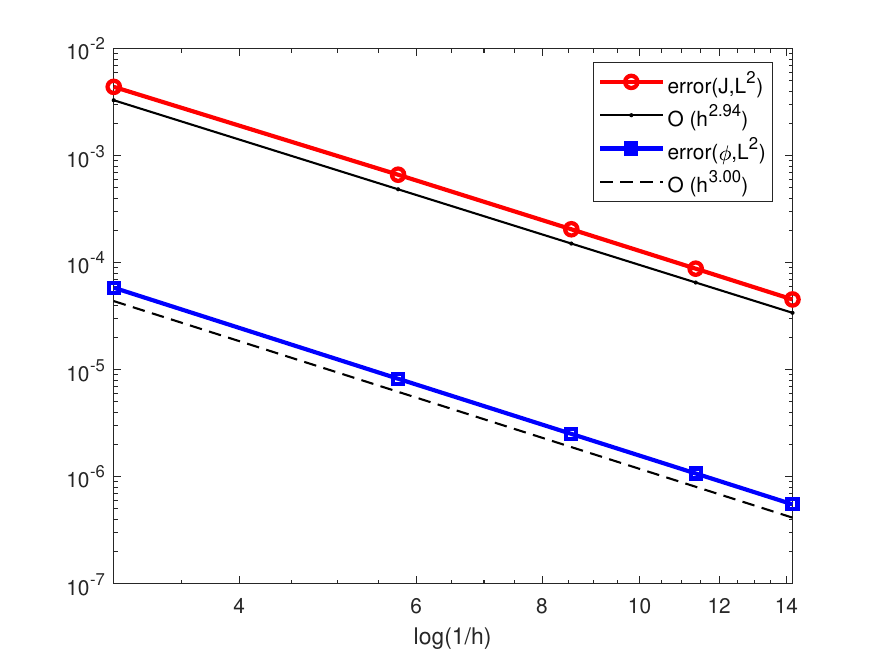}}
\hspace{-0.7cm}
\subfloat[$\{\mathcal{T}_{h}^{c}\}_{N_t^c}$]{
\label{fig.6b}
\includegraphics[width=5.5cm]{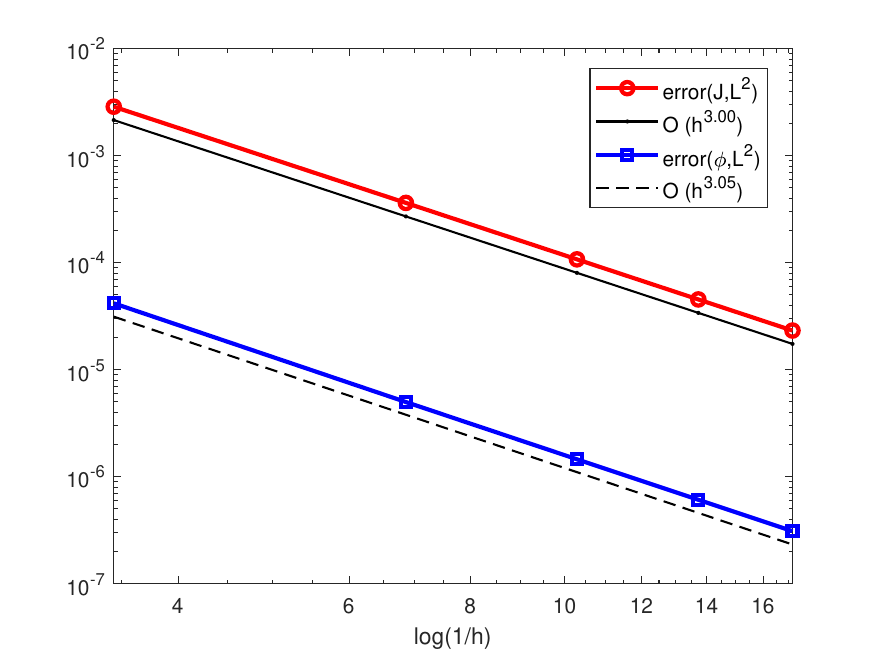}}
\caption{Convergence lines for the current density and electric potential with $k_1 = 3$, $k_2=2$.}
\label{fig.6}%与\ref连用，在文章与表格之间建立链接
\end{figure}

Finally, we give the results of conservation properties through Table \ref{tab:FD}. We find that the proposed method fully satisfies the conservations of mass and charge, that is, both $\nabla\cdot\boldsymbol{u}_{h} = 0$ and $\nabla\cdot\boldsymbol{J}_{h} = 0$ achieve machine accuracy.
\begin{table}\footnotesize
\setstretch{1.5}
\setlength{\belowcaptionskip}{5pt}
\centering
\setlength{\tabcolsep}{0.1mm}{
\caption{Conservation of the proposed method under three types of meshes\label{tab:FD}}
\begin{tabular*}{16cm}{@{\extracolsep{\fill}} c c c c c c c }
\hline
 & & $k_1=2$ & &$k_2=1$ & &\\
\hline
\multicolumn{1}{l}{\multirow{3}*{$\quad\{\mathcal{T}_{h}^{a}\}_{N_t^a}$}} &$N_t^a$ & 25 & 100 & 225 & 400 & 625$\quad$\\
& $\|\nabla\cdot\boldsymbol{u}_{h}\|_{0}$ & 8.5688e-07 & 2.8491e-07 & 6.6119e-09 & 2.4267e-08 & 8.7604e-10$\quad$\\
& $\|\nabla\cdot\boldsymbol{J}_{h}\|_{0}$ & 2.7686e-15 & 5.7942e-15 & 8.7247e-15 & 1.1518e-14 & 1.4877e-14$\quad$\\
\hline
\multicolumn{1}{l}{\multirow{3}*{$\quad\{\mathcal{T}_{h}^{b}\}_{N_t^b}$}} &$N_t^b$ & 25 & 100 & 225 & 400 & 625$\quad$\\
& $\|\nabla\cdot\boldsymbol{u}_{h}\|_{0}$ & 1.5832e-16 & 2.7300e-16 & 4.8885e-16 & 6.5470e-16 & 8.7174e-16$\quad$\\
& $\|\nabla\cdot\boldsymbol{J}_{h}\|_{0}$ & 3.0312e-15 & 5.9777e-15 & 9.5903e-15 & 1.3096e-14 & 1.6391e-14$\quad$\\
\hline
\multicolumn{1}{l}{\multirow{3}*{$\quad\{\mathcal{T}_{h}^{c}\}_{N_t^c}$}} &$N_t^c$ & 25 & 100 & 225 & 400 & 625$\quad$\\
& $\|\nabla\cdot\boldsymbol{u}_{h}\|_{0}$ & 1.1745e-16 & 2.6489e-16 & 3.9901e-16 & 4.8691e-16 & 6.2374e-16$\quad$\\
& $\|\nabla\cdot\boldsymbol{J}_{h}\|_{0}$ & 2.9465e-15 & 6.0252e-15 & 9.4660e-15 & 1.2578e-14 & 1.5259e-14$\quad$\\
\hline
 & & $k_1=3$ & &$k_2=2$ & &\\
\hline
\multicolumn{1}{l}{\multirow{3}*{$\quad\{\mathcal{T}_{h}^{a}\}_{N_t^a}$}} &$N_t^a$ & 25 & 100 & 225 & 400 & 625$\quad$\\
& $\|\nabla\cdot\boldsymbol{u}_{h}\|_{0}$ & 8.2948e-10 & 6.8605e-11 & 4.3703e-13 & 1.6281e-12 & 2.6458e-14$\quad$\\
& $\|\nabla\cdot\boldsymbol{J}_{h}\|_{0}$ & 3.8793e-15 & 8.4146e-15 & 1.3391e-14 & 1.7778e-14 & 2.2439e-14$\quad$\\
\hline
\multicolumn{1}{l}{\multirow{3}*{$\quad\{\mathcal{T}_{h}^{b}\}_{N_t^b}$}} &$N_t^b$ & 25 & 100 & 225 & 400 & 625$\quad$\\
& $\|\nabla\cdot\boldsymbol{u}_{h}\|_{0}$ & 6.7811e-16 & 1.1385e-15 & 2.1360e-15 & 2.9119e-15 & 3.7930e-15$\quad$\\
& $\|\nabla\cdot\boldsymbol{J}_{h}\|_{0}$ & 5.2034e-15 & 1.0417e-14 & 1.7125e-14 & 2.2661e-14 & 2.7996e-14$\quad$\\
\hline
\multicolumn{1}{l}{\multirow{3}*{$\quad\{\mathcal{T}_{h}^{c}\}_{N_t^c}$}} &$N_t^c$ & 25 & 100 & 225 & 400 & 625$\quad$\\
& $\|\nabla\cdot\boldsymbol{u}_{h}\|_{0}$ & 3.8742e-16 & 5.8802e-16 & 8.5715e-16 & 1.0343e-15 & 1.3430e-15$\quad$\\
& $\|\nabla\cdot\boldsymbol{J}_{h}\|_{0}$ & 4.8106e-15 & 9.4772e-15 & 1.3886e-14 & 1.7857e-14 & 2.2509e-14$\quad$\\
\hline
\end{tabular*}}
\end{table}

\subsection{Robustness of pressure and electric potential- Revise by using Newton iterative}
In this subsection, the robustness of the pressure and the electric potential are only investigated under the situation $k_{1} = 2$ and $k_{2} = 1$. In other higher-order cases, we can still obtain similar results. Referencing \cite{john2017divergence,zhang2021coupled,zhang2022fully}, we set $\mathrm{\Omega}=[0,1]^{2}$, $\boldsymbol{B}=(0,0,1)$, $\nu=Sc=1$, and choose the force terms $\boldsymbol{f}$ and $\boldsymbol{g}$ such that the exact solutions are given by
\begin{align}
\nonumber
&\boldsymbol{u}=(0,0),\quad p = (x^3+2x^2-\frac{11}{12})s_{1},\\\nonumber
&\boldsymbol{J}=(0,0),\quad \phi = (x^2+x-\frac{5}{6})s_{2},
\end{align}
where $s_{1},s_{2}\leq0$ are parameters.

On the one hand, we set $s_{1}=0$ and vary $s_{2}$ from 1, 100, 10000.
According to the errors of the numerical solutions given in Tables 2-4, we observe that the errors in velocity, pressure, and current density are independent of the mesh at a fixed $s_2$. As varying $s_2$ from 1, 100, 10000, due to an increase in the right-hand side function $\boldsymbol{g}$ such that the errors of numerical solutions will have a little bigger. In addition,
the errors for electric potential converge optimally. The results show that our
scheme is robust with respect to the electric potential.
\begin{table}\footnotesize
\setstretch{1}
\setlength{\belowcaptionskip}{5pt}
\centering
\setlength{\tabcolsep}{0.1mm}{
\caption{Robustness of the electric potential with $s_{2} = 1, 100, 10000$ under three types of meshes. \label{tab:robust-potential-s_2-1}}
\begin{tabular*}{16cm}{@{\extracolsep{\fill}} c c l l l l l }
\hline
 $s_{2} = 1$ & & $k_1=2$ & &$k_2=1$ & &\\
\hline
\multicolumn{1}{l}{\multirow{5}*{$\quad\{\mathcal{T}_{h}^{a}\}_{N_t^a}$}} &$N_t^a$ & 25 & 100 & 225 & 400 & 625$\quad$\\
& $\|\boldsymbol{u}-\boldsymbol{u}_{h}\|_{1}$ & 1.4038e-17 & 2.2981e-17 & 3.0505e-17 & 3.9662e-17 & 4.5267e-17$\quad$\\
& $\|p-p_h\|_{0}$ & 5.5188e-17 & 3.6679e-17 & 4.9080e-17 & 3.7742e-17 & 4.2697e-17$\quad$\\
& $\|\boldsymbol{J}-\boldsymbol{J}_{h}\|_{0}$ & 1.1694e-15 & 2.3260e-15 & 3.5396e-15 & 5.0605e-15 & 6.9883e-15$\quad$\\
& $\|\phi-\phi_h\|_{0}$ & 3.7525e-03(-) & 1.0049e-03(1.84) & 4.1406e-04(2.05) & 2.3643e-04(1.91) & 1.4899e-04(2.09)$\quad$\\
\hline
\multicolumn{1}{l}{\multirow{5}*{$\quad\{\mathcal{T}_{h}^{b}\}_{N_t^b}$}} &$N_t^b$ & 25 & 100 & 225 & 400 & 625$\quad$\\
& $\|\boldsymbol{u}-\boldsymbol{u}_{h}\|_{1}$ & 1.0300e-17 & 1.3279e-17 & 1.3266e-17 & 1.6146e-17 & 1.6368e-17$\quad$\\
& $\|p-p_h\|_{0}$ &  6.8436e-17 & 5.6800e-17  & 3.9894e-17 & 3.6314e-17 & 3.5646e-17$\quad$\\
& $\|\boldsymbol{J}-\boldsymbol{J}_{h}\|_{0}$ & 1.2599e-15 & 2.2999e-15 & 3.5039e-15 & 4.2394e-15 & 5.6181e-15$\quad$\\
& $\|\phi-\phi_h\|_{0}$ & 4.2123e-03(-) & 1.1670e-03(1.97) & 5.2978e-04(1.99) & 3.0027e-04(1.99) & 1.9285e-04(2.00)$\quad$\\
\hline
\multicolumn{1}{l}{\multirow{5}*{$\quad\{\mathcal{T}_{h}^{c}\}_{N_t^c}$}} &$N_t^c$ & 25 & 100 & 225 & 400 & 625$\quad$\\
& $\|\boldsymbol{u}-\boldsymbol{u}_{h}\|_{1}$ & 2.1006e-17 & 2.9916e-17 & 3.7833e-17 & 4.6568e-17 & 5.0823e-17$\quad$\\
& $\|p-p_h\|_{0}$ & 7.6793e-17 & 6.3892e-17 & 4.5267e-17 & 4.9397e-17 & 4.9483e-17$\quad$\\
& $\|\boldsymbol{J}-\boldsymbol{J}_{h}\|_{0}$ & 1.3120e-15 & 3.0357e-15 & 4.4069e-15 & 6.1490e-15 & 7.2545e-15$\quad$\\
& $\|\phi-\phi_h\|_{0}$ & 3.3705e-03(-) & 8.3988e-04(2.00) & 3.7286e-04(2.00) & 2.0961e-04(2.00) & 1.3411e-04(2.00)$\quad$\\
\hline
 $s_{2} = 100$ & & $k_1=2$ & &$k_2=1$ & &\\
\hline
\multicolumn{1}{l}{\multirow{5}*{$\quad\{\mathcal{T}_{h}^{a}\}_{N_t^a}$}} &$N_t^a$ & 25 & 100 & 225 & 400 & 625$\quad$\\
& $\|\boldsymbol{u}-\boldsymbol{u}_{h}\|_{1}$ & 1.2840e-15 & 2.4069e-15 & 3.1347e-15 & 4.0484e-15 & 5.0694e-15$\quad$\\
& $\|p-p_h\|_{0}$ & 7.1870e-15 & 3.9017e-15 & 4.4053e-15 & 3.8523e-15 & 3.9589e-15$\quad$\\
& $\|\boldsymbol{J}-\boldsymbol{J}_{h}\|_{0}$ & 1.1714e-13 & 2.3217e-13 & 3.6027e-13 & 5.0793e-13 & 7.3304e-13$\quad$\\
& $\|\phi-\phi_h\|_{0}$ & 3.7525e-01(-) & 1.0049e-01(1.85) & 4.1406e-02(2.05) & 2.3643e-02(1.91) & 1.4899e-02(2.09)$\quad$\\
\hline
\multicolumn{1}{l}{\multirow{5}*{$\quad\{\mathcal{T}_{h}^{b}\}_{N_t^b}$}} &$N_t^b$ & 25 & 100 & 225 & 400 & 625$\quad$\\
& $\|\boldsymbol{u}-\boldsymbol{u}_{h}\|_{1}$ & 9.4523e-16 & 1.4999e-15 & 1.2935e-15 & 1.6928e-15 & 1.6773e-15$\quad$\\
& $\|p-p_h\|_{0}$ & 7.0120e-15 & 4.1375e-15 & 3.7761e-15 & 3.5601e-15 & 4.2104e-15$\quad$\\
& $\|\boldsymbol{J}-\boldsymbol{J}_{h}\|_{0}$ & 1.2035e-13 & 2.3500e-13 & 3.4596e-13 & 4.1909e-13 & 5.4691e-13$\quad$\\
& $\|\phi-\phi_h\|_{0}$ & 4.2123e-01(-) & 1.1670e-01(1.97) & 5.2978e-02(1.99) & 3.0027e-02(1.99) & 1.9285e-02(2.00) $\quad$\\
\hline
\multicolumn{1}{l}{\multirow{5}*{$\quad\{\mathcal{T}_{h}^{c}\}_{N_t^c}$}} &$N_t^c$ & 25 & 100 & 225 & 400 & 625$\quad$\\
& $\|\boldsymbol{u}-\boldsymbol{u}_{h}\|_{1}$ & 2.4746e-15 & 2.6882e-15 & 3.7472e-15 & 4.4290e-15 & 6.0902e-15$\quad$\\
& $\|p-p_h\|_{0}$ & 7.6446e-15 & 6.3170e-15 & 4.3570e-15 & 5.0499e-15 & 4.4499e-15$\quad$\\
& $\|\boldsymbol{J}-\boldsymbol{J}_{h}\|_{0}$ & 1.5589e-13 & 3.0238e-13 & 4.4964e-13 & 6.2002e-13 & 7.4859e-13$\quad$\\
& $\|\phi-\phi_h\|_{0}$ & 3.3705e-01(-) & 8.3988e-02(2.00) & 3.7286e-02(2.00) & 2.0961e-02(2.00) & 1.3411e-02(2.00)$\quad$\\
\hline
 $s_{2} = 10000$ & & $k_1=2$ & &$k_2=1$ & &\\
\hline
\multicolumn{1}{l}{\multirow{5}*{$\quad\{\mathcal{T}_{h}^{a}\}_{N_t^a}$}} &$N_t^a$ & 25 & 100 & 225 & 400 & 625$\quad$\\
& $\|\boldsymbol{u}-\boldsymbol{u}_{h}\|_{1}$ & 1.4234e-13 & 2.0567e-13 & 3.0912e-13 & 3.9268e-13 & 4.6491e-13$\quad$\\
& $\|p-p_h\|_{0}$ & 6.8886e-13 & 4.3126e-13 & 4.1756e-13 & 3.6822e-13 & 3.8735e-13$\quad$\\
& $\|\boldsymbol{J}-\boldsymbol{J}_{h}\|_{0}$ & 1.1725e-11 & 2.3065e-11 & 3.5231e-11 & 5.0388e-11 & 7.0641e-11$\quad$\\
& $\|\phi-\phi_h\|_{0}$ & 3.7525e+01(-) & 1.0049e+01(1.85) & 4.1406e+00(2.05) & 2.3643e+00(1.91) & 1.4899e+00(2.09)$\quad$\\
\hline
\multicolumn{1}{l}{\multirow{5}*{$\quad\{\mathcal{T}_{h}^{b}\}_{N_t^b}$}} &$N_t^b$ & 25 & 100 & 225 & 400 & 625$\quad$\\
& $\|\boldsymbol{u}-\boldsymbol{u}_{h}\|_{1}$ & 8.6213e-14 & 1.2054e-13 & 1.3848e-13 & 1.4475e-13 & 1.8435e-13$\quad$\\
& $\|p-p_h\|_{0}$ & 6.0537e-13 & 4.5795e-13 & 3.9193e-13 & 3.3546e-13 & 4.1438e-13$\quad$\\
& $\|\boldsymbol{J}-\boldsymbol{J}_{h}\|_{0}$ & 1.1377e-11 & 2.2713e-11 & 3.5203e-11 & 4.0336e-11 & 5.7377e-11$\quad$\\
& $\|\phi-\phi_h\|_{0}$ & 4.2123e+01(-) & 1.1670e+01(1.97) & 5.2978e+00(1.99) & 3.0027e+00(1.99) & 1.9285e+00(2.00)$\quad$\\
\hline
\multicolumn{1}{l}{\multirow{5}*{$\quad\{\mathcal{T}_{h}^{c}\}_{N_t^c}$}} &$N_t^c$ & 25 & 100 & 225 & 400 & 625$\quad$\\
& $\|\boldsymbol{u}-\boldsymbol{u}_{h}\|_{1}$ & 2.5246e-13 & 2.5583e-13 & 3.6619e-13 & 4.9866e-13 & 5.6965e-13$\quad$\\
& $\|p-p_h\|_{0}$ & 8.8202e-13 & 6.4251e-13 & 5.6004e-13 & 5.3525e-13 & 5.3539e-13$\quad$\\
& $\|\boldsymbol{J}-\boldsymbol{J}_{h}\|_{0}$ & 1.5581e-11 & 2.9220e-11 & 4.3027e-11 & 6.1705e-11 & 7.1356e-11$\quad$\\
& $\|\phi-\phi_h\|_{0}$ & 3.3705e+01(-) & 8.3988e+00(2.01) & 3.7286e+00(2.00) & 2.0961e+00(2.00) & 1.3411e+00(2.00)$\quad$\\
\hline
\end{tabular*}}
\end{table}

On the other hand, we set $s_{2}=0$ and vary $s_{1}$ from 1, 100, 10000.
According to the errors of the numerical solutions given in Tables 5-7, we observe that the errors in velocity, pressure, and current density are independent of the mesh at a fixed $s_1$. As varying $s_1$ from 1, 100, 10000, due to an increase in the right-hand side function $\boldsymbol{f}$ such that the errors of numerical solutions will have a little bigger. In addition,
the errors for electric potential converge optimally. The results show that our
scheme is robust with respect to the electric potential.
\begin{table}\footnotesize
\setstretch{1}
\setlength{\belowcaptionskip}{5pt}
\centering
\setlength{\tabcolsep}{0.1mm}{
\caption{Robustness of the pressure with $s_{1} = 1,\;100,\;10000$ under three types of meshes. \label{tab:robust-potential-s_1-1}}
\begin{tabular*}{16cm}{@{\extracolsep{\fill}} c c l l l l l }
\hline
 $s_{1} = 1$ & & $k_1=2$ & &$k_2=1$ & &\\
\hline
\multicolumn{1}{l}{\multirow{5}*{$\quad\{\mathcal{T}_{h}^{a}\}_{N_t^a}$}} &$N_t^a$ & 25 & 100 & 225 & 400 & 625$\quad$\\
& $\|\boldsymbol{u}-\boldsymbol{u}_{h}\|_{1}$ & 1.1586e-16 & 8.9756e-17 & 8.6007e-17 & 8.8406e-17 & 8.3868e-17$\quad$\\
& $\|p-p_h\|_{0}$ & 1.3033e-02(-) & 3.5604e-03(1.82) & 1.4612e-03(2.06) & 8.6327e-04(1.79) & 5.5409e-04(2.01)$\quad$\\
& $\|\boldsymbol{J}-\boldsymbol{J}_{h}\|_{0}$ & 2.2776e-18 & 1.0237e-18 & 8.4768e-19 & 6.8322e-19 & 5.9517e-19$\quad$\\
& $\|\phi-\phi_h\|_{0}$ & 7.0098e-19 & 6.3796e-19 & 3.2599e-19 & 2.8614e-19 & 3.1854e-19$\quad$\\
\hline
\multicolumn{1}{l}{\multirow{5}*{$\quad\{\mathcal{T}_{h}^{b}\}_{N_t^b}$}} &$N_t^b$ & 25 & 100 & 225 & 400 & 625$\quad$\\
& $\|\boldsymbol{u}-\boldsymbol{u}_{h}\|_{1}$ & 7.9077e-17 & 8.1391e-17 & 9.2824e-17 & 8.1362e-17 & 8.0287e-17$\quad$\\
& $\|p-p_h\|_{0}$ & 1.5147e-02(-) & 4.2029e-03(1.97) & 1.9084e-03(1.99) & 1.0818e-03(1.99) & 6.9481e-04(2.00)$\quad$\\
& $\|\boldsymbol{J}-\boldsymbol{J}_{h}\|_{0}$ & 9.9887e-19 & 8.6839e-19 & 8.7459e-19 & 6.3256e-19 & 5.3379e-19$\quad$\\
& $\|\phi-\phi_h\|_{0}$ & 6.0627e-19 & 3.4229e-19 & 2.7354e-19 & 2.3673e-19 & 2.1563e-19$\quad$\\
\hline
\multicolumn{1}{l}{\multirow{5}*{$\quad\{\mathcal{T}_{h}^{c}\}_{N_t^c}$}} &$N_t^c$ & 25 & 100 & 225 & 400 & 625$\quad$\\
& $\|\boldsymbol{u}-\boldsymbol{u}_{h}\|_{1}$ & 1.2267e-16 & 1.0013e-16 & 9.3688e-17 & 9.9873e-17 & 7.9725e-17$\quad$\\
& $\|p-p_h\|_{0}$ & 1.1462e-02(-) & 2.9339e-03(1.97) & 1.3157e-03(1.98) & 7.4352e-04(1.98) & 4.7721e-04(1.99) $\quad$\\
& $\|\boldsymbol{J}-\boldsymbol{J}_{h}\|_{0}$ & 4.1419e-18 & 1.8081e-18 & 1.2560e-18 & 1.0301e-18 & 6.9548e-19$\quad$\\
& $\|\phi-\phi_h\|_{0}$ & 7.7735e-19 & 3.4985e-19 & 3.5279e-19 & 2.8518e-19 & 1.7725e-19$\quad$\\
\hline
 $s_{1} = 100$ & & $k_1=2$ & &$k_2=1$ & &\\
\hline
\multicolumn{1}{l}{\multirow{5}*{$\quad\{\mathcal{T}_{h}^{a}\}_{N_t^a}$}} &$N_t^a$ & 25 & 100 & 225 & 400 & 625$\quad$\\
& $\|\boldsymbol{u}-\boldsymbol{u}_{h}\|_{1}$ & 9.5984e-15 & 7.7415e-15 & 9.0305e-15 & 8.7799e-15 & 8.1488e-15$\quad$\\
& $\|p-p_h\|_{0}$ & 1.3033e+00(-) & 3.5604e-01(1.82) & 1.4612e-01(2.06) & 8.6327e-02(1.79) & 5.5409e-02(2.01)$\quad$\\
& $\|\boldsymbol{J}-\boldsymbol{J}_{h}\|_{0}$ & 1.8599e-16 & 9.5700e-17 & 8.8457e-17 & 6.8496e-17 & 5.9400e-17$\quad$\\
& $\|\phi-\phi_h\|_{0}$ & 4.2950e-17 & 2.4435e-17 & 4.1856e-17 & 3.2372e-17 & 1.7283e-17$\quad$\\
\hline
\multicolumn{1}{l}{\multirow{5}*{$\quad\{\mathcal{T}_{h}^{b}\}_{N_t^b}$}} &$N_t^b$ & 25 & 100 & 225 & 400 & 625$\quad$\\
& $\|\boldsymbol{u}-\boldsymbol{u}_{h}\|_{1}$ & 7.8585e-15 & 8.9846e-15 & 8.9851e-15 & 8.3134e-15 & 7.8817e-15$\quad$\\
& $\|p-p_h\|_{0}$ & 1.5147e+00(-) & 4.2029e-01(1.97) & 1.9084e-01(1.99) & 1.0818e-01(1.99) & 6.9481e-02(2.00)$\quad$\\
& $\|\boldsymbol{J}-\boldsymbol{J}_{h}\|_{0}$ & 9.6668e-17 & 9.9319e-17 & 8.6918e-17 & 6.1416e-17 & 5.2059e-17$\quad$\\
& $\|\phi-\phi_h\|_{0}$ & 2.9379e-17 & 2.1606e-17 & 2.1082e-17 & 2.8409e-17 & 1.6956e-17$\quad$\\
\hline
\multicolumn{1}{l}{\multirow{5}*{$\quad\{\mathcal{T}_{h}^{c}\}_{N_t^c}$}} &$N_t^c$ & 25 & 100 & 225 & 400 & 625$\quad$\\
& $\|\boldsymbol{u}-\boldsymbol{u}_{h}\|_{1}$ & 1.1457e-14 & 1.0516e-14 & 9.3583e-15 & 9.6834e-15 & 7.8838e-15$\quad$\\
& $\|p-p_h\|_{0}$ & 1.1462e+00(-) & 2.9339e-01(1.97) & 1.3157e-01(1.98) & 7.4352e-02(1.98) & 4.7721e-02(1.99)$\quad$\\
& $\|\boldsymbol{J}-\boldsymbol{J}_{h}\|_{0}$ & 3.6189e-16 & 1.9780e-16 & 1.2236e-16 & 1.0353e-16 & 6.8903e-17$\quad$\\
& $\|\phi-\phi_h\|_{0}$ & 5.3327e-17 & 4.0973e-17 & 2.1874e-17 & 1.6691e-17 & 2.6032e-17$\quad$\\
\hline
 $s_{1} = 10000$ & & $k_1=2$ & &$k_2=1$ & &\\
\hline
\multicolumn{1}{l}{\multirow{5}*{$\quad\{\mathcal{T}_{h}^{a}\}_{N_t^a}$}} &$N_t^a$ & 25 & 100 & 225 & 400 & 625$\quad$\\
& $\|\boldsymbol{u}-\boldsymbol{u}_{h}\|_{1}$ & 1.0344e-12 & 8.6090e-13 & 8.4833e-13 & 8.2517e-13 & 8.0170e-13$\quad$\\
& $\|p-p_h\|_{0}$ & 1.3033e+02(-) & 3.5604e+01(1.82) & 1.4612e+01(2.06) & 8.6327e+00(1.80) & 5.5409e+00(2.01)$\quad$\\
& $\|\boldsymbol{J}-\boldsymbol{J}_{h}\|_{0}$ & 2.2218e-14 & 1.0016e-14 & 8.4395e-15 & 6.7571e-15 & 5.8470e-15$\quad$\\
& $\|\phi-\phi_h\|_{0}$ & 3.1215e-15 & 2.2643e-15 & 2.2568e-15 & 2.3685e-15 & 1.0430e-15$\quad$\\
\hline
\multicolumn{1}{l}{\multirow{5}*{$\quad\{\mathcal{T}_{h}^{b}\}_{N_t^b}$}} &$N_t^b$ & 25 & 100 & 225 & 400 & 625$\quad$\\
& $\|\boldsymbol{u}-\boldsymbol{u}_{h}\|_{1}$ & 1.0834e-12 & 9.0888e-13 & 9.3482e-13  & 8.2879e-13 & 7.8379e-13$\quad$\\
& $\|p-p_h\|_{0}$ & 1.5147e+02(-) & 4.2029e+01(1.97) & 1.9084e+01(1.99) & 1.0818e+01(1.99) & 6.9481e+00(2.00)$\quad$\\
& $\|\boldsymbol{J}-\boldsymbol{J}_{h}\|_{0}$ & 1.4619e-14 & 9.9619e-15 & 8.3761e-15 & 6.1658e-15 & 5.3620e-15$\quad$\\
& $\|\phi-\phi_h\|_{0}$ & 1.0911e-14 & 3.4001e-15 & 3.8574e-15 & 3.0887e-15 & 1.4449e-15$\quad$\\
\hline
\multicolumn{1}{l}{\multirow{5}*{$\quad\{\mathcal{T}_{h}^{c}\}_{N_t^c}$}} &$N_t^c$ & 25 & 100 & 225 & 400 & 625$\quad$\\
& $\|\boldsymbol{u}-\boldsymbol{u}_{h}\|_{1}$ & 1.2496e-12 & 1.0182e-12 & 9.2227e-13 & 9.8056e-13 & 7.6705e-13$\quad$\\
& $\|p-p_h\|_{0}$ & 1.1462e+02(-) & 2.9339e+01(1.97) & 1.3157e+01(1.98) & 7.4352e+00(1.98) & 4.7721e+00(1.99)$\quad$\\
& $\|\boldsymbol{J}-\boldsymbol{J}_{h}\|_{0}$ & 3.6673e-14 & 1.9954e-14 & 1.2263e-14 & 1.0301e-14 & 6.8154e-15$\quad$\\
& $\|\phi-\phi_h\|_{0}$ & 8.2712e-15 & 4.4074e-15 & 2.5894e-15 & 1.5982e-15 & 1.7892e-15$\quad$\\
\hline
\end{tabular*}}
\end{table}

Overall, the robustness of the pressure and the electric potential are satisfied under our scheme and the results are consistent with the predictions in Theorem \ref{thm:error}.

\subsection{Inductionless MHD problem with singular solution\label{subsec:Singular test}}
In this subsection, the inductionless MHD problem with singular solution is considered in a non-convex L-shaped domain $\mathrm{\Omega}=[-0.5,0.5]\times[-0.5,0.5]\setminus[0,0.5]\times(-0.5,0]$ to
verify the ability of the proposed method to capture the singularities. Let $\boldsymbol{B}=(0,0,1)$, $Re = Sc = 1$, and the functions $\boldsymbol{f}$ and $\boldsymbol{g}$ such that
exact solution is
\begin{align*}
&u_1(r,\vartheta) = r^{\varepsilon}\big((1+\varepsilon)\sin(\vartheta)\Theta(\vartheta)+
\cos(\vartheta)\Theta^{'}(\vartheta)\big),\quad u_2 = r^{\varepsilon}\big(-(1+\varepsilon)\cos(\vartheta)\Theta(\vartheta)+
\sin(\vartheta)\Theta^{'}(\vartheta)\big),\\
&p(r,\vartheta) = -r^{\varepsilon-1}\big((1+\varepsilon)^{2}\Theta^{'}(\vartheta)+\Theta^{'''}(\vartheta)\big)/(1-\varepsilon),\quad
\phi(r,\vartheta) = 0,\\
&J_1(r,\vartheta) = -\big(2\sin(\vartheta/3)\big)/(3r^{1/3}),\quad J_2(r,\vartheta) = \big(2\cos(\vartheta/3)\big)/(3r^{1/3}),
\end{align*}
where $(r,\vartheta)$ denotes the polar coordinate, $\mu = \frac{3\pi}{2}$, the value of the parameter $\varepsilon$ is the smallest positive solution of $\sin(\varepsilon\mu)+\varepsilon\sin(\mu)=0$ and
\begin{equation}
\nonumber
\Theta(\vartheta) = \sin\big((1+\varepsilon)\vartheta\big)\cos(\varepsilon\mu)/(1+\varepsilon)-
\cos\big((1+\varepsilon)-\sin\big((1-\varepsilon)\vartheta\big)\cos(\varepsilon\mu)/(1-\varepsilon)
+\cos\big((1-\varepsilon).
\end{equation}
Significantly, we get $(\boldsymbol{u},p)\in\boldsymbol{H}^{1+\varepsilon}\times H^{\varepsilon}$ and $\boldsymbol{J}\in \boldsymbol{H}^{2/3}$. the conductive boundary condition is considered in this example, it means that the electric potential $\phi=0$ on $\partial\mathrm{\Omega}$. By observing, we can clearly determine that the exact solution have strong singularity at the corner.

 Due to the low regularity of the exact solutions, the pursuit of using higher-order schemes to solve this problem may slightly reduce the error but will not change the convergence law. Therefore, we only provide the calculation results under situation $k_1 = 2$ and $k_2 =1$ here. From the errors and convergence rates of the numerical solutions under the L-domain triangular meshes $\{\mathcal{T}_{h}^{L}\}_{N_{t}^{L}}$ given in Table \ref{tab:L-k2k1}, we can conclude that the convergence rates are consistent with the results of Theorem \ref{thm:error} in theoretical analysis. Moreover, the discrete velocity and current density also satisfy the conditions of divergence-fee. Finally, we display the numerical solution on the grid $\{\mathcal{T}_{h}^{L}\}_{384}$ in figure \ref{fig.7}, which clearly shows that the proposed method can effectively capture singularities.

\begin{table}\footnotesize
\setstretch{1}
\setlength{\belowcaptionskip}{5pt}
\centering
\setlength{\tabcolsep}{0.1mm}{
\caption{The errors and convergence rates of numerical solutions under triangular meshes---Revise-data $\{\mathcal{T}_{h}^{L}\}_{N_{t}^{L}}$.\label{tab:L-k2k1}}
\begin{tabular*}{16cm}{@{\extracolsep{\fill}} r l l l l l l }
\hline
\quad$N_t^L$ & $\|\boldsymbol{u}-\boldsymbol{u}_{h}\|_{1}$ & $\|p-p_h\|_{0}$   & $\|\boldsymbol{J}-\boldsymbol{J}_{h}\|_{0}$ & $\|\phi-\phi_h\|_{0}$ & $\|\nabla\cdot\boldsymbol{u}_{h}\|_{0}$ & $\|\nabla\cdot\boldsymbol{J}_{h}\|_{0}\quad$\\
96   &  5.1869e-01(-)    & 6.7796e-01(-)    & 4.3097e-02(-)    & 9.3642e-04(-) & 3.2502e-11 & 8.9327e-15 \quad\\
384  &  3.5333e-01(0.554) & 4.3717e-01(0.633) & 2.7374e-02(0.655) & 3.8768e-04(1.272) & 1.1154e-11 & 1.7795e-14 \quad\\
1536 &  2.4159e-01(0.548) & 2.9173e-01(0.584) & 1.7308e-02(0.661) & 1.7780e-04(1.124) & 3.8246e-12 & 3.5711e-14 \quad\\
1944 &  2.2651e-01(0.547) & 2.7281e-01(0.569) & 1.6007e-02(0.663) & 1.5661e-04(1.077) & 3.1884e-12 & 4.0106e-14 \quad\\
2345 &  1.9356e-01(0.546) & 2.3191e-01(0.565) & 1.3223e-02(0.664) & 1.1532e-04(1.063) & 2.0447e-12 & 5.3722e-14 \quad\\
\hline
\end{tabular*}}
\end{table}
\begin{figure}
\centering%图片居中
\subfloat[$\{\mathcal{T}_{h}^{L}\}_{384}$]{
\label{fig.7a}
\includegraphics[width=6.65cm]{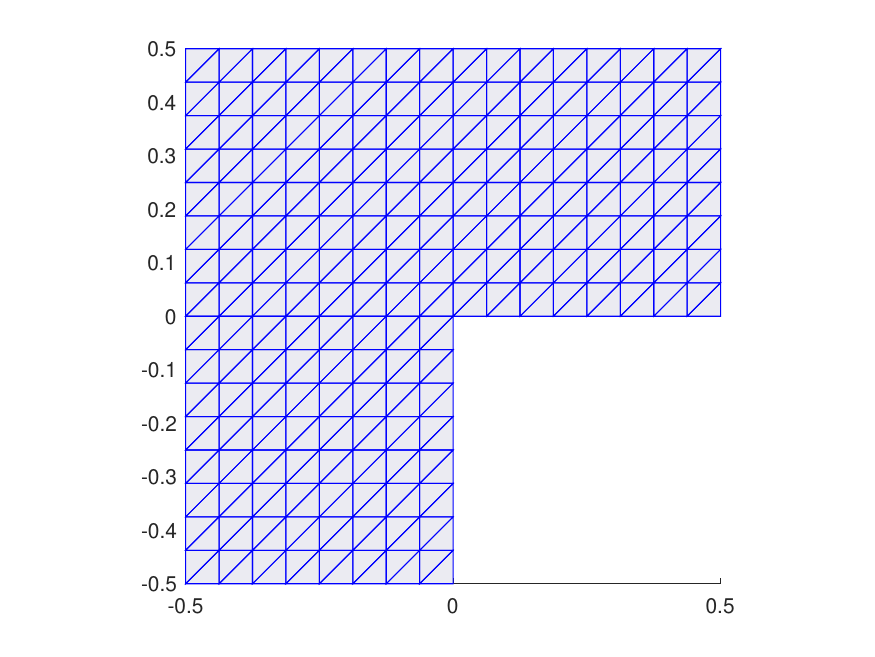}
}
\hspace{-1.41cm}
\subfloat[Numerical approximation of $\boldsymbol{u}$]{
\label{fig.7b}
\includegraphics[width=5.35cm]{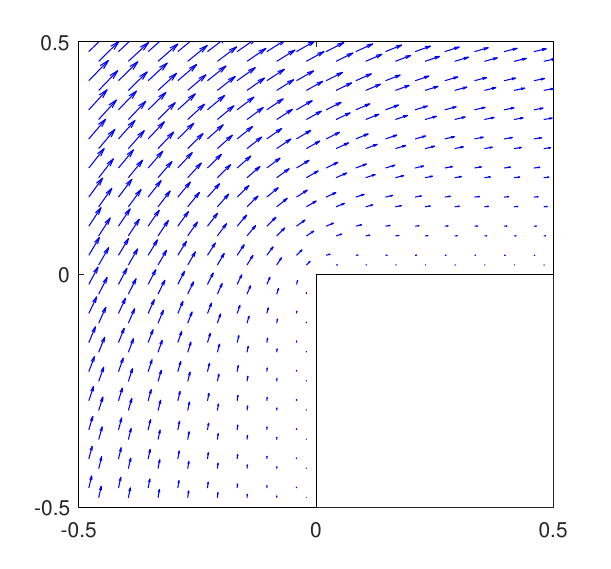}}
\hspace{-0.61cm}
\subfloat[Numerical approximation of $\boldsymbol{J}$]{
\label{fig.7c}
\includegraphics[width=5.35cm]{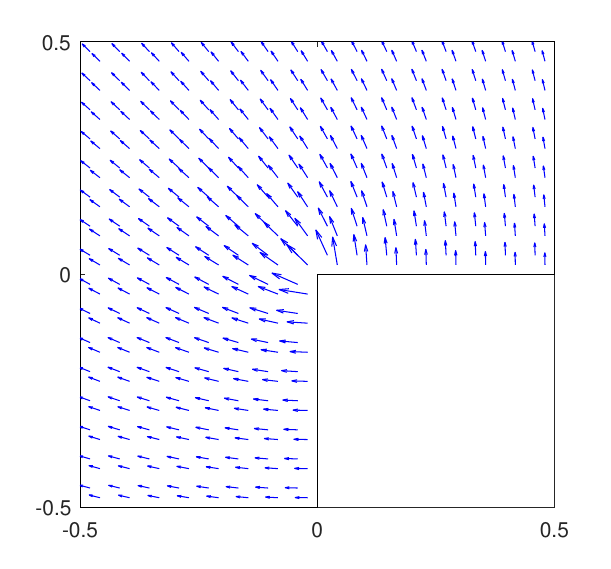}}
\caption{A set of fine polygonal meshes ($N_{t}^{i} = 625$) used to the first numerical experiment.}
\label{fig.7}%与\ref连用，在文章与表格之间建立链接
\end{figure}

\subsection{The inductionless MHD flow around a circular cylinder}
In this subsection, we consider a type of 2D benchmark problem, which is the flow around a circular cylinder \cite{liu2022pressure}.  The main feature of this model is that with the change of Reynolds number from low to high, the flow of the fluid will gradually evolve into a flow that loses symmetry, rather than maintaining non-separation from the cylinder, without vortices, and the upstream and downstream flow lines of the cylinder are symmetrical. The physical structure, boundary conditions and mesh used for this numerical simulation see Fig.\ref{fig.8}.
\begin{figure}
\centering%图片居中
\subfloat[Physical structure]{
\label{fig.8a}
\includegraphics[width=7.6cm]{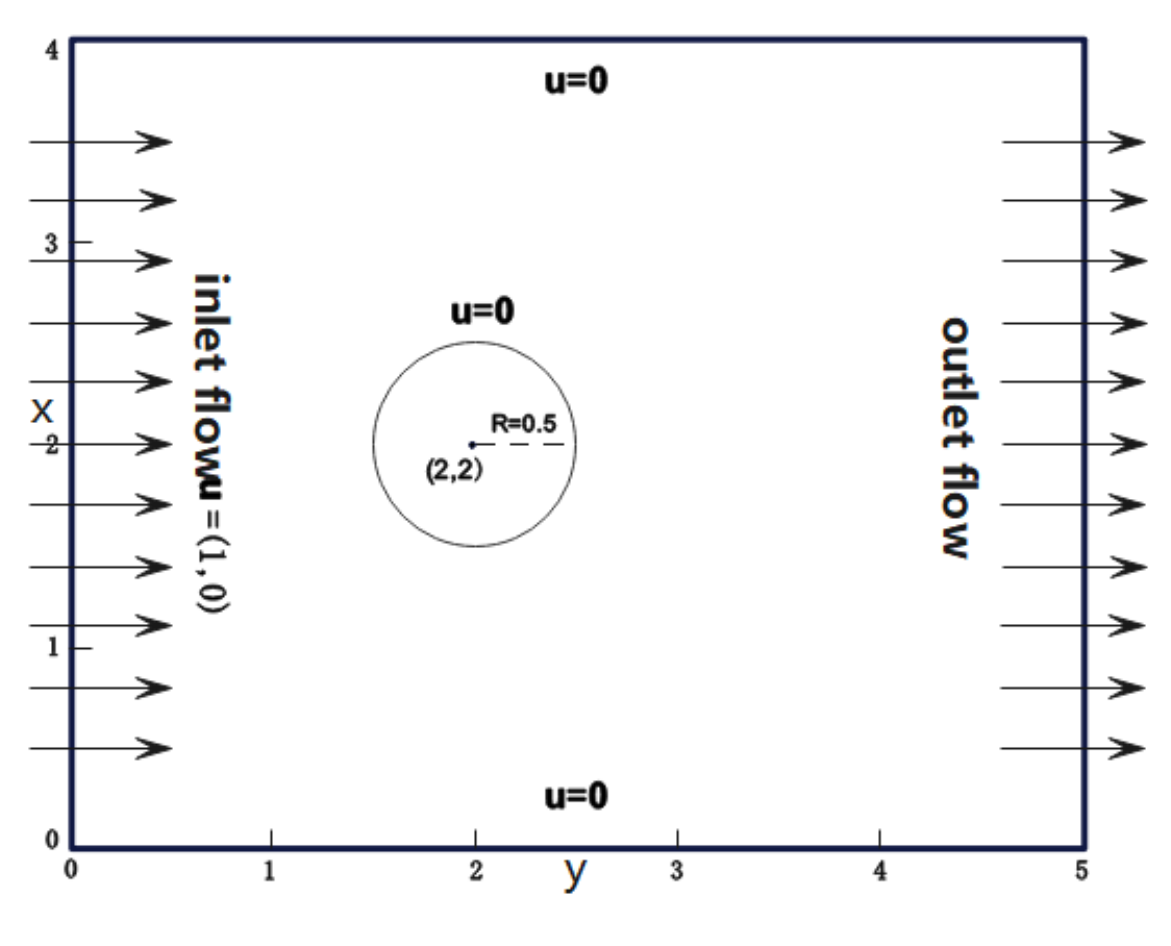}
}
%\hspace{-0.02cm}
\subfloat[$\{\mathcal{T}_{h}^{c}\}_{3000}$]{
\label{fig.8b}
\includegraphics[width=8.6cm,height=6.4cm]{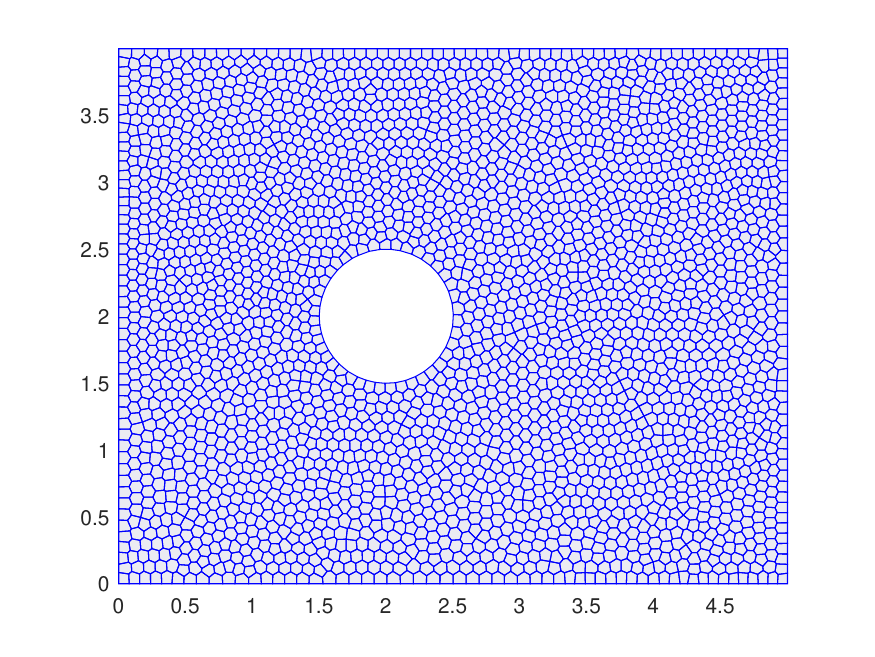}}
\caption{Physical structure and mesh used for the numerical simulation.}
\label{fig.8}%与\ref连用，在文章与表格之间建立链接
\end{figure}

Obviously, we can clearly observing that the streamlines of cylinder are symmetric about both $x=2$ and $y=2$ at $\nu = 1$ and the fluid flowing near the top of the cylinder begins to separate and forms a fixed pair of symmetrical eddies downstream of the cylinder at $\nu = 0.01$.  This is consistent with the theoretical analysis of physical phenomena.

\begin{figure}
\centering%图片居中
\subfloat[Counter]{
\label{fig.9a}
\includegraphics[width=7.6cm]{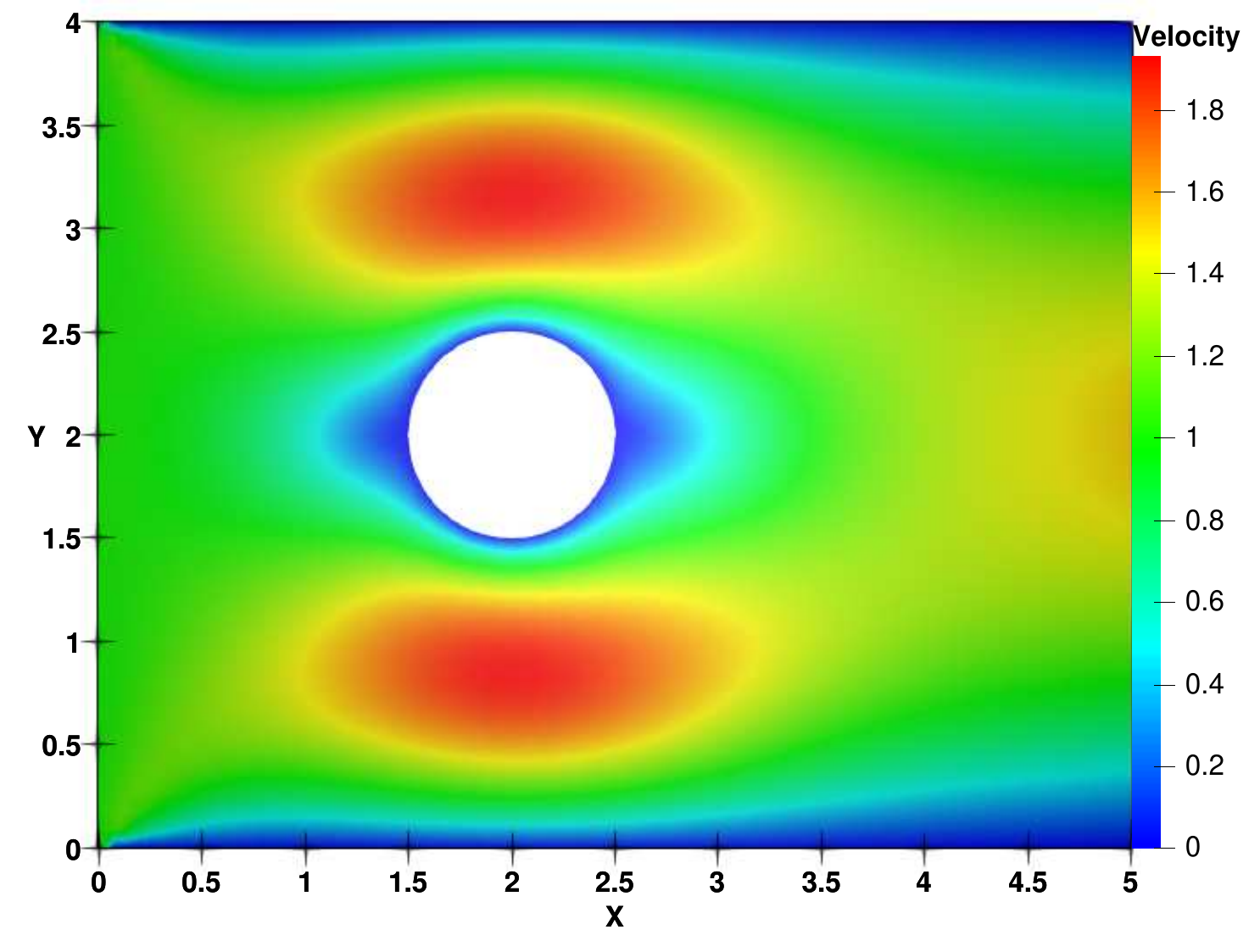}
}
\hspace{0.2cm}
\subfloat[Streamlines]{
\label{fig.9b}
\includegraphics[width=7.6cm]{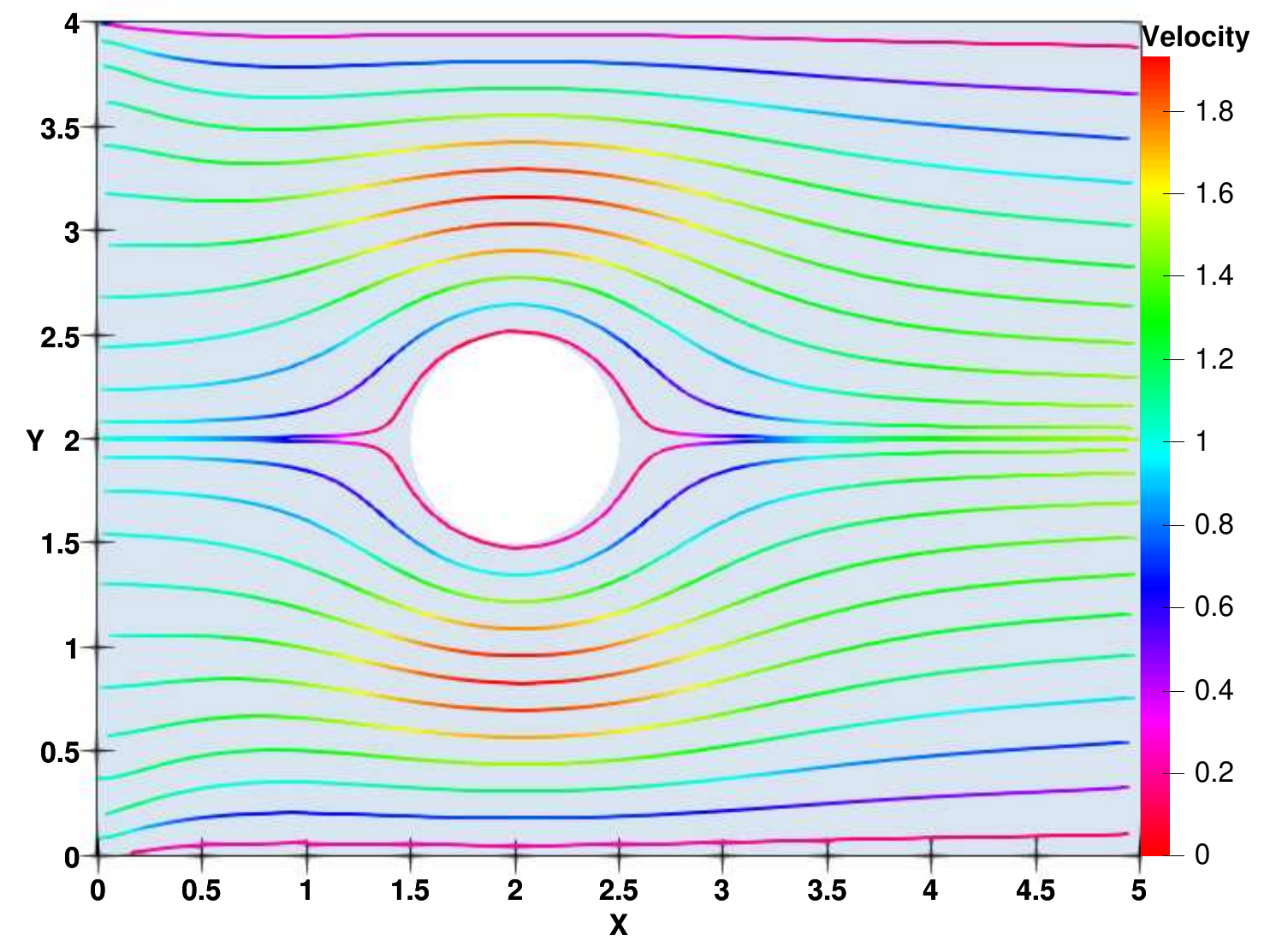}
}
\caption{Counter and streamlines of discrete velocity $\boldsymbol{u}_{h}$ at $\nu = 1$.}
\label{fig.9}%与\ref连用，在文章与表格之间建立链接
\end{figure}
\begin{figure}
\centering%图片居中
\subfloat[Counter]{
\label{fig.10a}
\includegraphics[width=7.6cm]{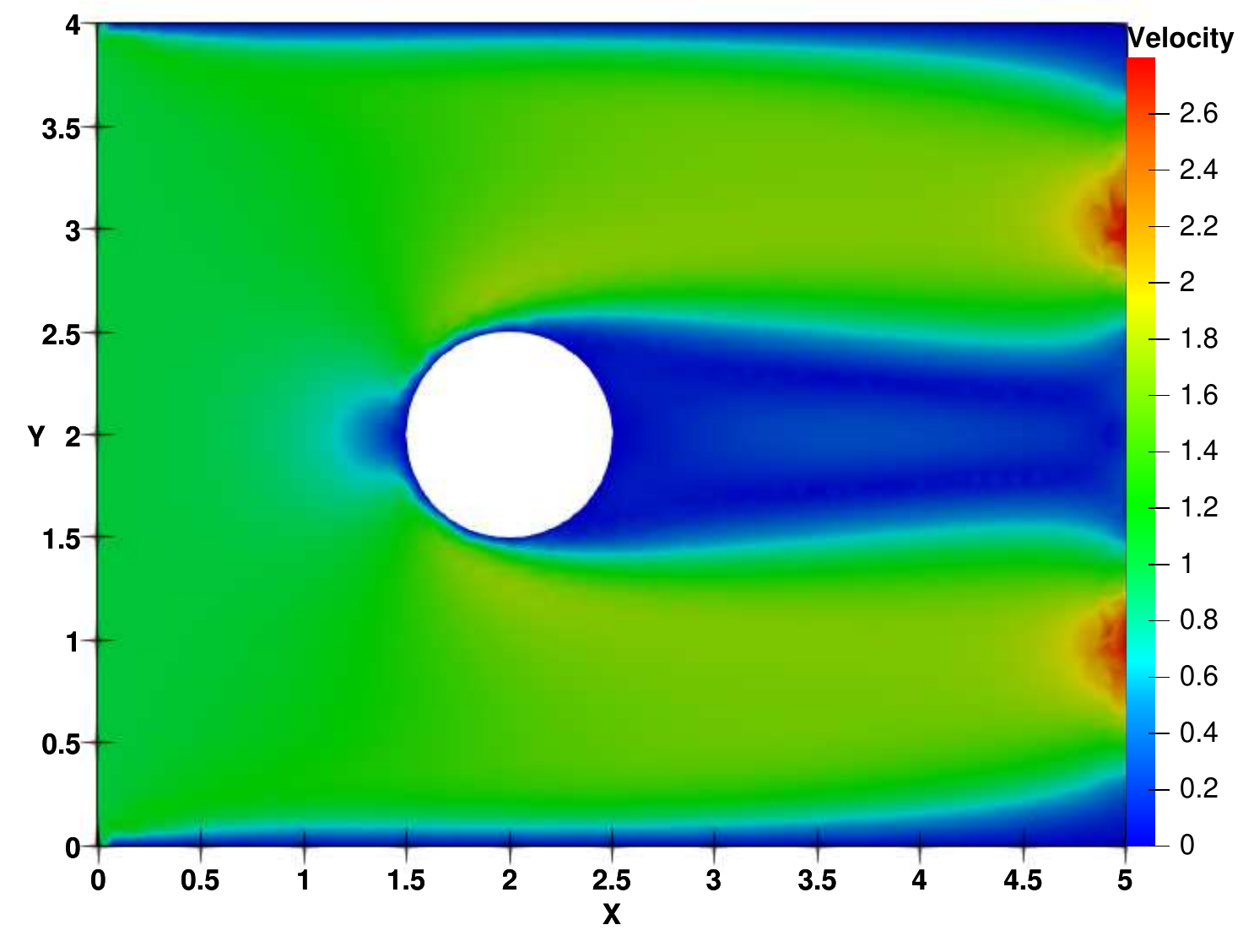}
}
\hspace{0.25cm}
\subfloat[Streamlines]{
\label{fig.10b}
\includegraphics[width=7.6cm]{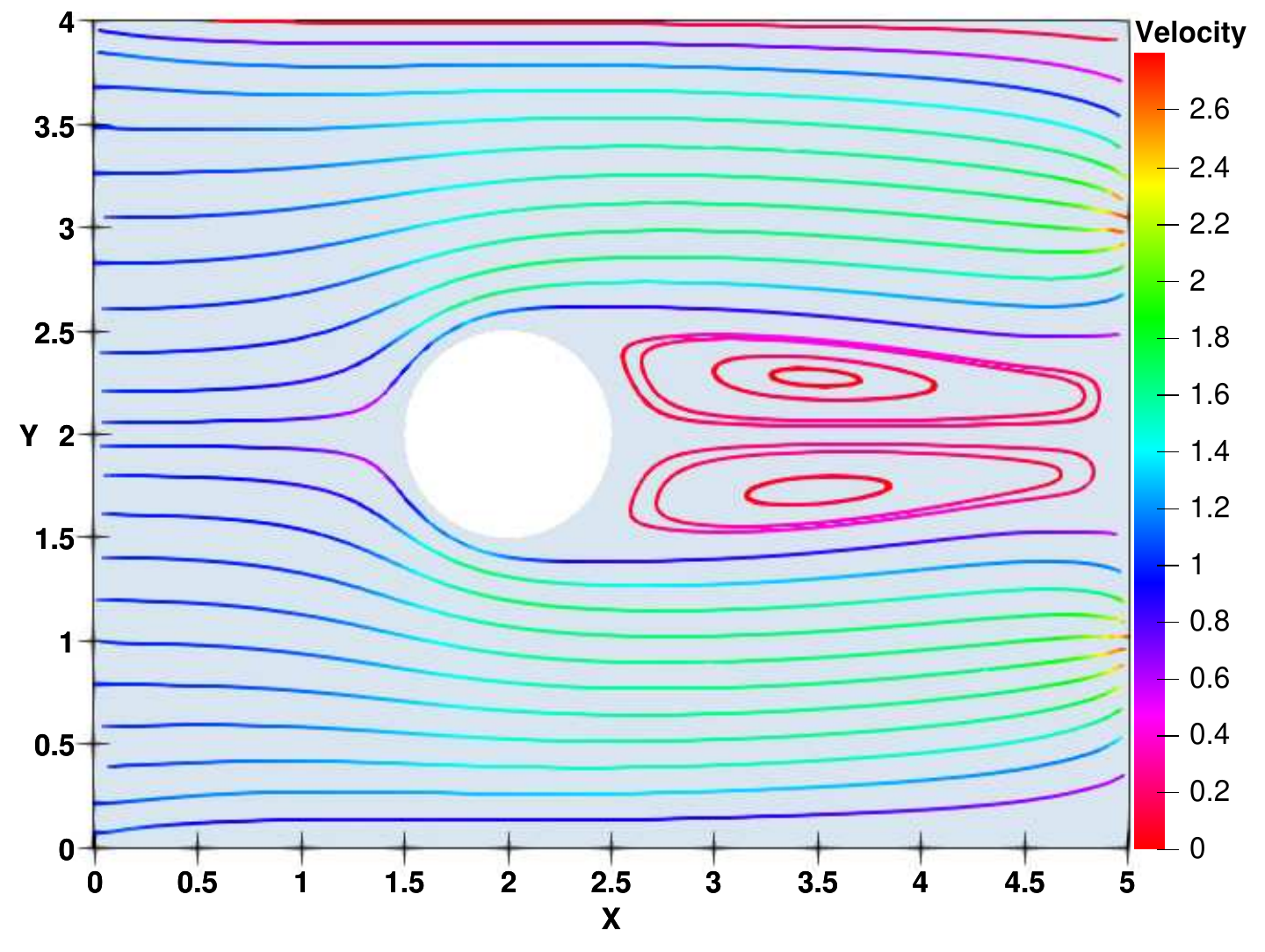}
}
\caption{Counter and streamlines of discrete velocity $\boldsymbol{u}_{h}$ at $\nu = 0.01$.}
\label{fig.11}%与\ref连用，在文章与表格之间建立链接
\end{figure}

\subsection{The backstage step inductionless MHD flow}
\section{Conclusion\label{sec:conclu}}
In this paper, we present a full divergence-free of high order virtual finite element method to approximation of stationary inductionless magnetohydrodynamic
equations on polygonal meshes and process rigorous error analysis to show that the proposed method is stable and convergent. In the following work, we will propose the lowest order 2D/3D virtual finite element method to solve the steady inductionless MHD equations. Of course, this format still satisfies full divergence-free.
%\section{Acknowledgements\label{sec:Acknow}}

%The authors would like to thank the referees and editor for their valuable comments and suggestions which helped to improve the results of this
%paper. Particularly, many thanks to Xiaodi Zhang for his advice and help.

\bibliographystyle{plain}
\bibliography{ref}

\end{document}